\documentclass[12pt,letter]{amsart}
\usepackage{amscd}
\usepackage{amssymb}
\usepackage[centering,text={15.5cm,23cm}]{geometry} 

\usepackage{amsfonts,amssymb, amscd, latexsym, graphicx, psfrag, color,float}

\usepackage{caption}
\usepackage{comment}
\usepackage{wrapfig}
\usepackage[all]{xy}
\usepackage{mathrsfs}
\usepackage{mathtools}
\usepackage{marvosym}
\usepackage{stmaryrd}
\usepackage{srcltx}
\usepackage{hyperref}
\usepackage{upgreek} 
\usepackage[inline]{enumitem} 
\usepackage[dvipsnames]{xcolor} 

\usepackage{tikz}
\usetikzlibrary{matrix,shapes,arrows,arrows.meta,calc,topaths,intersections,hobby,positioning,decorations.pathreplacing,decorations.pathmorphing,fit,patterns}
\tikzset{cross/.style={cross out, draw=black, fill=none, minimum size=2*(#1-\pgflinewidth), inner sep=0pt, outer sep=0pt}, cross/.default={2pt}}

\usepackage{braket}

\DeclareFontFamily{U}{mathx}{}
\DeclareFontShape{U}{mathx}{m}{n}{ <-> mathx10 }{}
\DeclareSymbolFont{mathx}{U}{mathx}{m}{n}
\DeclareFontSubstitution{U}{mathx}{m}{n}
\DeclareMathAccent{\widecheck}{0}{mathx}{"71}

\definecolor{mygray}{gray}{0.75} 

\definecolor{shadecolor}{rgb}{1,0.9,0.7}

\newtheorem{theorem}{Theorem}[section]
\newtheorem{lemma}[theorem]{Lemma}
\newtheorem{lemma-definition}[theorem]{Lemma-Definition}
\newtheorem{proposition}[theorem]{Proposition}
\newtheorem{corollary}[theorem]{Corollary}
\newtheorem{conjecture}[theorem]{Conjecture}

\theoremstyle{definition}

\newtheorem{definition}[theorem]{Definition}
\newtheorem{construction}[theorem]{Construction}

\newtheorem{example}[theorem]{Example}

\theoremstyle{remark}
\newtheorem{remark}[theorem]{Remark}

\numberwithin{equation}{section}
\numberwithin{figure}{section}

\newcommand {\lfor} {\llbracket}
\newcommand {\rfor} {\rrbracket}

\newcommand{\NN} {\mathbb{N}}
\newcommand{\ZZ} {\mathbb{Z}}
\newcommand{\QQ} {\mathbb{Q}}
\newcommand{\RR} {\mathbb{R}}
\newcommand{\CC} {\mathbb{C}}
\newcommand{\FF} {\mathbb{F}}
\newcommand{\PP} {\mathbb{P}}
\renewcommand{\AA} {\mathbb{A}}

\newcommand {\shC}  {\mathcal{C}}

\newcommand {\shE}  {\mathcal{E}}
\newcommand {\shF}  {\mathcal{F}}

\newcommand {\shL}  {\mathcal{L}}

\newcommand {\shN}  {\mathcal{N}}
\newcommand {\shO}  {\mathcal{O}}

\newcommand {\shT}  {\mathcal{T}}

\newcommand {\shX}  {\mathcal{X}}

\newcommand{\bC}{\mathbb{C}}

\newcommand{\bP}{\mathbb{P}}

\newcommand{\bZ}{\mathbb{Z}}

\newcommand{\cM}{\mathscr{M}}
\newcommand{\cN}{\mathscr{N}}

\newcommand{\cS}{\mathscr{S}}

\newcommand {\can} {{\mathrm{can}}}

\newcommand {\ev}  {\operatorname{ev}}

\newcommand {\Hom}  {\operatorname{Hom}}

\newcommand {\hra} {\hookrightarrow}

\newcommand {\N}  {\operatorname{N}}

\newcommand {\NE}  {\operatorname{NE}}

\let \op  \operatorname

\renewcommand{\P}  {\mathscr{P}}

\newcommand {\Pic}  {\operatorname{Pic}}

\newcommand {\ra}  {\to}

\newcommand {\red}[1]{{\color{red} #1}}
\newcommand {\reg}  {{\mathrm{reg}}}

\newcommand {\Spec} {\operatorname{Spec}}

\newcommand {\Spf}  {\operatorname{Spf}}

\newcommand {\trop}  {{\operatorname{trop}}}

\newcommand {\vir}  {{\operatorname{vir}}}
\newcommand {\vdim}  {{\operatorname{vdim}}}

\newcommand {\vt} {\vartheta}

\def\mydate{\ifcase\month \or January\or February\or March\or
April\or May\or June\or July\or August\or September\or October\or 
November\or December\fi \space\number\day,\space\number\year}

\makeatletter
\let\oldcite\cite
\renewcommand{\cite}{\@ifnextchar[{\@newcite}{\oldcite}}
\def\@newcite[#1]#2{\oldcite{#2},\,#1}
\makeatother

\begin{document}

\title[LG potential is open mirror map]
{The proper Landau-Ginzburg potential\\ is the open mirror map}

\author{Tim Gr\"afnitz} 
\address{\tiny Tim Gr\"afnitz, University of Cambridge, DPMMS, Wilberforce Road, Cambridge, CB3 0WB, UK}
\email{tg485@cam.ac.uk}

\author{Helge Ruddat} 
\address{\tiny Helge Ruddat, JGU Mainz, Institut f\"ur Mathematik, Staudingerweg 9, 55128 Mainz,
Germany}
\email{ruddat@uni-mainz.de}

\author{Eric Zaslow} 
\address{\tiny Eric Zaslow, Department of Mathematics, Northwestern University,
Evanston, IL, USA}
\email{zaslow@math.northwestern.edu}

\begin{abstract}
The mirror dual of a smooth toric Fano surface $X$ equipped with an anticanonical divisor $E$ is a Landau--Ginzburg model with superpotential, $W$.  
Carl--Pumperla--Siebert give a definition of the the superpotential in terms of tropical disks \cite{CPS} using a toric degeneration of the
pair $(X,E)$. 
When $E$ is smooth, the superpotential is proper.  
We show that this proper superpotential
equals the open mirror map for outer Aganagic--Vafa branes in the canonical bundle $K_X$, in framing zero.  
As a consequence, the proper Landau--Ginzburg potential is a solution to the Lerche--Mayr Picard--Fuchs equation.

Along the way, we prove a generalization
of a result about relative Gromov--Witten invariants by Cadman--Chen to arbitrary genus using the multiplication rule of quantum theta functions. 
In addition, we generalize a theorem of Hu that relates Gromov--Witten invariants of a surface under a blow-up from the absolute to the relative case.
One of the two proofs that we give introduces birational modifications of a scattering diagram.
We also demonstrate how the Hori--Vafa superpotential is related
to the proper superpotential by mutations
from a toric chamber to the unbounded chamber of
the scattering diagram.
\end{abstract}

\maketitle
\setcounter{tocdepth}{1}
\tableofcontents
\vfill

\section{Introduction and Summary}
\label{sec:intro}

Mirror Symmetry relates Fano manifolds to Landau--Ginzburg theories.  
The Landau--Ginzburg superpotential $W$ is a virtual count of disks bounding Lagrangian torus fibers of a fibration in the complement of an anticanonical divisor. 
In the simplest case of a toric Fano with toric boundary as the anticanonical divisor, the Landau--Ginzburg superpotential is a Laurent polynomial defined on an algebraic torus.
Also associated to a Fano manifold $X$ is its canonical bundle, a Calabi--Yau manifold $K_X$ with mirror geometry determined by
the superpotential $W$.  
Aganagic--Vafa studied a class of Lagrangian A-branes in a toric $K_X$ and their mirror B-branes: the open mirror map relates parameters for their common moduli space \cite{AV}.  
In this paper we conjecture that the proper Landau--Ginzburg superpotential for a Fano manifold $X$ with
smooth anticanonical divisor is equal to the open mirror map, and establish this for the case $X=\bP^2$ and more generally for toric del Pezzo surfaces.

Smoothing the boundary divisor inside a toric Fano manifold $X$ yields an enlargement of the mirror dual to a quasiprojective variety $\widecheck{X}$.
Its algebra of global functions is generated by the Landau--Ginzburg potential $W$ which defines a proper morphism to $\AA^1$ in this situation (unlike the non-proper potential on the algebraic torus). 

This quasiprojective variety $\widecheck X$ is constructed from gluing infinitely many cluster torus charts that are glued by wallcrossing transformations. 
The corresponding walls and wall functions are encoded in a consistent 
scattering diagram that is contained in the dual intersection complex $B$ of the toric degeneration of X.  
This dual intersection complex is a real affine manifold $B$ with singularities and polyhedral decomposition $\P$. It is referred to as the \emph{fan picture} for $X$.
There are in total 16 toric degenerations within the linear system $K_X$ when $X$ is toric.
 
We follow the definition of the Landau--Ginzburg potential $W$ for $\widecheck X$ as the primitive \emph{generalized theta function}, i.e., the generating series of broken lines with primitive unbounded leg in $B$ \cite{CPS,GHS}. 
Counting of \emph{broken lines} for the scattering diagram is equivalent to counting tropical disks in $B$ and these in turn control the count of holomorphic disks in $X$. 

We focus on the case where $X$ is a toric Fano surface with smooth anticanonical elliptic curve $E$. 
The scattering theory invites a comparison between holomorphic disks in $X$ and $K_X$ and the log Gromov--Witten theory of the pair $(X,E)$, as
well as comparisons with the blow-up $\widehat{X}$ and the local Gromov--Witten theory of $K_{\widehat X}$ --- relationships we exploit to prove our main result.  
We also demonstrate how wall crossing relates the superpotentials for toric and smooth anticanonical divisors.

When $X = \bP^2$ and we take the toric anticanonical divisor, the superpotential was
computed by Hori--Vafa \cite{HV} through T-duality and proven mathematically by Cho--Oh \cite{CO}:  $W_0 = x + y + \frac{s}{xy},$
where $s$ is the complex modulus of the mirror, dual to the symplectic modulus of $\bP^2.$
When $E$ is a \emph{smooth} cubic curve, our theorem takes the following form.
Let
\begin{equation}
\label{eq:Meq}
M(Q) = 1 - 2Q + 5Q^2 - 32Q^3 + 286Q^4 - \cdots
\end{equation}
 be the power series defined by
$(Q/z(Q))^\frac{1}{3},$ where $z(Q)$ is the inverse of $Q(z) := z\exp(\sum_{n\geq 1}(-1)^n\frac{(3n)!}{n(n!)^3}z^n)$.  

\begin{theorem}
\label{thm-main-p2}
Assume $X=\PP^2$. We take a torus fiber
near $E$ and set $Q = -t^3y^{-3},$ the broken line expansion of the Landau--Ginzburg superpotential is
$$W = yM(Q) = y + 2t^3y^{-2} + 5t^6y^{-5} + 32t^9y^{-8} + 286t^{12}y^{-11} + \cdots.$$
\end{theorem}

Here, $y$ is the local cluster monomial on $\widecheck X$ given by the unique unbounded direction while $\widecheck X$ is the toric degeneration that is discrete Legendre dual to the canonical maximal degeneration of $\PP^2$. The variable $t$ is the toric degeneration parameter so that $t=0$ gives the degenerate fiber $\widecheck X_0$ whose four irreducible components are isomorphic to $\PP^2/((\ZZ/3)^2)$ and three copies of $\PP^1\times\AA^1$ and the two generators of $(\ZZ/3)^2$ act by weights $(1,-1,0)$ and $(1,0,-1)$. The map $W$ on this degenerate fiber is the projection to the $\AA^1$ factor and it maps $\PP^2/((\ZZ/3)^2)$ to the origin. Each fiber $t\neq 0$ gives an irreducible variety on which $W$ defines a proper elliptic fibration over $\AA^1$.
By ``near $E$'', we mean that for every fixed power $k$ of $t$, the count of tropical disks bounding a torus fiber in any unbounded chamber of the scattering diagram consistent to order $k$ agrees with the corresponding coefficient of $M(Q)$.

The function $M(Q)$ originally appeared in the work of Aganagic--Klemm--Vafa \cite{AKV} as the
open mirror map relating Aganagic--Vafa branes \cite{AV} on the canonical bundle $K_{\bP^2}$ and its mirror, and was used to determine open Gromov--Witten invariants in \cite{GZ}.

The series $M(Q)$ also appears in the work of Gross and Siebert \cite{GS-reconstruction,GS14} on local mirror symmetry
for the canonical bundle $K_{\bP^2}$ as a normalized slab function controlling the mirror toric degeneration, where they also provide an enumerative interpretation in terms of tropical
trees.

The works of Chan, Cho, Lau, Leung, Tseng (\cite{CLL,CLT,CCLT}) give an enumerative 
interpretation of $M(Q)$ in terms of 
open Gromov--Witten
invariants of moment fibers in $K_{\bP^2}$,
and these were related by Lau, Leung and Wu \cite{LLW}
to local invariants of the
blow-up $\widehat{\bP^2}=\mathbb F_1$,
following Chan's open-closed relation \cite{C11}.  See in particular \cite[Theorem 1.1 and Example 4.6]{LLW} and \cite[Example 6.5 (3)]{CCLT}.

We learned from Hiroshi Iritani that the function $M(Q)$ also appears as the relation $-3 S_{nc} = S_c \cdot M(Q)$ of elements in the multiplicative group of the quantum cohomology ring of $K_{\PP^2}$.
Here, $S_{nc}$ is the product of the three Seidel elements of the noncompact toric divisors while $S_c$ is the Seidel element of the zero section. 
Seidel elements arise from Hamiltonian circle actions and the given relation is obtained by combining Theorem 1.1 in \cite{GI} with Theorem 1.1 in \cite{CLLT}.

The various appearances of $M(Q)$ have in common that each is a generating function of an enumerative invariant where the power of $Q$ records the degree of an algebraic curve in $\PP^2$.
Yu Wang observed the appearance of the sequence of coefficients of $M(Q)$ in his study of punctured invariants for $\PP^2$, \cite{W}.

\begin{remark} 
\label{first-remark}
There is a conceptual reason behind the insertion of $Q = -t^3y^{-3}$ into $M$ in Theorem~\ref{thm-main-p2}. 
Let $\beta_\can$ denote the tropical cycle that is the tropicalization of a line in $\PP^2$ with its three legs ending on the amoeba of the zero locus of $y-w(t)$ for some holomorphic function $w(t)$.
As a consequence of \cite{RS}, it is computed in \cite{GRS} that the exponentiated chain integral $\exp\frac1{2\pi i}\int_{\beta_\can}\Omega$ is given by 
$-t^3w(t)^{-3}$. 
We may therefore identify $Q$ as the exponentiated chain integral for $\beta_\can$ with end points given by the value of the cluster coordinate $y$ itself.

While $W$ itself is not a period integral, $W/y$ is closely related to the anti-derivative of the classical period $\pi_W(t):=\frac{1}{2\pi i}\int_{S^1}\frac{1}{1-W(y,t)}\frac{dy}{y}$ which in turn agrees with the classical period of the Hori--Vafa mirror\footnote{This potential agrees with $W_0 = x + y + \frac{s}{xy}$ upon defining $t=s^{\frac{1}{3}}$ and rescaling variables by $1/t$.} potential $t(x+y+\frac{1}{xy})$ computed in Example 3.5 in \cite{CCGGK}. 
Quite generally, the proper potential can be obtained from its classical period by an analytic expression, see \cite{R3} for more details.
\end{remark}

Our main theorem for more general smooth toric Fano surfaces requires
the correct notions of open mirror map and Landau--Ginzburg potentials.
By ``open mirror map,'' we mean the relation between complex and symplectic open parameters defined by the logarithmic solution to the Picard--Fuchs differential
equations of Lerche--Mayr \cite{LM} --- see Section \ref{section-mirmap}.
Now let $X$ be a smooth toric Fano surface and $E$ a smooth anticanonical divisor.
Let $\widecheck X$ be the Gross--Siebert mirror dual to the canonical degeneration of the log Calabi--Yau pair $(X,E)$, see for instance Section~7 in \cite{CPS}.  
The unique primitive theta function $W:=\vartheta_1$ defines the Landau-Ginzburg potential on $\widecheck{X}.$
Our analysis of this Landau-Ginzburg potential is in two ways somewhat more general than what is needed to prove Theorem~\ref{thm-main-p2}: We also study $\vartheta_r$ for $r>1$ and we study its $\boldsymbol{q}$-refined version $\vt_r(\boldsymbol{q})$ which reduces to the unrefined version $\vartheta_r$ by setting $\boldsymbol{q}=1$. The reason for this more general treatment is on the one hand that it works with no extra effort using the same tools and arguments and on the other hand we want to use the description of $\vt_r(\boldsymbol{q})$ in an upcoming article joint with Benjamin Zhou where we study the open mirror map in higher genus and higher winding \cite{GRZZ}.

Let $y$ denote the primitive asymptotic outward-pointing monomial in the intersection complex $B$ of $\widecheck X$.
We show in Proposition \ref{prop:theta} that the $\boldsymbol{q}$-refined power series expansion of the theta function $\vt_q$ in every unbounded chamber in $B$ takes the form, with $\boldsymbol{q}=e^{i\hbar}$,
\begin{equation} 
\label{eq-prop-theta}
\vt_q(\boldsymbol{q}) = y^{q} + \sum_{p\geq 1}\sum_{\beta:\beta.E=p+q} R_{p,q}^{g,\text{trop}}(X,\beta) \hbar^{2g}s^\beta t^{\deg(\beta)}y^{-p}.
\end{equation} 
The non-negative integer $R_{p,q}^{g,\text{trop}}(X,\beta)$ is the number of tropical curves in $B$, counted with multiplicity, of genus $g$ and class $\beta$ with two unbounded legs of multiplicity $p$ and $1$. The first leg is required to pass through a fixed point near $E$. See \S\ref{sec:tropicalcorrespondence} for the precise definition. 
The homology group of tropical cycles $\beta$ is canonically isomorphic to the dual of the Picard group of $X$.
Let $\NE(X)$ denote the monoid of effective curve classes in $X$ and
let $M(Q)\in\QQ\lfor\NE(X)\rfor=\left\{\left.\sum_{\beta\in \NE(X)} a_\beta Q^\beta\right|a_\beta\in\QQ\right\}$ 
denote the open mirror map of $X$ as defined in Section \ref{section-mirmap}.

\begin{theorem} 
\label{thm-main}
Let $X$ be a smooth toric Fano surface.  Let $W(s,t,y) := \vt_1 := \vt_1(\hbar=0)$ be the Landau--Ginzburg potential near $E$ and $M$ the open mirror map, as defined above. 
Set $Q^\beta:= (-1)^{\text{deg}(\beta)}s^\beta\cdot t^{\text{deg}(\beta)}\cdot y^{-\beta.E}$. Then
$$W=yM(Q).$$
\label{delPezzo-theorem}
\end{theorem} 
\vspace{-.8cm}We need $X$ to be toric only for the open-closed correspondence in step \emph{(6)} below. 
The tropical correspondence theorem in step \emph{(2)} requires that $X$ has very ample anticanonical bundle to ensure that the set of tropical curves of a given class is finite.
Moreover, the tropical proof of \ref{eq-blowup-one-two-intro} in step \emph{(4b)} as presented is also restricted to surfaces with very ample anticanonical bundle.
All other steps work for general del Pezzo surfaces by using the degenerations provided in Section~7 in \cite{CPS}.
As mentioned in Remark~\ref{first-remark}, the substitution expression for $Q^\beta$ is up to sign an exponentiated chain integral which is computed in \cite{GRS}. 

\subsection*{Steps of the proof and auxiliary results}
The proof consists of a chain of identifications that connects tropical disks in a toric Fano surface to open mirror maps for branes in Calabi--Yau threefolds. 
\smallskip 

\emph{(1) Tropical disks to tropical curves.} 
Carl--Pumperla--Siebert give a tropical definition of the Landau--Ginzburg superpotential \cite{CPS}.  
We compute the potential in an unbounded chamber and show that all broken lines are parallel to the unique asymptotic direction of the chamber.
The Landau--Ginzburg potential $W$ is thus a Laurent power series in a single variable $y$ with coefficients in $\QQ[\NE(X)][t]$. 
Note the difference from Hori--Vafa potentials,
which are Laurent polynomials in two variables
with coefficients in $\QQ[\NE(X)][t]$, even though both relate by wall-crossing --- see Corollary~7.9 in \cite{CPS} for a proof, and Appendix~\ref{app:wallcrossing} for a demonstration.

There is a unique outbound disk that contributes a summand $y$ whereas all further disks are inbound and each of these contributes a summand $y^{-p}$ for $p>0$.
For each inbound disk, it is a simple matter to extend it to a tropical curve which then has unbounded edges of weight $1$ and $p$. 
We arrive at the statement of Proposition \ref{prop:theta} giving \eqref{eq-prop-theta}.
\smallskip 

\emph{(2) Tropical to log invariants.}
Tropical curves arise as degenerations of holomorphic curves and the degeneration formula relates tropical multiplicities to the enumerative counts. 
In Section \ref{sec:tropicalcorrespondence} we quote a theorem (Theorem~\ref{prop:trop}) of the first-named author which relates the tropical invariants $R_{p,q}^{g,\trop}(X,\beta)$ to analogously defined two-point log invariants $R_{p,q}^g(X,\beta)$, 
\begin{equation}
\label{eq-trop-log}
 R_{p,q}^{g,\trop}(X,\beta)=p \cdot R_{p,q}^g(X,\beta). 
 \end{equation}
Consequently, the tropical Landau--Ginzburg superpotential \cite{CPS} agrees with the one
defined by Auroux and Fukaya--Oh--Ohta--Ono \cite{A,FOOO,L}. 
\smallskip 

\emph{(3) Theta calculations.} 
Using the multiplicative structure of Gross--Hacking--Siebert's theta functions in Section~\ref{section-theta}, we prove Theorem~\ref{thm:cadman-chen} which gives us the identity
\begin{equation} 
R_{1,n}^g(X,\beta) = n^2 \cdot R_{n,1}^g(X,\beta)
\end{equation}
and generalizes a theorem of Cadman-Chen \cite{CC}.
\smallskip 

\emph{(4) Trading a contact point for a blow-up.} 
We provide two proofs relating two-point log Gromov--Witten invariants of $(X,E)$ to one-point invariants of the blow-up $\pi\colon\widehat X\ra X$ with exceptional curve denoted $C$. 
One proof is by the degeneration formula, the other is tropical; the relation \eqref{eq-trop-log} justifies either approach. 
We refer to genus zero and class $\alpha$ Gromov--Witten invariants in $(\widehat X,E)$ with a single point of contact at $E$ by $R(\widehat X,\alpha)$.
 With $q=\beta.E-1$, for genus zero invariants the relation takes the simplest form (if $g=0$ we omit the superscript $g$):
\begin{equation} 
\label{eq-blowup-one-two-intro}
 R(\widehat X,\pi^*\beta-C) = R_{1,q}(X,\beta).
\end{equation} 
\smallskip 

\begin{enumerate}
\item[\emph{(4a)}]
In Section \ref{section-two-point} we prove Theorem~\ref{thm-one-to-multi-point} that provides the identity
$$R^{g}(\widehat X,\pi^*\beta-pC)(\gamma) =  \sum_{{p_1+...+p_r=p}\atop{{g_0+...+g_r=g}\atop{r>0,p_i>0,g_i\ge 0}}}\frac{p_1\cdot...\cdot p_r}{r!}\cdot
R^{g_0}_{p_1,...,p_r,q}(X,\beta)(\gamma)\cdot \prod_{i=1}^r N(g_i,p_i).$$
Here, $R^{g}_{p_1,...,p_r,q}(X,\beta)$ refers to genus $g$ invariants with the implicit insertion of the class $(-1)^g\lambda_g$ and with $r+1$ contacts at $E$ with tangencies $p_1,...,p_r,q$ where the first $r$ contacts are at prescribed point positions on $E$. The number $N(g,p)$ is an invariant of $\PP^1$ that is defined in \eqref{eq-cover-contribution}. 
We prove this result by degenerating $\widehat{X}$ to the union of $X$ and one another component, then analyzing the
contributions of various types of stable maps that occur in the degeneration formula. The result generalizes a theorem by \cite{Hu} to the relative situation.
\smallskip 
\item[\emph{(4b)}]
In Section \ref{sec:scattering-blowup}, we introduce the ``blow-up of a scattering diagram'' (see Figure~\ref{fig:blowup1}), and use it to construct a bijection between the sets of tropical curves with two unbounded legs in $X$ of class $\beta$ to the set of tropical curves with one unbounded leg in $\widehat{X}$ of class $\pi^\star\beta-C$ resulting in 
Theorem \ref{thm:blowupfan}, Corollary \ref{cor:tropcorrg} and a proof of the tropical analogue of \eqref{eq-blowup-one-two-intro}. Since the relevant tropical curves were shown to be equivalent to broken lines in \cite{CPS}, equivalently, one may phrase the bijection as one for broken lines.
We identify the initial wall structure of $X$ inside that of $\widehat{X}$. 
Uniqueness of scattering then relates the full structures and yields the desired correspondence of tropical curves, where
erasing a piece of a tropical curve in $\widehat{X}$ yields a similar curve in $X$ with one fewer unbounded leg.

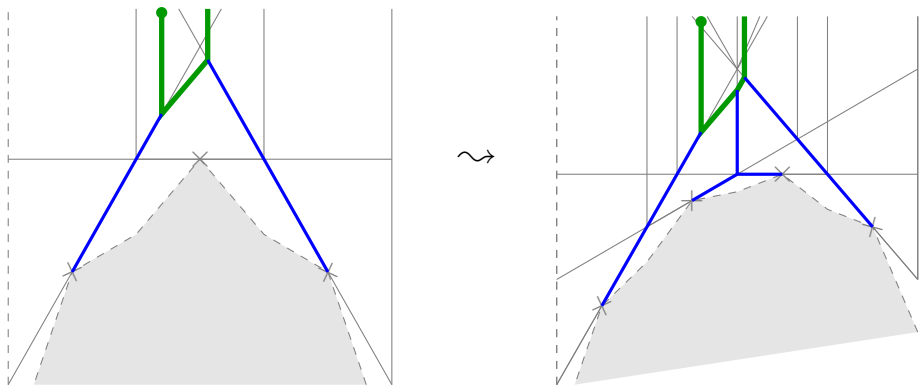
\begin{figure}[h!]
\centering
\captionsetup{width=.9\textwidth}
\begin{tikzpicture}[xscale=1.7,yscale=1,rotate=90]
\draw[gray] (-3,-1) -- (0,0) -- (0,1) -- (-3,2);
\draw[gray] (-3,-1) -- (2,-1);
\draw[gray] (0,0) -- (2,0);
\draw[gray] (0,1) -- (2,1);
\draw[dashed,gray] (-3,2) -- (2,2);
\draw[dashed,fill=black,fill opacity=0.2,gray] (-3,-0.8) -- (-1.5,-0.5) node[opacity=1,rotate=20]{$\times$} -- (-1,0) -- (0,0.5) node[opacity=1,rotate=0]{$\times$} -- (-1,1) -- (-1.5,1.5) node[opacity=1,rotate=70]{$\times$} -- (-3,1.8);
\draw[gray] (0,0.5) -- (0,2);
\draw[gray] (0,0.5) -- (0,-1);
\draw[gray] (-1.5,-0.5) -- (2,2/3);
\draw[gray] (-1.5,1.5) -- (2,1/3);
\draw[black!40!green,line width=2pt,line width=2pt] (1.95,0.8) node[fill,circle,inner sep=1.5pt]{} -- (0.6,0.8) -- (1.32,0.44) -- (2,0.44);
\draw[blue,line width=1.2pt] (-1.5,-0.5) -- (1.32,0.44);
\draw[blue,line width=1.2pt] (-1.5,1.5) -- (0.6,0.8);
\draw (0,-1.67) node{\large$\leadsto$};
\draw (0,-2.2);
\end{tikzpicture}
\begin{tikzpicture}[xscale=1.2,yscale=0.7,rotate=90]
\draw[gray] (-2,-1) -- (0,0) -- (0,1) -- (-1,2) -- (-4,3);
\draw[gray] (-2,-1) -- (3,-1);
\draw[gray] (0,0) -- (3,0);
\draw[gray] (0,1) -- (3,1); 
\draw[gray] (-1,2) -- (3,2);
\draw[dashed,gray] (-4,3) -- (3,3);
\draw[dashed,fill=black,fill opacity=0.2,gray] (-3,-1) -- (-1,-0.5) node[opacity=1,rotate=30]{$\times$} -- (-2/3,0) -- (0,0.5) node[opacity=1,rotate=0]{$\times$} -- (-1/3,1) -- (-1/2,1.5) node[opacity=1,rotate=45]{$\times$} -- (-5/3,2) -- (-2.5,2.5) node[opacity=1,rotate=70]{$\times$} -- (-4,2.8);
\draw[gray] (0,0.5) -- (0,3);
\draw[gray] (0,0.5) -- (0,-1);
\draw[gray] (-0.5,1.5) -- (2,-1);
\draw[gray] (-0.5,1.5) -- (-2,3);
\draw[gray] (-1,-0.5) -- (-2,-1);
\draw[gray] (-2.5,2.5) -- (-4,3);
\draw[gray] (-1,-0.5) -- (2,1);
\draw[gray] (2,1) -- (3,1+1/2);
\draw[gray] (2,1) -- (3,1+1/3);
\draw[gray] (-2.5,2.5) -- (2,1);
\draw[gray] (2,1) -- (3,1-1/3);
\draw[gray] (2,1) -- (3,1-1/4);
\draw[gray] (2/3,1/3) -- (3,1/3);
\draw[gray] (0,1+2/3) -- (3,1+2/3);
\draw[gray] (-2,-1) -- (3,-1);
\draw[gray,dashed] (-4,3) -- (3,3);
\draw[black!40!green,line width=2pt,line width=2pt] (2.9,1.4) node[fill,circle,inner sep=1.5pt]{} -- (0.8,1.4) -- (1.6,1) -- (1.84,0.92) -- (3,0.92);
\draw[blue,line width=1.2pt] (-1,-0.5) -- (1.84,0.92);
\draw[blue,line width=1.2pt] (0,0.5) -- (0,1) -- (-0.5,1.5);
\draw[blue,line width=1.2pt] (0,1) -- (1.6,1);
\draw[blue,line width=1.2pt] (-2.5,2.5) -- (0.8,1.4);
\end{tikzpicture}
\caption{A broken line (fat green) in the fan picture of $\mathbb{P}^2$ (left) and the corresponding broken line in the fan picture of $\widehat{\mathbb{P}^2}=\FF_1$ (right). The respective extensions to tropical curves are indicated in thick blue.}
\label{fig:brokencorr-intro}
\label{broken-curves}
\end{figure}
\end{enumerate}
\smallskip 

\emph{(5) log-local.}  
The work of \cite{GGR} establishes a correspondence between one point log Gromov--Witten invariants of $\widehat{X}$ and local Gromov--Witten invariants of the canonical bundle $K_{\widehat{X}}$, for $n=\beta.E$, in genus $0$,
\begin{equation}  
R(\widehat X,\beta) = (-1)^{n+1} n\cdot N(K_{\widehat{X}},\beta). 
\end{equation} 

Combining the identities in \emph{(1),(2),...,(5)} above, we obtain $$R_{\beta.E-1,1}^{\text{trop}}(X,\beta)=(-1)^{p+1}N(K_{\widehat{X}},\pi^*\beta-C),$$ hence 
\begin{equation} 
\label{eq-series-with-N}
 W/y = 1+\sum_{p\geq 1}\sum_{\beta : \beta.E=p+1} (-1)^{p+1}N(K_{\widehat{X}},\pi^*\beta-C) s^\beta t^{\deg(\beta)} y^{-p-1}.  
\end{equation}
\smallskip 

\emph{(6) Closed-open correspondence.}  
We finally use Theorem~1.1 in \cite{LLW} which establishes the identity 
\begin{equation}  
\label{eq-open-closed}
N(K_{\widehat{X}},\pi^*\beta-C)=n_{\beta+\beta_0}(K_X)
\end{equation}  
where $\beta_0$ is a holomorphic disk contained in a fiber of the projection $K_X\ra X$ and $n_{\beta+\beta_0}(K_X)$ denotes the Fukaya--Oh--Ohta--Ono-theoretic holomorphic disk invariant with boundary in a moment map fiber of disk class $\beta+\beta_0$. See Example~4.6 in \cite{LLW} and Example~6.5\,(3) in \cite{CCLT}.
The series in \eqref{eq-series-with-N}
is determined through closed mirror symmetry via
the hypergeometric series which produce the function $M(Q)$ in the case of $\PP^2$ and similar functions for other toric del Pezzo surfaces.
%

\begin{remark}
\label{rmk:cetalwork}
There are other open-closed relations in the literature.
The expansive series of works by
Chan, Cho, Lau, Leung, Tseng and Wu (\cite{CLL,CLT,CCLT,LLW}),
building on Chan's remarkable open-closed result \cite{C11},
relates open invariants on toric fibers of
the canonical bundle
of a toric surface to local Gromov--Witten invariants.
For Aganagic--Vafa branes, an open-closed relationship was found in \cite{GZ} and employed in \cite{FL} for
the so-called outer branes which are relevant here, see also \cite{BBvG}.
The computations of \cite{FL}
agree with Mayr's calculations using differential
equations \cite{M06}, a method further developed in \cite{LM}. 
The equivalence of all these methods establishes the folklore
result
that outer Aganagic--Vafa branes in framing zero
have the same invariants
as a moment map fiber in the toric Calabi-Yau.  
It is the open mirror map for Aganagic--Vafa branes, as computed first in \cite{AKV}, to which
the title of this paper refers.
\end{remark}

\begin{remark} 
By counting broken lines, we compute the tropical invariants for $\PP^2$ for degree $1$ and $2$ and arbitrary genus in Appendix~\ref{app:brokenlines}. 
In Appendix~\ref{app-verification} we use \cite{BFGW} to obtain the log invariants from local invariants to verify the computations in Appendix~\ref{app:brokenlines}. 
In this context, also the suitability of the insertion of the class $(-1)^g\lambda_g$ for the higher genus results in \emph{(4)} was discovered in \cite{Bou20}.
Steps \emph{(3)} and \emph{(4)} also hold true for positive genus. 
Concerning step \emph{(5)}, a higher genus version of this correspondence was proved in \cite{BFGW}.
We checked whether it permits an easy generalization of the statement of Theorem~\ref{thm-main} to higher genus but could not confirm this. 
\end{remark}

\begin{remark}
Tonkonog studies the proper Landau--Ginzburg potential $W$ from a symplectic topology point of view as the disk potential of an exact monotone Lagrangian torus \cite{Ton}. 
He relates the potential to the Borman--Sheridan class, a certain deformation class in the symplectic cohomology ring. Collins, Jacob and Lin proved the existence of a special Lagrangian fibration in the complement of a smooth elliptic curve in every del Pezzo surface \cite{CJL1}. They show that the SYZ mirror is a fibre-wise compactification of the Landau-Ginzburg model \cite{CJL2}. Furthermore, Lin proved that broken line counts agree with Maslov index two disks \cite{L}. Lau, Lee and Lin proved that the scattering diagram in the fan picture of $\PP^2$ from \cite{CPS,Gra1} matches the Floer theoretic scattering diagram of Maslov zero disks \cite{LLL1}. The relationship of the Landau-Ginzburg potential with Hyperkähler rotation has been explored in \cite{LLL2}. 
\end{remark}


\begin{remark}
Theta functions are typically hard to compute. Theorem~\ref{thm-main} computes $\vt_1$ in the region where the walls of the relevant scattering diagram are dense which can be seen as follows. If two rays intersect with determinant $m$, the scattering diagram is locally isomorphic via a change of lattice (see \cite{GHKK}, Proposition C.13) to the one obtained from the ``standard'' scattering diagram with function $1+atx^m$ and $1+bty^m$, for some $a,b\in\mathbb{Z}$. This latter diagram has a dense region given by the convex cone spanned by slopes $\frac{m-\sqrt{m^2-4}}{2}$ and $\frac{m+\sqrt{m^2-4}}{2}$ as can be shown by identification with a similar diagram from quiver representations \cite{Rei} or via mutations \cite{Pr20}. 

Prince studied the scattering diagram of $\mathbb{P}^2$ in the complement of the region where walls are dense \cite{Pr20}. Outside the dense region, all rays have coefficient $1$ since they are related via mutations to the initial rays. Prince shows that the chambers of the non-dense region are related to the bounded central cell via mutations of polytopes, and that these chambers are the Newton polytopes of the superpotential defined in them. By contrast, we are interested in the superpotential in a nested sequence of unbounded chambers. Every such a sequence leaves the region of bounded cells.

Toric degenerations have recently also been used to equate Laurent polynomial disk potentials in non-toric cluster variety situations with Floer theoretic potentials in \cite{KLZ}.
\end{remark}

\begin{remark}
While most of the steps above work more general, the closed-open correspondence $(6)$ is only established for winding number $p=1$ and genus $g=0$.
In future work joint with Benjamin Zhou \cite{GRZZ} we plan to extend the results of this paper to higher windings and higher genus.
\end{remark}


The superpotential is defined in each chamber of the wall structure determined by the dual intersection complex associated to the toric degeneration.  
Since this potential can be easily computed in the central chamber to agree with the Hori--Vafa superpotential, we can fix an order $k$ for the
wall structure and relate that to the function $M(Q)$ mod $t^k$
by a finite series of mutations \`a la \cite{CPS}.
An example of this mutation process is given in Appendix \ref{app:wallcrossing}.

The following conjectures concern the non-toric case and higher-dimensional situations.  
As there is no known system of Lerche--Mayr type in the non-toric case, we take \eqref{def-mirror-map} as a definition of the open mirror map.

\begin{conjecture}[intrinsic mirror construction] 
\label{conj-int-MS}
Given a Fano manifold $Y$ with smooth anticanonical divisor $D$, let $(X,W)$ be the intrinsic mirror dual, that is, for a suitable submonoid $P\subset \Pic(Y)$,  
$X\ra \Spf\CC[\Pic(Y)^*]\lfor P\rfor$ is a formal degenerating family of affine Calabi--Yau manifolds with $W\in\Gamma(X,\shO_X)$ the unique primitive theta function. 
The broken line expansion of $W$ in a chamber of the wall structure near infinity is the open mirror map of $K_Y$.
\end{conjecture}

\begin{conjecture}[symplectic geometry and FOOO theory] 
Let $Y$ be a symplectic Fano manifold with smooth anticanonical divisor $D$ so that $Y\setminus D$ supports a Lagrangian torus fibration that collapses an $S^1$ along $D$.
Let $L$ be a torus fiber near $D$ and $\delta$ the disc class of the thimble of the collapsing $S^1$. 
Let $m_0$ denote Fukaya--Oh--Ohta--Ono's obstruction term of the Lagrangian Floer complex for $L$ with itself.
Then $\exp(\int_\delta\omega)\cdot m_0$ is the open mirror map of $K_Y$.
\end{conjecture}

\subsection*{Acknowledgements}
We are grateful to many people for discussions about this work, in particular Denis Auroux, Andrea Brini, Kwokwai Chan, Tobias Ekholm, Michel van Garrel, Tom Graber, Mark Gross, Samuel Johnston,
Albrecht Klemm, Siu-Cheong Lau, Chiu-Chu Melissa Liu, Travis Mandel, Dhruv Ranganathan, Vivek Shende, Bernd Siebert, and Yu Wang.
H.R.~received support from DFG grant RU 1629/4-1 and E.Z.~ received support from grant NSF-DMS-1708503 and NSF-DMS-2104087. T.G.~received support from the ERC
Advanced Grant MSAG.

\section{The open and closed mirror maps}
\label{section-mirmap}

The purpose of this section is to elaborate on our distinction between open and closed mirror maps,
to give some historical context, and to discuss their enumerative interpretation as
open Gromov--Witten invariants. 
Most of the material is not needed in the main result, so the reader is free to take
Equation \ref{def-mirror-map} as a definition of the open mirror map and move on.
We assume that the reader who remains
is generally familiar with the enumerative form of mirror symmetry, so that the following lines serve as a reminder.
Given a toric Fano surface $X$ where $D\subset X$ denotes the toric boundary divisor, the mirror dual of the pair $(X,D)$ is a Landau--Ginzburg theory $W: (\bC^\times)^2\to \bC$ depending
on a complex parameter (or several parameters) $z$ which is related by the \emph{closed mirror map} to the symplectic parameter (or several parameters) $Q$
of $X$.
Relatedly, the total space of the canonical bundle $K_X$ is mirror dual to the subvariety $uv = W(x,y)$ in $\bC^2 \times (\bC^\times)^2,$
and has the same closed mirror map $z(Q)$.  
Suitable combinations $f(z)$ of period integrals over three-cycles in the variety $uv = W(x,y)$, when written as $f(z(Q))$, enumerate closed Gromov--Witten invariants
of $K_X$.  

Now let us turn to open mirrors.  
On $K_X$ there is a class of non-compact Lagrangian A-branes whose moment map image is a ray \cite{AV,FL}. 
The generating function of open Gromov--Witten invariants with boundary on the Lagrangian is a power series in a complex modulus $U$ encoding the symplectic area the disk together with the $U(1)$ holonomy of its boundary circle. Mirror dual to such A-branes are B-branes supported on $\{v=0\}$ and their moduli space is the mirror curve $C = \{W(x,y)=0\}\subset (\bC^\times)^2$.
By solving for $y$ in terms of $x$, we can then use $x$ as a local parameter for the moduli space $C$ around some reference point $x_0$.
Similar to the closed Gromov--Witten situation before, it is possible to extract open Gromov--Witten invariants from period integrals over chains whose boundary is the difference between B-branes described by $x_0$ and $x$. 
Since this period integral depends on $x$, we need to relate $x$ and $U$, and this is the \emph{open mirror map},
which will also depend on $Q$.\footnote{
You can vary the brane modulus $x$ in a fixed ambient geometry described by $z$, so the open mirror map depends as well on $z$, or equivalently $Q$. Geometrically, the open string moduli space fibers over the closed string moduli space.  Both can be combined into a single space, as done in \cite{M06,LM}.}

Somewhat confusingly, it turns out that for $X=\PP^2$, the closed and open mirror maps are related by taking a third root,
though they play different roles.
Both maps are determined by the Picard--Fuchs differential equation for $C$ as we discuss in the next section.

\subsection{The case of $\bP^2$:
Picard--Fuchs equations and mirror maps}
\label{sec:caseofp2}
In local mirror symmetry, the closed mirror map is
obtained from the logarithmic solution of the Picard--Fuchs differential equation \cite{CKYZ}.
In the case of $X=\PP^2$, the equation is\footnote{\emph{c} stands for ``closed.''} $\shL_{\rm c} f=0$ where
$\shL_{\rm c}:=\theta^3+3z\theta(3\theta+1)(3\theta+2)$ and $\theta=z\frac{d}{dz}$,
and the logarithmic solution is
$$t(z) = \log(z) + F(z),
\qquad F(z)=\sum_{k\ge 1}(-1)^{k}
\frac{(3k)!}{k\cdot (k!)^3} z^k.$$
Defining the closed-string symplectic parameter
$Q := e^t,$ we get
$$Q=ze^{F(z)}=z-6z^2+63z^3-866z^4+13899z^5-246366z^6+...$$
The inverse relationship is $ z=Q+6Q^2+9Q^3+56Q^4
-300Q^5+3942Q^6+... .$

As for open and closed moduli, the mirror maps have
been determined by \cite{AKV} in several examples,
following the pioneering work of Aganagic--Vafa \cite{AV}.  Generalizing to more general toric examples,
Mayr \cite{M06} and Lerche--Mayr \cite{LM} constructed differential
equations in both open and closed variables which determine the mirror map.
In this example, the equations are
$\shL_{\rm{oc}}^{(1)} f=0, \shL_{\rm{oc}}^{(2)} f = 0$ where
$$\shL_{\rm{oc}}^{(1)} :=\theta^2 (\theta - \theta_{\rm o}) + z(3\theta - \theta_{\rm o})(3\theta-\theta_{\rm o}+1)(3\theta-\theta_{\rm o}+2)$$
$$\shL_{\rm{oc}}^{(2)} := (\theta_{\rm o}-3\theta)\theta_{\rm o} - x(\theta_{\rm o}-\theta)\theta_{\rm o}$$
and $\theta_{\rm o} = x\frac{d}{dx}$, where $x$ is the open complex variable.
Note that when $f = f(z)$ is independent of $x$, the second equation is solved and the first equation reduces to the closed equation --- the logarithmic solution with no $x$-dependence
therefore gives the closed mirror relation $t(z).$
The logarithmic solution with nontrivial $x$ dependence
has the form
$t_{\rm o} = \log(x) + G(z)$, and in fact $G(z) = -\frac{1}{3}F(z).$  Exponentiating, we define
the open-string symplectic parameter
$U := e^{t_{\rm o}}$ and find
$x = Ue^{\frac{1}{3}F} = U(1 - 2z + 17z^2 - 218z^3 + 3404z^4 - 59644z^5 + 1127009z^6 + ....),$
or after substituting $z(Q)$ from above,
$$x = U(1 - 2Q + 5Q^2 - 32Q^3 + 286Q^4 - 3038Q^5 + 35870Q^6- 454880Q^7+ 6073311Q^8+...)$$


Note that $x/U = M(Q)$.  We therefore refer to $M(Q)$ as the
\emph{open mirror map}.
In more general settings, the open mirror map
is similarly defined:  exponentiate the logarithmic
solution of the Lerche--Mayr system
with nontrivial dependence on the
complex open-string parameter.
The open mirror map has an enumerative interpretation,
which we describe below.

By deriving the Picard--Fuchs differential
equations of the mirror geometry,
Lerche--Mayr \cite{LM} compute
the superpotential of the effective $4d$, $\cN=1$ theory arising from a IIA string compactification
$(K_X,L)$, where $X$ is a toric surface and $L$ is a spacetime-filling outer
Aganagic--Vafa
D$6$-brane. 
Following Aganagic--Vafa \cite{AV}, 
this $4d$
superpotential $\mathcal W$ (not to be confused with the Landau--Ginzburg potential $W$)
should be the generating function of disk invariants. 
These disk invariants can be computed directly, as well.
The case $X = \bP^2$ was studied
by Graber and the third author, albeit for an inner brane:
the open Gromov-Witten invariants were defined through localization,
then related to closed invariants --- see \cite[Section 4]{GZ}.
Inner branes and other framings in the $\bP^2$ case
were studied by Fang--Liu --- see \cite[Section 6.2]{FL}.
The invariants are then computed by employing the equivariant
mirror symmetry theorem.  The different methods agree.  Continuing
with the Lerche--Mayr approach,
when $X = \bP^2,$ the 4d superpotential
$\mathcal W$ is given by \cite[Equation (A.3)]{LM}. In our
our notation, this reads
\begin{equation}
\label{eq:lmpot}
{\mathcal W} = \sum_{n>m\geq 0}{(-1)^m} \frac{(n-m-1)!}{n(n-3m)!(m!)^2}x^nz^m.
\end{equation}
Substituting $x = U(1 - 2Q + ...)$ and $z = Q - 6Q^2 + ...$, the coefficient of $U^w Q^d$ gives the disk invariant in
class $(d,w)\in H_2(K_X,L) \cong H_2(X)\oplus H_1(L) \cong \bZ^2.$
We are interested in the invariants with
winding $w=1$.  Looking at the sum in Equation \eqref{eq:lmpot}, since $U$ appears with $x$ alone,
this requires $n = 1$ and therefore $m = 0$, and we get simply $x$.
Thus
$$x = U - 2UQ + 5UQ^2 - 32UQ^3 + 286UQ^4 + ...$$
is the generating function of disk invariants with winding $1.$

As mentioned in Remark \ref{rmk:cetalwork},
similar results were obtained for moment fibers,
rather than Aganagic--Vafa outer branes, in \cite{CLL} ---
see their Remark 5.7.

\subsection{Other smooth toric Fano surfaces}
\label{sec-open-mirror-map}

The story we have told above for $X = \bP^2$ above
remains true for other toric Fano surfaces --- see Equation (2.23) of \cite{LM}.
The work of Fang-Liu proves the validity
of the Lerche--Mayr
method --- see Section 5.3
of \cite{FL}.
In \cite[Section 6]{FL}, the canonical
bundles for each of the five toric Fano
surfaces are studied.  As with $K_{\bP^2}$,
each has an outer brane whose
generating function of disk invariants with winding $1$
and in framing zero
agrees with the open mirror map defined
by the Lerche--Mayer equations.\footnote{
The recent work of Liu--Yu \cite{LY} proves that the Lerche--Mayr
program also computes four-fold invariants, see also \cite{YZ}.}

The works of \cite{LLW,CCLT}
concern holomorphic disks bounding moment fibers
of toric Calabi--Yau threefolds, which for us are $K_X$
with $X$ a smooth toric Fano surface.
By \cite[Theorems 1.1 and 4.5]{LLW}, these
are equivalent to local (closed) Gromov--Witten invariants
of the canonical bundle of the blow-up, $K_{\widehat{X}}.$
Formulas for these are given in
\cite[Equation (6.14) and Propositions 6.14 and 6.15]{CCLT}.
Comparison with \cite[Equations (35)--(37)]{FL},
shows that counts of Aganagic--Vafa branes in framing
zero are the same as those of
moment fibers.\footnote{A few comments on this are
in order.  First, the methods of \cite{LLW} require
the blow-up $\widehat{X}$ to be toric, and while
this is not true for the del Pezzo surface $dP_3$,
the result still holds for this marginal case, as
was kindly explained
to us by Siu-Cheong Lau. 
Second,
Aganagic--Vafa branes
in the resolved conifold are large-$N$ dual to
unknot conormals and have generalizations to other
toric Calabi--Yau manifolds, but the large-$N$ duals
of general knot conormals have no known (to us)
generalizations.}

We summarize these results here.
Let $X$ be a smooth toric del Pezzo surface and $L$ be a general fiber of the toric moment map of $K_X$. 
Let $\beta_0$ denote the holomorphic disc that is embedded in a fiber of the projection $K_X\ra X$ and has boundary on $L$.
For $\beta\in \op{NE}(X)$ a curve class in $X$, following Section~4 in \cite{LLW}, we let $n_{\beta_0+\beta}$ denote the genus zero open Gromov--Witten invariant of $(K_X,L)$ with a single marking on $L$ and of class $(\beta_0+\beta)\in\pi_2(K_X,L)$.
Set $M\in \QQ\lfor\op{NE}(X) \rfor=\{\sum_{\beta \in \op{NE}(X)} a_\beta Q^\beta \mid a_\beta\in\QQ\}$ to be the generating function
\begin{equation}
\label{def-mirror-map}
M=\sum_{\beta \in \op{NE}(X)} n_{\beta_0+\beta} Q^\beta.
\end{equation}
Invert the (closed) mirror map $z(Q)$ --- namely,
$z(Q)$ is a multicoordinate of
logarithmic solution to the local
Picard--Fuchs equation for $X$ \cite{CKYZ} --- to write $Q = Q(z)$.  Then $M(Q(z))$ is the
solution to the open Picard--Fuchs equations
of \cite{LM} with leading dependence $\log(x)$,
where $x$ is the open-string complex parameter
of \cite{LM}, as in Section
\ref{sec:caseofp2}.
That is, $M(Q)$ is the open mirror map.  Due to these equivalences, we can take Equation \ref{def-mirror-map} as
a mathematical definition of the open mirror map.

\section{The Landau-Ginzburg potential in the proper case}
\label{section-LG}
In this section, we first recall the construction of a Landau--Ginzburg model from \cite{CPS}. 
The construction takes as input a triple $(B,\P,\varphi)$ where $B$ is an integral affine manifold with singularities, $\P$ is a polyhedral decomposition of $B$ into lattice polyhedra and 
$\varphi\colon B\to\RR$ is a piecewise affine function that is strictly convex with respect to the polyhedral decomposition. In the most general form, $\varphi$ may be multi-valued but it is going to be single-valued for the examples relevant to this article. We the define the triple $(B,\P,\varphi)$ for the mirror dual of $\PP^2$.
We give a summary about how the work \cite{CPS} takes $(B,\P,\varphi)$ as input and then outputs a formal toric degeneration $\widecheck{\frak{X}}\ra\Spf A\lfor t\rfor$ and a regular function $W\colon \widecheck{\frak{X}}\ra\AA^1$. 
We then introduce a particular base ring $A_N$ to use in place of $A$ and relate it to unbounded tropical cycles in $B$ in order to obtain a very natural description for $W$ when $B$ is asymptotically cylindrical. Except for {\S}\ref{visual-surface}, {\S}\ref{subsec-Fanosurface} and {\S}\ref{subsec-qwalls}, this section treats $B$ of arbitrary dimension. 
We specialize to the surface case in Section~\ref{sec:tropicalcorrespondence}.

\begin{figure}[H]
    \centering
    \includegraphics[scale=.55]{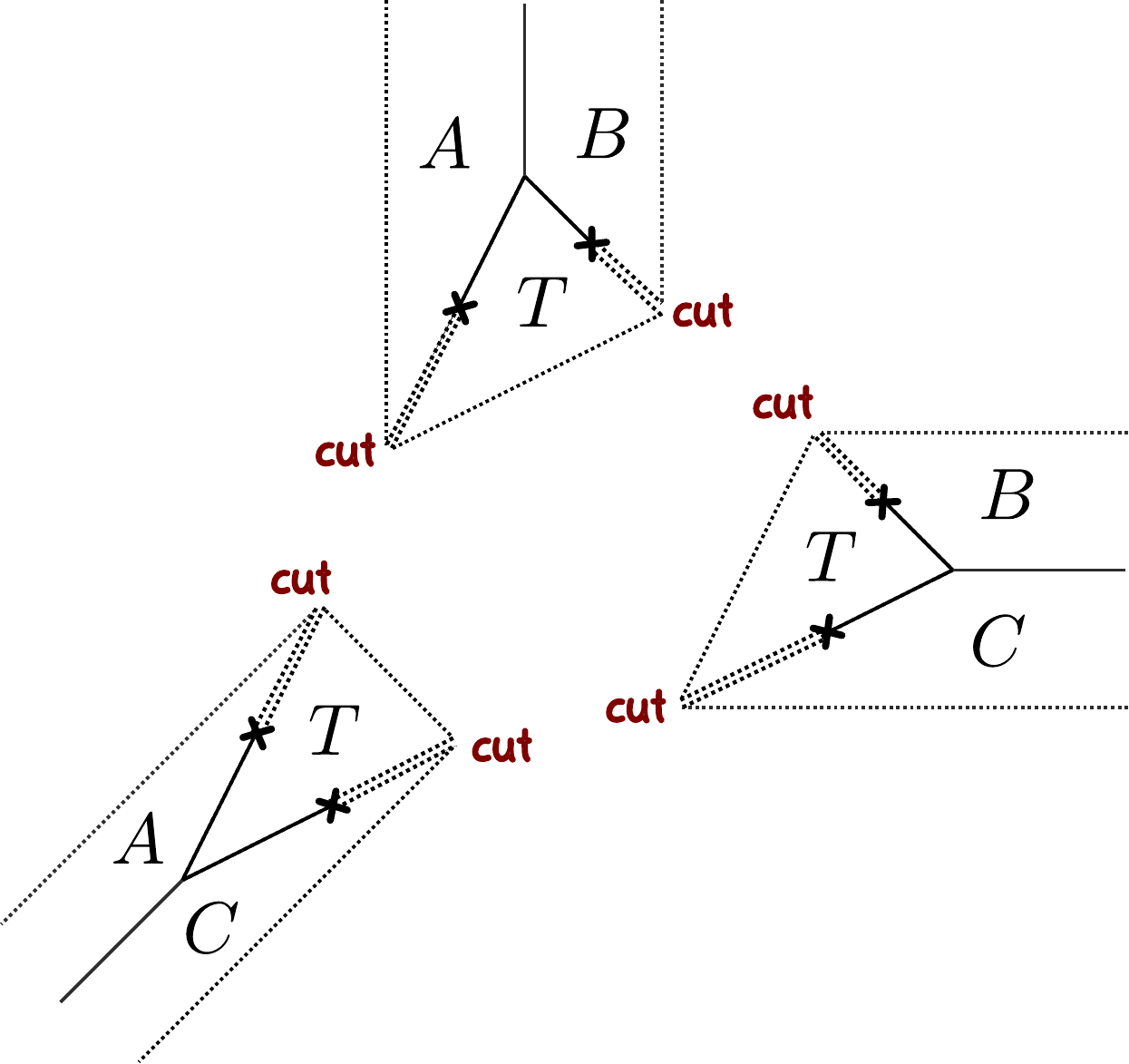}
    \caption{$A,B,C$ and $T$ represent the images of regions of the underlying singular affine surface in various coordinate charts.  The intersection of any two charts will be the disjoint union of the interiors of $T$ and one of $A,B,C.$  The transition function is the identity along $T$ and determined by the conditions it maps the singularity to itself, leaves the edge of the triangle invariant, and identifies the remaining regions labeled by the same letter.  For example, if we take the upper right singularity to be the origin of the plane, the upper and right charts are related by the linear transformation $\binom{\;2\;\;1}{-1\;0}$.}
    \label{fig:affine-charts}
\end{figure}

\subsection{Visualizing the polarized integral affine surface $(B,\P,\varphi)$ that underlies the mirror dual of $\PP^2$}
\label{visual-surface}
One can visualize a point singularity of a surface with integral affine structure in a manner similar to a branch cut in the complex plane:
surround the singularity with two contractible open sets intersecting in two components,
and specify the two transition functions between coordinate patches.  One then arranges that one of the transition functions is the identity, so that both coordinate regions are locally identified with two overlapping sets in the same Euclidean plane,
with the remaining transition function identifying the two regions corresponding to the other component of the intersection.
In our cases, the local monodromy around a point will be a cyclic element of ${\mathbb Z}^2\rtimes SL_2(\bZ)$ with a single invariant direction, such as 
$\binom{1\;n}{0\; 1}$, with $n\neq 0$, or a conjugate.
Such an affine transformation of a surface is determined by specifying where a point goes (the singularity will be mapped to itself), an invariant direction, and the image of any transverse ray.  The information is then easily encoded pictorially, see Figure~\ref{fig:affine-charts}.

Another useful visual is to avoid any overlapping, similar to how one visualizes a cone angle. 
We stress, however, that the affine transformation is an infinite-order shear, not a rotation.
To do so above, in the top figure we take the left half of region $B$ after slicing it with a vertical line from the singularity,
then the bottom half of region $B$ in the right picture, after slicing it with a horizontal line.  Then we could draw these two pictures on the same plane, with wedge-shaped region removed and the boundary edge identified by the shear.  This is indicated in the diagram shown in Figure~\ref{fig:P2-cones} for the affine manifold of interest to us.

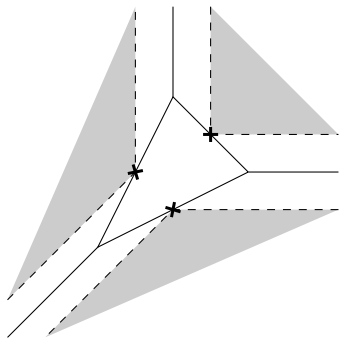
\begin{figure}[ht]
\centering
\begin{tikzpicture}[scale=1]
\draw (1,0) -- (0,1) -- (-1,-1) -- (1,0);
\draw (1,0) -- (2.2,0);
\draw (0,1) -- (0,2.2);
\draw (-1,-1) -- (-2.2,-2.2);
\draw[dashed,fill=black,fill opacity=0.2] (0.5,2.2) -- (0.5,0.5) node[opacity=1,rotate=45]{} -- (2.2,0.5);
\draw[line width=0.4mm] (0.5,0.4) -- (0.5,0.6);
\draw[line width=0.4mm] (0.4,0.5) -- (0.6,0.5);
\draw[dashed,fill=black,fill opacity=0.2] (-2.2,-1.7) -- (-0.5,0) node[opacity=1,rotate=60]{} -- (-0.5,2.2);
\draw[line width=0.4mm] (-0.47,-0.1) -- (-0.53,0.1);
\draw[line width=0.4mm] (-0.4,0.02) -- (-0.6,-0.02);
\draw[dashed,fill=black,fill opacity=0.2] (2.2,-0.5) -- (0,-0.5) node[opacity=1,rotate=30]{} -- (-1.7,-2.2);
\draw[line width=0.4mm] (-0.02,-0.6) -- (0.02,-0.4);
\draw[line width=0.4mm] (-0.1,-0.47) -- (0.1,-0.53);
\end{tikzpicture}
\caption{In each local halfplane outside an edge, a shear identifies the two boundaries of the excised shaded region, while preserving the edge.}
\label{fig:P2-cones}
\end{figure}

Finally, we could have chosen the identity transformation to be on the \emph{outside}
of the triangle.  To draw this on the single plane, we need to unwrap the triangle. 
A fundamental domain for the interior of the triangle will be the complement of a region which meets the singularities in wedges which determine the affine
transformation, as above.  The picture is drawn in the universal ($\bZ$-fold) cover of the complement of the extracted region.  Choosing the top region to define the common plane, we see from the affine transformation that all unbounded rays become vertical.  At this point, we rewrap the triangle by taking
a fundamental domain for the $\bZ$ action to arrive at
the picture displayed in Figure~\ref{fig:brokencorr-0}.

\begin{figure}[ht]
\centering
\begin{tikzpicture}[scale=1.1,rotate=90]
\draw[thick] (-3,-1) -- (0,0) -- (0,1) -- (-3,2);
\draw[dashed,thick] (-3,-1) -- (2,-1);
\draw[thick] (0,0) -- (2,0);
\draw[thick] (0,1) -- (2,1);
\draw[dashed,thick] (-3,2) -- (2,2);
\draw[dashed,fill=black,fill opacity=0.2,thick] (-3,-0.8) -- (-1.5,-0.5) node[opacity=1,rotate=20]{} -- (-1,0) -- (0,0.5) node[opacity=1,rotate=0]{} -- (-1,1) -- (-1.5,1.5) node[opacity=1,rotate=70]{} -- (-3,1.8);
\draw[line width=0.4mm] (0.08,0.58) -- (-0.08,0.42);
\draw[line width=0.4mm] (-0.08,0.58) -- (0.08,0.42);
\draw[line width=0.4mm] (-1.55,1.58) -- (-1.45,1.42);
\draw[line width=0.4mm] (-1.58,1.46) -- (-1.42,1.54);
\draw[line width=0.4mm] (-1.55,-0.58) -- (-1.45,-0.42);
\draw[line width=0.4mm] (-1.58,-0.46) -- (-1.42,-0.54);
\draw[thick] (0,0.5) -- (0,2);
\draw[thick] (0,0.5) -- (0,-1);
\end{tikzpicture}
\caption{Another fundamental domain, with cut regions in the central triangle.  The central singularity is at the origin, while the other two are at $(\pm1,- \frac{3}{2})$ on lines of slope $\pm 3$.  The $\bZ$ quotient identifying outer edges is effected by the affine transformation $(x,y)$ $\mapsto$ $(x+3,y-9x-\frac{27}{2})$.}
\label{fig:brokencorr-0}
\end{figure}
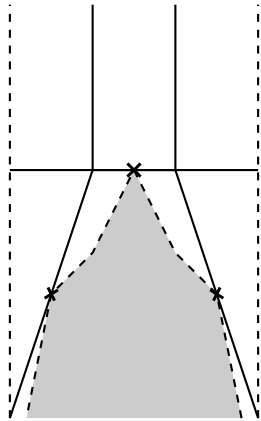

We have defined the affine manifold $B$ that underlies the mirror dual of $\PP^2$ and we have also given a decomposition into lattice polyhedra. In order to have the entire triple $(B,\P,\varphi)$, it remains to define the piecewise affine convex function $\varphi\colon B\to\RR$. The function is going to be integral affine on each of the four polyhedra and it is constant of value zero on the central triangle $T$. The function is then uniquely determined by requiring it to take value $1$ at the first integral point along each unbounded ray outside the central triangle $T$. Alternatively, instead of prescribing the values along rays, we can specify $\varphi$ by insisting that it changes integral affine slope by the value $1$ at each of the edges of the central triangle. This slope change is called \emph{kink}.

\subsection{Review of the construction of $W\colon \widecheck{\frak{X}}\ra\AA^1$ from $(B,\P,\varphi)$}
Let $A$ be some base ring and $B$ be an integral affine manifold with finite polyhedral decomposition $\P$ and a multi-valued strictly convex piecewise linear function $\varphi$, so that the triple $(B,\P,\varphi)$ forms a polarized integral affine manifold with simple singularities as defined in \cite{GS-dataI}, Definition 1.60. 
The key examples of interest to us are those which arise from the 16 toric Fano surfaces defined by reflexive polygons --- see, e.g., Figure~1 of \cite{CKYZ} or Figure~1.5 of \cite{Gra1}.
We first discuss the general situation. 
We are interested in situations where $B$ is non-compact and we require that the vertices of the singular locus are contained in the union of compact cells of $\P$.
Let $\iota:B_\reg\ra B$ denote the inclusion of the complement of the singular locus and $\Lambda\subset \shT_{B_\reg}$ the local system of integral tangent vectors on $B_\reg$.
In particular, for each vertex $v$ in $\P$, the stalk $\Lambda_v$ is a free abelian group of rank equal to the dimension of $B$ because the vertices are contained in $B_\reg$ by the definition of simple singularities. 

As in Appendix A of \cite{CPS}, we pick a set of slab functions $f_{\rho,v}\in A[\Lambda_v]$, one function for each pair of codimension-one-cell $\rho\in\P$ and vertex $v$ of $\rho$.
The slab functions have the property that the Newton polytope of 
$f_{\rho,v}$ coincides with the inner monodromy polytope of $\rho$ for the local monodromy group with base point $v$ --- see p.117 in \cite{R} or Definition 1.58 in \cite{GS-dataI}.
In particular, if the interior of $\rho$ does not meet any singularities of $B$ then $f_{\rho,v}$ is a unit.
Furthermore, if $v$ and $v'$ are two vertices of a codimension-one cell $\rho,$ then
after identifying $\Lambda_{v'}$ with $\Lambda_{v}$ by parallel transport, $f_{\rho,v}$
and $f_{\rho,v'}$ differ only by a monomial factor $az^m$ with $a\in A^\times$ a unit and $m\in\Lambda_v$. 
Finally, the collection $\{f_{\rho,v}\}_{v,\rho}$ satisfies the multiplicative condition (3.2) in \cite{GS-dataI}.

A slab function $f_{\rho,v}$ is called \emph{normalized} if its constant coefficient equals $1$ and its Newton polytope is in the correct position, that is, its Newton polytope is contained in the cone generated by $\rho-v$ --- see Definition~4.23 in \cite{GS-dataI}. 
The collection of slab functions together with the codimension one cells of $\P$ and the kinks of $\varphi$ gives the initial wall structure $\cS_0$ that is used for the algorithmic production of wall structures $\cS_k$ for $k\ge 0$.

Recall from \cite{GS-reconstruction}, Definition 2.22, that a \emph{wall structure $\cS_k$ to order $k$} is a locally finite polyhedral subdivision of $\P$ whose codimension-one-cells are called \emph{walls} and each wall $\frak{p}$ comes with a \emph{wall-function} $f_\frak{p}\in A[\Lambda_\frak{p}]\lfor t\rfor$.
Here, $\Lambda_\frak{p}$ denotes the subgroup of $\Gamma(\frak{p},\Lambda)$ given by integral vectors that are contained in the tangent space of $\frak{p}$.
A wall that is contained in a codimension-one-cell $\rho$ of $\P$ is called a \emph{slab}. 
Other than usual walls, slabs carry not just a single wall-function but one wall function $f_{\frak{p},v}$ for each vertex $v$ of $\rho$ and the reductions of $f_{\frak{p},v}$ modulo $t$ 
agree with the slab function $f_{\rho,v}$. For two vertices $v,v'$ of $\rho$, $f_{\frak{p},v}$ and $f_{\frak{p},v'}$ relate to each other by the same rule as $f_{\rho,v}$ and $f_{\rho,v'}$. 

The wall functions are used to glue charts of the form $\Spec A[\Lambda_p][t]/(t^k)$ for $p$ a point on either side of the wall via identifications of the form given in Appendix~\ref{app:wallcrossing}. When $\dim B=2$, the transformation at a non-slab wall $\frak{p}$ takes the form $y\mapsto y\cdot (1+at^kx)$ for $f_\frak{p}=(1+at^kx)$ and $x,y$ suitable monomials that generate a stalk of $\Lambda$ with $x\in\CC[\Lambda_\frak{p}]$, $a\in A^\times$, $k>0$.
When the wall is a slab, the function $\varphi$ dictates a $t$-power $\kappa$ that enters the chart transition on top of the usual transformation by the wall function, so crossing a slab takes the form $y\mapsto y\cdot t^\kappa\cdot f_{\frak{p},v}$ where $f_{\frak{p},v}=(1+ax)+tg$ for some $g\in A[\Lambda_\frak{p}]\lfor t \rfor$. We should remark that walls may lie on top of each other and crossing one of them means to cross them all. For slabs it is customary to combine all such walls into a single slab and single function by taking their product. This is how the term $tg$ in the given transformation should be understood.

The \emph{reduction} of a wall structure $\cS_k$ modulo $t^l$ for $l<k$ is the result of reducing all wall functions modulo $t^l$.
The structure is \emph{consistent} if the path-ordered composition of the corresponding wall crossing automorphisms associated with the walls that are being crossed gives the identity for all circular loops that have finite intersection with the union of walls.
The algorithm from \cite{GS-reconstruction} proves the existence of a canonical consistent structure $\cS_k$ for every $k$ so that $\cS_0$ is given by the set of slabs together with the chosen slab functions. These wall structures are \emph{compatible} in the sense that for any $l<k$ the reduction of $\cS_k$ modulo $t^l$ results in a subdivision of $\cS_l$ where every additional wall has wall function $1$.

From now on, we denote by $\{\cS_k\}$ the set of canonical, consistent and compatible wall structures that result from $(B,\P,\varphi)$ together with the universal set of slab functions by following the algorithm of \cite{GS-reconstruction}. The main point of \cite{GS-reconstruction} is to use $\cS_k$ in order to glue a scheme $\widecheck{X}_k$ over $A[t]/t^k$ for each $k$. Compatibility implies the existence of natural projection maps $\widecheck{X}_k\ra \widecheck{X}_l$ over 
$A[t]/t^k\ra A[t]/t^l$ whenever $l<k$, so the inverse limit 
$$\widecheck{\frak{X}}=\lim_{k\to\infty} \widecheck{X}_k$$ 
is a formal scheme over the formal spectrum of $A\lfor t\rfor$.
We say $(B,\P,\varphi)$ is the \emph{cone picture} of the degeneration $\widecheck{\frak{X}}\ra \Spf A\lfor t\rfor$.
We use the notation $\cS_\infty$ to refer to the inverse system of compatible wall structures $\{\cS_k\}$.

An integral tangent vector of a one-dimensional unbounded cell in $(B,\mathscr{P},\varphi)$ that points in the unbounded direction is called an \emph{asymptotic direction}. 
By parallel transport, we consider these asymptotic directions also in the neighbouring maximal cells of the ray that defines them.
We recall Definition 4.9 from \cite{Gr10}, see also Definition 3.3 in \cite{GHS}.

\begin{definition}
\label{defi:brokenline}
A \textit{broken line} $\mathfrak{b}$ is a proper continuous map $\frak{b}\colon[0,\infty)\ra B$ 
with finitely many break points $-\infty<t_1<...<t_r<0$ that map into walls $\frak{p_i}$ of $\mathscr{S}_k$ for some fixed $k$ so that $\frak{b}$ is affine linear in the complement $[0,\infty)\setminus\{t_i\}$ of the these points.
Each component $I$ of this complement is decorated with a monomial $a t^{d} z^m \in A[\Lambda_I][t]$ with $a\in A^\times$, $d\ge 0$ and $m\in\Lambda_I:=\Gamma(\frak{b}(I),\Lambda)$ a non-zero integral tangent vector on $I$ that is tangent to $\frak{b}(I)$ --- precisely, $\frak{b}'|_I\equiv -m$.
The monomials of successive line segments are required to be related by \emph{monomial transport} past the wall $\frak{p_i}$ which contains the image $\frak{b}(t_i)$ of the break point that separates the line segments. 
If $a t^{d} z^m , b t^{e} z^{m'}$ are the monomials of successive segments separated by $t_i$, the \emph{monomial transport} property says that $b t^{e} z^{m'}$ is a monomial summand in the expansion of the product
$(a t^{d} z^m)\cdot f^k_\frak{p}$ where the positive integer $k$ is the index of $m$ in the rank one quotient $\Lambda_{\frak{b}(t_i)}/\Lambda_\frak{p_i}\cong\ZZ$.
Finally, it is required that the unique unbounded segment $(-\infty,t_1)$ of $\mathfrak{b}$ carries the monomial $z^m$ with $m$ an asymptotic direction on $B$. 
The image of $0\in[0,\infty)$ is a point $P\in B$ that we call \emph{endpoint} and we say the broken line \emph{ends} in $P$. 
\end{definition}

A broken line is \textit{primitive} if the unbounded segment carries a monomial $z^m$ with $m$ primitive.
Note that in \cite{CPS}, Definition 4.2, all broken lines are assumed to be primitive. 
Write $a_{\mathfrak{b}}t^{d_{\mathfrak{b}}}z^{m_{\mathfrak{b}}}$ for the \emph{ending monomial}, i.e., the monomial attached to the last segment of $\mathfrak{b}$ which ends in $P$. 
The $t$-exponent $d_{\mathfrak{b}}$ is called the $t$-\emph{order} or just \emph{order} of the broken line $\mathfrak{b}$.
Since monomial transport may increase but never lowers the $t$-exponent of a monomial, all the walls $\frak{p}$ at which a broken line of order $k$ breaks are contained in $\mathscr{S}_{d_{\mathfrak{b}}}$. The number of broken lines of fixed order is finite by Lemma 3.7 in \cite{GHS}.

\begin{definition}
\label{defi:theta}
For a chamber $\mathfrak{u}$ of the wall structure $\cS_k$ and $P\in\mathfrak{u}$ a general point let $\mathfrak{B}^{(k)}_m(P\in\mathfrak{u})$ be the set of broken lines of $t$-order $k$ ending in $P$ with asymptotic monomial $z^m$. When using the notation $P\in\mathfrak{u}$ in the following, we implicitly require that $P$ is in general position.
If $\mathfrak{B}_m(P\in\mathfrak{u}):=\bigcup_{l\ge 0} \mathfrak{B}^{(l)}_m(P\in\mathfrak{u})$ denotes the set of broken lines of any order with respect to $\cS_k$, then in fact
$$\mathfrak{B}_m(P\in\mathfrak{u})=\bigcup_{l=0,...,k} \mathfrak{B}^{(l)}_m(P\in\mathfrak{u}).$$ 
By \cite{CPS}, Lemma 4.7, the sum
\[ \vt^{(k)}_m(\mathfrak{u}) = \sum_{\mathfrak{b}\in\mathfrak{B}^{(k)}_m(P\in\mathfrak{u})} a_{\mathfrak{b}}t^{d_{\mathfrak{b}}}z^{m_{\mathfrak{b}}}. \]
is independent of where $P$ lies inside $\mathfrak{u}$, justifying the notation. 
For a nested sequence $(\mathfrak{u}_k)_{k\in \NN}$ of containments with $\mathfrak{u}_k$ a chamber of $\cS_k$, one defines the theta function
\[ \vt_m((\mathfrak{u}_k)_{k\in \NN}) = \sum_{k\ge 0}  \vt^{(k)}_m(\mathfrak{u}_k) =  \sum_{k\ge 0}\sum_{\mathfrak{b}\in\mathfrak{B}^{(k)}_m(P\in\mathfrak{u}_k)} a_{\mathfrak{b}}t^{d_{\mathfrak{b}}}z^{m_{\mathfrak{b}}}. \]
The \textit{superpotential} $W((\mathfrak{u}_k)_{k\in \NN})$ is the sum over all theta functions $\vt_m((\mathfrak{u}_k)_{k\in \NN})$ with $m$ primitive,
\[ W((\mathfrak{u}_k)_{k\in \NN}) = \sum_m \vt_m((\mathfrak{u}_k)_{k\in \NN}) = \sum_{m,k}\sum_{\mathfrak{b}\in\mathfrak{B}^{(k)}_m(P\in \mathfrak{u}_k)} a_{\mathfrak{b}}t^{d_{\mathfrak{b}}}z^{m_{\mathfrak{b}}}. \]
The sums run over all primitive asymptotic directions $m$ of $B$.
\end{definition}

By the main result of \cite{CPS, GHS}, the collection of $\vt_m((\mathfrak{u}_k)_{k\in \NN})$ and in particular $W((\mathfrak{u}_k)_{k\in \NN})$ give well-defined global regular functions on $\widecheck{\frak{X}}$, so they give elements $W,\vt_m\in\Gamma(\widecheck{\frak{X}},\shO_{\widecheck{\frak{X}}})$. 
It will be relevant for the multiplication rule in Section~\ref{section-theta} that we include the function $\vt_0=1$.

\subsection{Properness of the Landau--Ginzburg potential}
We write $W^0$ for the $t$-order zero part of $W$.

\begin{definition}
\label{defi:cyl}
We say that $(B,\P)$ is \emph{asymptotically cylindrical} if $B$ is non-compact and every polyhedron $\sigma$ in $\P$ has the property that all of its unbounded one-faces are parallel with respect to the affine structure on $\sigma$. 
\end{definition}

For $a\in\Spec A$ a point, we denote the base change of $\widecheck{\frak{X}}$ to 
$a\times\Spf \CC\lfor t\rfor$ by $\widecheck{\frak{X}}_a$ and we similarly define $(\widecheck{X}_k)_a$. 
We next recall Proposition 2.1 in \cite{CPS}.

\begin{proposition} 
$W^0\colon (\widecheck{X}_0)_a \ra \AA^1$ is a proper morphism for every $a\in \Spec A$ with $(\widecheck{X}_0)_a\neq\varnothing$ if and only if $(B,\P)$ is asymptotically cylindrical.
\end{proposition}

Now assume that $(B,\mathscr{P})$ is asymptotically cylindrical, without boundary and that there is only one unbounded direction $m_{\text{out}}$, unlike say for the fan of $\PP^1$ which has two such directions.
Then theta functions are labeled by integers, $\vt_q := \vt_{qm_{\text{out}}}$. 
The superpotential is given by the primitive theta function $\vt_1$ and is also called the \emph{proper Landau--Ginzburg superpotential}. 
All other theta functions are polynomials in $\vt_1$, so the proper Landau--Ginzburg superpotential generates the ring of theta functions. 

Let $K$ be the union of compact polytopes in $B$.
The asymptotic boundary $B_\infty$ of $B$ is the integral affine manifold with singularities and polyhedral decomposition given as the quotient $(B\setminus K)/m_{\text{out}}$. It follows directly that $B_\infty$ is compact and closed.

We learned the statement of the following proposition from Bernd Siebert.

\begin{proposition}
\label{prop:parallel}
Let $(B,\P)$ be asymptotically cylindrical and let $\mathfrak{u}$ be an unbounded chamber of $\mathscr{S}_k$. 
If $\mathfrak{b}\in\mathfrak{B}_m^{(k)}(P\in\mathfrak{u})$, then $m_{\mathfrak{b}}$ is parallel to $m_{\text{out}}$.
\end{proposition}

\begin{proof} 
The number of walls in $\mathscr{S}_k$ is finite. 
We claim that every unbounded wall is necessarily parallel to $m_{\text{out}}$, i.e., it projects to a compactly supported codimension one subset of $B_\infty$. 
This is clearly true for unbounded slabs. 
If we assume to the contrary to have a non-slab wall $\frak{p}$ not containing $m_{\text{out}}$, then by compactness of $B_\infty$, the projection of $\frak{p}\setminus K$ to $B_\infty$ is an infinite cover of its image. 
Every ray contained in $\frak{p}\setminus K$ therefore has infinitely many crossings with unbounded slabs in $B$.
Since by the definition of a wall, the wall function has a monomial that generates such a ray, this monomial picks up an arbitrarily high $t$-power under these slab crossings which is a contradiction to 
the wall being contained in $\mathscr{S}_k$. 

A similar argument now excludes a broken line $\frak{b}$ of order $\le k$ ending at $P$ which does not have its final segment parallel to $m_{\text{out}}$.
Let $\widetilde K$ denote the union of bounded cells in $\mathscr{S}_k$.
Since $\vt^{(k)}_m(\mathfrak{u})$ is independent of the endpoint $P$, so we can assume $P$ is arbitrarily far away from $\widetilde K$. 
Now, if $\frak{b}$ is a broken line with endpoint $P$ and monomial $m_\frak{b}$ not an integer multiple of $m_{\text{out}}$, then the projection of $m_\frak{b}$ to $B_\infty$ is non-trivial and it generates a ray. 
By moving $P$ arbitrarily far away from $\widetilde K$, we can force $\frak{b}$ to traverse arbitrarily many slabs in $B$: indeed, a slab crossing or wall crossing in $B\setminus \widetilde K$ may modify $m_\frak{b}$ to some other monomial but never to a multiple of $m_{\text{out}}$ because, as we found before, $m_{\text{out}}$ is parallel to every unbounded slab and wall. 
So the projection of the new monomials to $B_\infty$ remain non-trivial. 
Each slab crossing however increases the $t$-order of the broken line contradicting that the broken line is of finite $t$-order $\le k$. 
\end{proof}

\begin{corollary}
If $(B,\P)$ is asymptotically cylindrical and $\mathfrak{u}$ an unbounded chamber of $\mathscr{S}_k$, denoting $y=z^{m_{\text{out}}}$, we have
$\vt^{(k)}_{qm_{\text{out}}}(\mathfrak{u})$ for any $q>0$ as well as
 $W^{\le k}(\mathfrak{u}):= \sum_{i=0}^{k} \vt^{(i)}_{m_{\text{out}}}(\mathfrak{u}) \in y A[y^{-1}][t]$ are independent of the choice of $\mathfrak{u}$.
Consequently, $W((\mathfrak{u}_k)_{k\in \NN})$ and $\vt_{qm_{\text{out}}}((\mathfrak{u}_k)_{k\in \NN})$ are independent of the choice of nested sequence $(\mathfrak{u}_k)_{k\in \NN}$ of unbounded chambers.
\end{corollary}

\subsection{Perfect pairing with curve classes} 
\label{subsec-pairing}
The purpose of this section is a specify a particular base ring to use for $A$ so that the family of Landau--Ginzburg models $W\colon\widecheck{\frak{X}}\ra\AA^1$ over $\Spf A\lfor t\rfor$ is versal in a suitable sense. We follow \cite{MR21}.
We keep the assumption of $(B,\mathscr{P})$ being asymptotically cylindrical with a single primitive asymptotic direction $m_{\text{out}}$. 
We denote by $H_1(B,\iota_\star\Lambda)$ the first \emph{sheaf homology} with coefficients in $\iota_\star\Lambda$. Set $\widecheck{\Lambda}:=\Hom({\Lambda},\ZZ)$. 
By \cite{R2}, Theorem 3, there is a natural pairing
$H_1(B,\iota_\star\Lambda) \otimes H^1(B,\iota_\star\widecheck{\Lambda}) \rightarrow \mathbb{Z},$
which is perfect over $\QQ$. In other words, we have a natural homomorphism
\begin{equation}
\mathsf{D}\colon H_1(B,\iota_\star\Lambda)\ra H^1(B,\iota_\star\widecheck\Lambda)^*
\label{eq-map-D}
\end{equation}
 that becomes an isomorphism when tensoring with $\QQ$. 

In \cite{MR21}, a natural compactification $\overline{B}$ of $(B,\P)$ with the feature $\overline{B}\setminus B=B_\infty$ is introduced. 
It comes with an open subset
$\overline{B}_\reg\subset \overline{B}$ so that $B_\reg=\overline{B}_\reg\cap B$. 
There are two different natural extensions of $\Lambda$ from $B_\reg$ to $\overline{B}_\reg$.
One extension is simply the direct image sheaf under the open embedding $B_\reg\hra \overline{B}_\reg$ which is also denoted $\Lambda$. 
The other extension is the subsheaf $\Lambda_N\subset\Lambda$ of this extension given by the property that the stalk of $\Lambda_N$ at a point in $\overline{B}\setminus B$ is generated by the asymptotic directions. The notation is borrowed from \cite{Ba} where a similar analysis is done for the discrete Legendre dual where $B$ is already compact.
In our case of a unique asymptotic direction $m_{\text{out}}$, the stalk of $\Lambda_N$ at every point in $\overline{B}\setminus B$ is given by $\ZZ m_{\text{out}}$.
We also use $\iota\colon \overline{B}_\reg\subset \overline{B}$ for the open embedding of the domain of definition of $\Lambda$ on $\overline{B}$ which is justified by the existence of canonical restriction isomorphisms
$H^i(\overline B,\iota_\star\Lambda)\ra H^i(B,\iota_\star\Lambda)$. 
By our assumption about a unique asymptotic direction, $B_\infty=\overline B\setminus B$ is itself an integral affine manifold with singularities and polyhedral decomposition. 
In this case, the quotient sheaf 
$\Lambda_\infty:=\Lambda/\Lambda_N$ is supported on $ B_\infty$ and naturally identified with the integral tangent sheaf on $B_{\infty,\reg} = \overline{B}_\reg\cap B_\infty$. 

We also consider two extensions of $\widecheck{\Lambda}$ from $B_\reg$ to $\overline{B}_\reg$. 
The first extension is the direct image sheaf  which we also denote by $\widecheck{\Lambda}$. 
The second is the subsheaf of the first given by those dual vectors that annihilate $m_{\text{out}}$ and we denote this extension by $\widecheck{\Lambda}_N$. Despite the notation, it is not true that $\widecheck{\Lambda}_N$ is the dual of $\Lambda_N$.
The quotient sheaf $\widecheck{\Lambda}/\widecheck{\Lambda}_N$ is supported on $B_{\infty,\reg}$ with stalks naturally isomorphic to $(\ZZ m_{\text{out}})^*$ and we denote this quotient sheaf and its direct image on $ B_\infty$ by $(\ZZ m_{\text{out}})^*$ as well.

Let
$\iota\colon  B_{\infty,\reg}\hra B_\infty$ denote the inclusion.
We consider the following two long exact sequences of cohomology groups 
\begin{equation}
\cdots\ra H^0( B_\infty, \iota_\star\Lambda_N)\ra  H^1(\overline{B}, \iota_\star\Lambda_N) \ra H^1(\overline{B},B_\infty; \iota_\star\Lambda_N) \ra H^1(B_\infty, \Lambda_N) \ra\cdots \label{les-1}
\end{equation}
\begin{equation}
\cdots\ra H^0(B_\infty, (\ZZ m_{\text{out}})^*)\ra  H^1(\overline{B}, \iota_\star\widecheck{\Lambda}_N) \ra H^1(B, \iota_\star\widecheck{\Lambda}) \ra H^1(B_\infty, (\ZZ m_{\text{out}})^*) \ra \cdots \label{les-2}
\end{equation}
Note that $H^0(B_\infty, (\ZZ m_{\text{out}})^*)\cong \ZZ$ and let $n_{\text{out}}$ denote the generator that evaluates to $1$ on $m_{\text{out}}$.
The image of $n_{\text{out}}$ under the coboundary homomorphism gives a distinguished class 
\begin{equation}
\label{eq-dist-class}
\beta_{B_\infty}\in
H^1(\overline{B}, \iota_\star\widecheck{\Lambda}_N)
\end{equation}
that we will come back to in the next section. 
We are interested in the Poincar\'e--Lefschetz dual sequence of \eqref{les-1} that we obtain from Theorem~1 in \cite{R2}. It is shown in \cite{MR21}, that this Poincar\'e--Lefschetz dual of \eqref{les-1} permits a natural map to the dual of the sequence \eqref{les-2} that is shown in the following diagram. The displayed vertical maps become isomorphisms when tensoring with $\QQ$,
\begin{equation}
\resizebox{\textwidth}{!}{%
\xymatrix{
\cdots \ra H_1(B_\infty, \Lambda_N) \ar[r]\ar^{\hbox{\rotatebox[origin=c]{90}{$\sim$}}}[d]&  
H_1( B, \iota_\star\Lambda) \ar[r]\ar[d]^{\mathsf{D}}& 
H_1(\overline{B},B_\infty; \iota_\star\Lambda_N) \ar[r]\ar[d]^{\mathsf{D}_N}&  
H_0(B_\infty, \Lambda_N) \ar^{\hbox{\rotatebox[origin=c]{90}{$\sim$}}}[d]\ra\cdots \\
\cdots\ra
H^1(B_\infty, (\ZZ m_{\text{out}})^*)^* \ar[r]&
H^1(B, \iota_\star\widecheck{\Lambda})^* \ar[r]&
H^1(\overline{B}, \iota_\star\widecheck{\Lambda}_N)^* \ar[r]&
H^0(B_\infty, (\ZZ m_{\text{out}})^*)^* \ra  
\cdots
}
}
\label{eq-pairing-diagram}
\end{equation}
We define the ring
$$A_N:=\CC[H^1(\overline{B}, \iota_\star\widecheck{\Lambda}_N)^*] = \left\{ \sum^\text{finite}_\beta a_\beta s^\beta|a_\beta\in \CC, \beta\in H^1(\overline{B},\iota_\star\widecheck{\Lambda}_N)^*\right\}$$
which serves as a natural base and coefficient ring $A$ for the family of Landau--Ginzburg models $\widecheck{\frak{X}}$ as used in the previous section, see also \cite{MR21}.
For the sake of the form of slab functions with coefficients in $A_N$, please note that
the units in $A_N$ are given by $$A_N^\times =\CC^*\oplus H^1(\overline{B},\iota_\star\widecheck{\Lambda}_N)^* =\{ a s^\beta|a\in \CC^*, \beta\in H^1(\overline{B},\iota_\star\widecheck{\Lambda}_N)^*\}.$$
\begin{remark}
We remark here that $H^1(\overline{B}, \iota_\star\widecheck{\Lambda}_N)\otimes_\ZZ \CC$ can be naturally identified with the first cohomology of the sheaf 
$\Theta_{\overline{\widecheck{X}_0}(\log D)/\Spec A_N}(-D)$ of relative log derivations if $D=\overline{\widecheck{X}_0}\setminus \widecheck{X}_0$ is any toroidal crossing divisor that compactifies the fibers of $X_0\ra \Spec A$. 
The group $H^1(\overline{\widecheck{X}_0},\Theta_{\overline{\widecheck{X}_0}(\log D)/\Spec A_N}(-D))$ classifies infinitesimal log deformations of $\overline{\widecheck{X}_0}$ over $A_N$ that induce the identity on $D$. 
In other situations where $\widecheck{X}_0$ is already proper, so if $B$ is compact, 
the ring $\CC[H^1(B,\iota_\star\widecheck\Lambda)^*]$ was used in \cite{GHS,RS} in place of $A_N$ as a universal coefficient ring and versality of ${\widecheck{\frak{X}}}\ra\Spf \CC[H^1(B,\iota_\star\widecheck\Lambda)^*]\lfor t\rfor$ was shown in \cite{RS}.
\end{remark}

Since $\varphi$ is locally determined only up to adding a linear function, we may assume that it annihilates $m_{\text{out}}$ and therefore gives a class in
$H^1(\overline{B}, \iota_\star\widecheck{\Lambda}_N)$. 
By pairing with this class we define the degree homomorphism
\begin{equation}
\label{eq-map-deg}
\deg\colon H^1(\overline{B}, \iota_\star\widecheck{\Lambda}_N)^* \ra \ZZ.
\end{equation}

The relevance of the homomorphism $\mathsf{D}_N\colon H_1(B,B_\infty;\iota_\star\Lambda_N)\ra H^1(\overline{B},\iota_\star\widecheck{\Lambda}_N)^*$ for the Landau--Ginzburg potential $W$ will become clear in Section~\ref{sec:tropicalcorrespondence}.

\subsection{A Landau-Ginzburg model as a mirror dual to a Fano variety}
\label{subsec-Fanosurface}
In \S\ref{S:toricdeg}, we discuss the situation where $(B,\P,\varphi)$ is the dual intersection complex, i.e., the \emph{fan picture}, of a toric degeneration $\mathfrak{X}$ of a log Calabi--Yau pair $(X,D)$. In general for a fan picture of $(X,D)$,
the polyhedral decomposition $\P$ contains unbounded cells if and only if $D$ is non-empty.
The existence of unbounded cells in turn is necessary for $W$ to be non-constant when taking $(B,\P,\varphi)$ as the cone picture of a Landau--Ginzburg model as we did in the previous sections. 
We call $(\widecheck{\frak{X}},W)$ the Landau--Ginzburg mirror dual to $(X,D)$.
An easy example is given by a toric Fano variety $X$ with toric anticanonical divisor $D$ where the degeneration $\mathfrak{X}$ is just a trivial family of the pair $(X,D)$ and then $B=\RR^n$, $\P$ is the set of cones in the fan of $X$ and $\varphi$ a piecewise linear function that gives a polarization for $X$. 
It is also possible that both $\mathfrak{X}$ and $\widecheck{\mathfrak{X}}$ carry non-constant Landau--Ginzburg potentials, e.g., with mirror symmetry for varieties of general type \cite{GKR}. Unlike for toric varieties, we will be predominantly interested in situations where $X$ is a compact Fano manifold and $D$ is a smooth irreducible anticanonical divisor. 
In the situation where $B$ is homeomorphic to a disk, it follows from \cite{RZ} that there is in fact a natural isomorphism
\begin{equation}
H^1(\overline{B}, \iota_\star\widecheck{\Lambda}_N) \cong \Pic(X)
\end{equation}
and thus the base ring $A_N$ of the mirror family $\widecheck{\frak{X}}$ from the section~\ref{subsec-pairing} is canonically identified with $\CC[\Pic(X)^*]$.

\subsection{$\boldsymbol{q}$-refined wall structures}
\label{subsec-qwalls}
While the main result of this articles is about genus zero invariants, the comparison of log invariants under a blow-up can be done more generally for arbitrary genus by the same methods. Since we are going to use this higher-genus correspondence in the upcoming article \cite{GRZZ}, we cover it here.
To keep track of the genus, we introduce the variables $\boldsymbol{q}$ and $\hbar$, which are related by $\boldsymbol{q}=e^{i\hbar}$.
We define the $\boldsymbol{q}$-refined initial wall structure $\mathscr{S}_0(\boldsymbol{q})$ to have the same slabs as $\mathscr{S}_0$ but with their normalized slab functions $f_{\rho,v}=1+a s^\beta z^m$ replaced by the power series $f_{\rho,v}(\boldsymbol{q})\in A\llbracket \hbar\rrbracket[\Lambda_{\mathfrak{p}}]$ which satisfies
\[ \textup{log }f_{\rho,v}(\boldsymbol{q}) = \sum_{k\geq 1} \frac{(-1)^{k+1}i\hbar}{\boldsymbol{q}^{k/2}-\boldsymbol{q}^{-k/2}} a^k s^{k\beta}z^{km}. \]

\begin{remark}

The $\boldsymbol{q}$-refinement will correspond to higher genus Gromov-Witten invariants with insertions of $\lambda$-classes (see \S\ref{sec:tropicalcorrespondence}). More precisely, the genus $g$ invariant will be the $\hbar^{2g}$-term. The limit $\hbar \rightarrow 0$ or, equivalently, $\boldsymbol{q} \rightarrow 1$ reproduces the genus $0$ case. Indeed, note that $\boldsymbol{q}^{k/2}-\boldsymbol{q}^{-k/2}=ik\hbar + \mathcal{O}(\hbar^2)$, so in the limit $\hbar\rightarrow 0$ the above expression for $\textup{log }f_{\rho,v}(\boldsymbol{q})$ becomes the Taylor expansion for the original $f_{\rho,v}=1+a s^\beta z^m$.
\end{remark}

The inductive definition of wall structures and the associated broken lines, as explained at the beginning of \S\ref{section-LG}, can be performed in $A\llbracket\hbar\rrbracket[\Lambda]\llbracket t\rrbracket $ instead of $A[\Lambda]\llbracket t\rrbracket $. 
The formulas are the same, the only difference being that all functions carry an additional $\hbar$-variable.
Hence, for each finite order $k>0$ we obtain a $\textbf{q}$-refined wall structure $\mathscr{S}_k(\boldsymbol{q})$.
Let $\mathscr{S}_\infty(\boldsymbol{q})$ denote the resulting inverse system $\{\mathscr{S}_k(\boldsymbol{q})\}$ of compatible consistent \emph{$\boldsymbol{q}$-refined} wall structures 
obtained from $\mathscr{S}_0(\boldsymbol{q})$. 

As in Definition \ref{defi:theta} we define theta function as sums over broken lines $\mathfrak{b}$ with initial monomial $z^m$. When $\mathfrak{b}$ crosses a wall of the $\mathbf{q}$-refined wall structure $\mathscr{S}_k(\boldsymbol{q})$, it becomes dependent on $\mathbf{q}$. So also $a_{\mathfrak{b}}$, the coefficient of its ending monomial, will depend on $\mathbf{q}$ and we indicate this by writing $a_{\mathfrak{b}}(\mathbf{q})$.

\begin{definition}
\label{defi:thetaq}
For each chamber $\mathfrak{u}\in \mathscr{S}_k(\boldsymbol{q})$ and asymptotic direction $m$ one defines 
\[ \vt^{(k)}_m(\mathfrak{u},\boldsymbol{q}) = \sum_{\mathfrak{b}\in\mathfrak{B}^{(k)}_m(P\in\mathfrak{u})} a_{\mathfrak{b}}(\boldsymbol{q})t^{d_{\mathfrak{b}}}z^{m_{\mathfrak{b}}}. \]
and then for a nested sequence of chambers $(\mathfrak{u}_k)_{k\in \NN}$ one defines the \emph{$\boldsymbol{q}$-refined theta function}
$$\vt_m((\mathfrak{u}_k)_{k\in \NN},\boldsymbol{q})=\sum_{k\ge 0} \vt^{(k)}_m(\mathfrak{u}_k,\boldsymbol{q})$$
\end{definition}

\begin{definition}
\label{defi:Bpq}
If $(B,\P)$ is asymptotically cylindrical and $\mathfrak{u}$ an unbounded chamber of $\mathscr{S}_k$, let $\mathfrak{B}^{(k)}_{p,q}(P\in \mathfrak{u})$ denote the set of broken lines $\mathfrak{b}\in\mathfrak{B}^{(k)}_{q m_{\text{out}}}(X,P\in\mathfrak{u})$ with ending monomial of the form $a_{\mathfrak{b}} t^{k}y^{-p}$ where $y=z^{m_{\text{out}}}$.
\end{definition}

If $(B,\P)$ is asymptotically cylindrical, $\mathfrak{u}$ is an unbounded chamber of $\mathscr{S}_k(\boldsymbol{q})$ and $y=z^{m_{\text{out}}}$ then
Proposition~\ref{prop:parallel} implies
\begin{equation}
\label{eq-theta-parallel}
\vt^{(k)}_{q m_{\text{out}}}(\mathfrak{u},\boldsymbol{q}) = \sum_{p\ge 1}\sum_{\mathfrak{b}\in\mathfrak{B}_{p,q}^{(k)}(\mathfrak{u},\boldsymbol{q})} a_{\mathfrak{b}}(\boldsymbol{q})t^{k}y^{-p} \in A\llbracket \hbar,t,y\rrbracket. 
\end{equation}

\begin{corollary}
If $(B,\P)$ is asymptotically cylindrical, then the ring of $\boldsymbol{q}$-refined theta functions (\textit{quantum theta ring}) is commutative.
\end{corollary}

\begin{proof}
We show that for $k>0$ and $\mathfrak{u}$ a chamber of $\mathscr{S}_k$, the $\boldsymbol{q}$-refined theta functions $\vt_{qm_{\text{out}}}^{(k)}(\mathfrak{u},\boldsymbol{q})$ commute for different values of $q\in \mathbb{Z}_{>0}.$  

In fact, it suffices to establish this for $\mathfrak{u}$ an unbounded chamber, since theta functions in other chambers are related by conjugation by wall-crossing isomorphisms, which preserves commutativity. If $\mathfrak{u}$ is an unbounded chamber, then by \eqref{eq-theta-parallel} we have that $\vt_{qm_{\text{out}}}^{(k)}(\mathfrak{u},\boldsymbol{q}) \in A\llbracket \hbar,t,y\rrbracket$ only depends on the global variable $y$ of $A[\Lambda]$ and not on any other local variable $z^m$. While different variables $z^m$ and $z^{m'}$ in the quantum ring do not necessarily commute, the single variable $y$ trivially commutes with itself. Thus the $\vt_{qm_{\text{out}}}^{(k)}(\mathfrak{u},\boldsymbol{q})$ all commute with each other. 

Using results of \S\ref{section-theta} there is another proof: For $(B,\P)$ asymptotically cylindrical the theta ring is generated by the primitive theta function $\vt_{m_{\text{out}}}^{(k)}(\mathfrak{u},\boldsymbol{q})$. The other theta functions $\vt_{qm_{\text{out}}}^{(k)}(\mathfrak{u},\boldsymbol{q})$ can be expressed uniquely in terms of it via the multiplication rule in Proposition \ref{prop:multstr}. By Proposition \ref{prop:theta} the structure constants of the multiplication $\vt_{pm_{\text{out}}}^{(k)}(\mathfrak{u},\boldsymbol{q}) \cdot \vt_{qm_{\text{out}}}^{(k)}(\mathfrak{u},\boldsymbol{q})$ are invariant under swapping $p$ and $q$. Thus $\vt_{pm_{\text{out}}}^{(k)}(\mathfrak{u},\boldsymbol{q})$ and $\vt_{qm_{\text{out}}}^{(k)}(\mathfrak{u},\boldsymbol{q})$ commute for all $p,q$.
\end{proof}


\section{Definition of invariants and tropical correspondence}
\label{sec:tropicalcorrespondence}

\subsection{Log Gromov-Witten invariants}

For an effective curve class $\beta$ of a Fano surface $X$ with log structure given by the smooth anticanonical divisor $E$, let $\mathscr{M}_{g,1}(X,\beta)_{(d)}$ be the moduli stack of basic stable log  maps to $X$ of genus $g$, class $\beta$ and maximal tangency $d=\beta.E$ with $E$ at a single unspecified point. 
The virtual dimension is $g$ and we cut this dimension down to zero by inserting the \textit{lambda class} $\lambda_g$ which is defined as follows. 
Let $\pi : \mathcal{C} \rightarrow \mathscr{M}_{g,1}(X,\beta)_{(d)}$ be the universal curve with relative dualizing sheaf $\omega_\pi$. Then $\mathbb{E}=\pi_\star\omega_\pi$ is a rank $g$ vector bundle over $\mathscr{M}_{g,1}(X,\beta)_{(d)}$, known as the \emph{Hodge bundle}. The lambda classes are the Chern classes of the Hodge bundle, $\lambda_j=c_j(\mathbb{E})$. We define the $1$-marked log Gromov--Witten invariants
\[ R_d^g(X,\beta) := \int_{\llbracket\mathscr{M}_{g,1}(X,\beta)_{(d)}\rrbracket} (-1)^g\lambda_g \in \mathbb{Q}. \]
Similarly, for an effective curve class $\beta$ of $X$, let $\mathscr{M}_{g,2}(X,\beta)_{(p,q)}$ be the moduli stack of basic stable log maps to $X$ of genus $g$, class $\beta$ and two marked points with contact orders $p$ and $q$ with $E$. It has virtual dimension $g+1$. We cut this dimension down to zero by fixing the image of the first marked point on $E$ and inserting a lambda class. Let $\textup{ev} \colon \mathscr{M}_{g,2}(X,\beta)_{(p,q)} \rightarrow E$ be the evaluation map at the marked point of order $p$. 
We define the $2$-marked log Gromov--Witten invariant
\[ R_{p,q}^g(X,\beta) := \int_{\llbracket\mathscr{M}_{g,2}(X,\beta)_{(p,q)}\rrbracket} (-1)^g\lambda_g \textup{ev}^\star[\textup{pt}] \in \mathbb{Q}. \]
Note that $R_{p,q}^g(X,\beta)=0$ unless $p+q=\beta.E$.
If $q=0$, the corresponding marking is no longer constrained to $E$ and can be anywhere in $X$ in that case. The virtual class thus pulls back under the map that forgets this marking and the projection formula implies the invariant vanishes in this case, $R_{p,0}^g(X,\beta)=0$.
Similarly, for $p\leq 0$ we define $R_{p,q}^g(X,\beta)=0$.

\begin{definition}
\label{defi:Rq}
We write $R_d(X,\beta):=R_d^0(X,\beta)$ and define, with $\boldsymbol{q}=e^{i\hbar}$,
\[ R_d(X,\beta,\boldsymbol{q}) := \sum_{g\geq 0}R_d^g(X,\beta)\hbar^{2g}. \]
Similarly, write $R_{p,q}(X,\beta):=R_{p,q}^0(X,\beta)$ and define
\[ R_{p,q}(X,\beta,\boldsymbol{q}) = \sum_{g\geq 0}R_{p,q}^g(X,\beta)\hbar^{2g}. \]
This notation has the feature that
$R_d(X,\beta,\boldsymbol{q})|_{\hbar=0}=R_d(X,\beta)$ and $R_{p,q}(X,\beta,\boldsymbol{q})|_{\hbar=0}=R_{p,q}(X,\beta)$.
\end{definition}

\subsection{Tropical curves}

We use the definition of tropical curves from \cite{Gra1}, Definition 3.3, c.f. Definition~1 in \cite{MaR20}. 
That is, tropical curves $h : \Gamma \rightarrow B$ may have bounded legs ending in affine singularities and unbounded legs that are parallel to asymptotic directions. 
They do not have uni- or bivalent vertices and fulfill the ordinary balancing condition $\sum_{E\ni V} u_{(V,E)}=0$ at vertices $V$, if not specified otherwise.
Here $u_{(V,E)}\in\iota_\star\Lambda_{h(V)}$ is the weight vector of the flag $(V,E)$, as defined in \cite{Gra1}, Definition 3.3.

For the purpose of the following construction to work\footnote{Elementary tropical curves will no longer generate the Picard group in non-toric situations. The construction can be generalized to arbitrary toric degenerations of del Pezzo surfaces by including additional generators. 
The statement of Lemma~\ref{lemma-degree-varphi} needs to be modified in order to hold for non-toric del Pezzo degenerations.}, it is important to consider a toric degeneration $\mathfrak{X}\rightarrow\mathbb{A}^1$ that is among the list of 16 cases given in Figure 1.4 in \cite{Gra1}. 
We will refer to such a degeneration as a \emph{toric del Pezzo degeneration}. The corresponding fan pictures are depicted in Figure 1.5 in \cite{Gra1}.
To a tropical curve $h$ we associate an algebraic curve class $\beta$ as follows. 
There are certain \textit{elementary tropical curves} that are defined as follows. An elementary tropical curve is contained in the 1-skeleton of $\P$. 
It has one unbounded leg with image given by an unbounded edge of $B$ and two bounded legs with weight $1$ that end on singularities of $B$. 
Every tropical curve $h$ gives a cycle $\beta(h)$ in 
$H_1(\overline{B},B_\infty\cup\Delta;\iota_\star\Lambda_N)$. 
Replacing the endpoints on the singularities with loops as explained at the end of the introduction in \cite{R2}, the tropical curve gives a class in 
$H_1(\overline{B},B_\infty;\iota_\star\Lambda_N)$. 
Contracting the primitive global 2-form yields a class in $H_1(\overline{B},B_\infty;\iota_\star\widecheck{\Lambda}_N)$. 
Finally, the Poincar\'e--Lefschetz isomorphism proved in Theorem 1 in \cite{R2} gives a class in 
$H^1(B,\iota_\star\widecheck{\Lambda}_N)$.

On the other hand, the
projection to the unbounded direction or deformation of the toric degeneration as in \cite{Gra1}, {\S}3.4, shows that an elementary tropical curve corresponds to a toric divisors of $X$.  The natural isomorphism $H^1(\overline{B},\iota_\star\widecheck{\Lambda}_N)\cong\text{Pic}(X)$ gives the toric divisor (line bundle) corresponding to the elementary tropical curve.
In particular, the classes of elementary tropical curves generate $H^1(B,\iota_\star\widecheck{\Lambda}_N)$ and therefore, given any tropical curve $h : \Gamma \rightarrow B$, we can read off the class of $h$ in $H_1(\overline{B},B_\infty;\iota_\star\Lambda_N)$ from its intersection with the elementary tropical curves. 
The intersection can be computed according to Theorem~6 in \cite{R2} and it agrees with the ordinary intersection of the corresponding singular homology 2-cycles in $X$ by Theorem~7 in \cite{R2}.

The union of the elementary tropical curves yields the one-skeleton of $\P$ in $B$. 
The entire one-skeleton of $\P$ in $B$ is the support of the tropical 1-cycle that belongs to the piecewise linear function $\varphi$. 
In fact, it is straightforward to check that $c_1(\varphi)\in H^1(B,\iota_\star\widecheck{\Lambda}_N)$ agrees with the sum of the classes of the elementary tropical curves. 
This observation is related to the following important consequence for us. 
Recall from \eqref{eq-map-deg} that the degree homomorphism $\text{deg} : H^1(\overline{B},\iota_\star\widecheck{\Lambda}_N)\ra \ZZ$ is defined by $\beta\mapsto \beta.c_1(\varphi)$.
For a tropical cycle $\beta\in H_1(\overline{B},B_\infty;\iota_\star\Lambda_N)$, denote the sum of the weights of its unbounded legs weighted by the orientation of the leg by $\beta. B_\infty$, so that  $\beta. B_\infty$ agrees with the image of $\beta$ under the boundary map $H_1(\overline{B},B_\infty;\iota_\star\Lambda_N)\ra H_0(B_\infty,\iota_\star\Lambda_N)\cong\ZZ$.

\begin{lemma} 
\label{lemma-degree-varphi}
If $(B,\P,\varphi)$ is the fan picture of a toric del Pezzo degeneration, then
for every $\beta\in H_1(\overline{B},B_\infty;\iota_\star\Lambda_N)$ we have
$$\deg(\mathsf{D}_N(\beta)) = \beta.B_\infty.$$
\end{lemma}
\begin{proof}
It is straightforward to compute that the pairing of $c_1(\varphi)$ with an elementary tropical curve is $1$. 
Also, the pairing of the elementary tropical curve with $B_\infty$ equals $1$ since it has a unique unbounded leg of weight one. 
Since the class of any tropical curve is a linear combination of elementary tropical curves, the result follows.
\end{proof}
\begin{remark}
Recall the distinguished class $\beta_{B_\infty}$ from \eqref{eq-dist-class}. 
It is not hard to show for $B$ of arbitrary dimension that for every tropical 1-cycle $\beta$ we have $\beta.B_\infty = \beta.\beta_{B_\infty}$. 
By the perfectness of the pairing, the statement of the Lemma~\ref{lemma-degree-varphi} is therefore equivalent to saying that for toric del Pezzo degenerations, we have
$\beta_{B_\infty}=c_1(\varphi)$.
In \cite{CPS}, particular toric degenerations for \emph{non-toric} del Pezzo surfaces were given using the linear system $2K_X$. 
In these situations, $2\beta_{B_\infty}=c_1(\varphi)$.
For yet more general toric degenerations, even when $\dim B=2$, $c_1(\varphi)$  need not be a multiple of $\beta_{B_\infty}$.
\end{remark}
For shorter notation, we will denote the composition $\deg\circ \mathsf{D}_N$ simply by $\deg$.

\begin{definition}
\label{def-Hpq}
Following the same steps as for elementary tropical curves, we can associate an effective class $\beta\in\Pic(X)$ to every tropical curve $h\colon \Gamma\ra B$.
We write $\mathfrak{H}_n(X,\beta)$ for the set of tropical curves on the dual intersection complex $B$ of $X$, of class $\beta\in \Pic(X)$ and with one unbounded leg of weight $n$. Similarly, for a point $P$ on $B$, we write $\mathfrak{H}_{p,q}(X,\beta,P)$ for the set of tropical curves with $2$ unbounded legs of weight $p$ and $q$ such that the image of the former contains $P$.
\end{definition}

\begin{definition}[Tropical multiplicity] 
\label{def:tropmult}
Let $h : \Gamma \rightarrow B$ be a tropical curve. Write $V(\Gamma)$ for the set of vertices of $\Gamma$ and $L_\Delta(\Gamma)$ for the set of bounded legs, which necessarily have to end in affine singularities of $B$. For a trivalent vertex $V\in V(\Gamma)$ define, with $u_{(V,E)}\in\iota_\star\Lambda_{h(V)}$ as in \cite{Gra1}, Definition 3.3,
\[ m_V=\lvert u_{(V,E_1)}\wedge u_{(V,E_2)}\rvert=\lvert\text{det}(u_{(V,E_1)}|u_{(V,E_2)})\rvert, \]
where $E_1,E_2$ are any two edges adjacent to $V$. For a vertex $V\in V(\Gamma)$ of valency $\nu_V>3$ let $h_V$ be the one-vertex tropical curve describing $h$ locally at $V$ and let $h'_V$ be a deformation of $h_V$ to a trivalent tropical curve. 
This deformation has $\nu_V-2$ vertices. We define 
\[ m_V=\prod_{V'\in V(h'_V)}m_{V'} \]
and, by Proposition 2.7 in \cite{GPS}, this expression is independent of the deformation $h'_V$, hence well-defined. For a bounded leg $E\in L_\Delta(\Gamma)$ with weight $w_E$ define 
\[ m_E=\frac{(-1)^{w_E+1}}{w_E^2}. \]
We define the \textit{multiplicity} of $h$ to be
\[ m_h = \frac{1}{|\text{Aut}(h)|} \cdot \prod_V m_V \cdot \prod_{E\in L_\Delta(\Gamma)} m_E. \]
\end{definition}

\begin{definition}%
[$\boldsymbol{q}$-refined tropical multiplicity]
\label{def:qtropmult}
Let $h : \Gamma \rightarrow B$ be a tropical curve. For a trivalent vertex $V$ with multiplicity $m_V$ (Definition \ref{def:tropmult}) define, with $\boldsymbol{q}=e^{i\hbar}$,
\[ m_V(\boldsymbol{q}) = \frac{1}{i\hbar}\left(\boldsymbol{q}^{m_V/2}-\boldsymbol{q}^{-m_V/2}\right). \]
For a vertex with higher valency define $m_V(\boldsymbol{q}) = \prod_{V'\in V(h'_V)} m_{V'}(\boldsymbol{q})$ with $h'_V$ as in Definition \ref{def:tropmult}. For a bounded leg $E$ with weight $w_E$ define
\[ m_E(\boldsymbol{q}) = \frac{(-1)^{w_E+1}}{w_E}\cdot \frac{i\hbar}{\boldsymbol{q}^{w_E/2}-\boldsymbol{q}^{-w_E/2}}. \]
Then define the \textit{$\boldsymbol{q}$-refined multiplicity} of $h$ to be
\[ m_h(\boldsymbol{q}) = \frac{1}{|\textup{Aut}(h)|} \cdot \prod_{V\in V(\Gamma)} m_V(\boldsymbol{q}) \cdot \prod_{E\in L_\Delta(\Gamma)} m_E(\boldsymbol{q}). \]
Note that setting $\hbar$ to 0 in $m_h(\textbf{q})$ gives $m_h$ as defined above
\end{definition}

\begin{definition}
We define $R_{p,q}^{\text{trop}}(X,\beta,\boldsymbol{q}):=0$ if $p\leq 0$ or if $q=0$ and otherwise 
\[ R_{p,q}^{\text{trop}}(X,\beta,\boldsymbol{q}) = \sum_{h\in\mathfrak{H}_{p,q}(X,\beta,P)} m_h(\boldsymbol{q}), \]
where $P$ is a general point in an unbounded chamber of $\mathscr{S}_{p+q}$. 
Let $R_{p,q}^{g,\text{trop}}(X,\beta)$ denote the $\hbar^{2g}$-coefficient of $R_{p,q}^{\text{trop}}(X,\beta,\boldsymbol{q})$, i.e.,
\[ R_{p,q}^{\text{trop}}(X,\beta,\boldsymbol{q}) = \sum_{g\geq 0} R_{p,q}^{g,\text{trop}}(X,\beta)\hbar^{2g}, \]
and we write $R_{p,q}^{\text{trop}}(X,\beta):=R_{p,q}^{0,\text{trop}}(X,\beta)$.
\end{definition}

\subsection{Tropical correspondence}

In the setting considered here, the following correspondence theorem for log Calabi--Yau surfaces is proved in \cite{Gra2}, Theorem 6.2.

\begin{theorem}
\label{prop:trop}
$R_{p,q}^{g,\trop}(X,\beta)=p\cdot R_{p,q}^g(X,\beta)$. 
\end{theorem}

A \textit{tropical disk} is a tropical curve with a single $1$-valent vertex $V_\infty$, at which no balancing condition is required. Let $\frak{H}^\circ_q(P)$ denote the set of tropical disks with a single unbounded leg of weight $q$ and $1$-valent vertex mapping to a given point $P\in B$.
We cite Proposition 3.2 from \cite{Gra2}. 
\begin{proposition} 
\label{prop-finite-extension-to-disks}
If $\dim B=2$, $\frak{u}$ is an unbounded chamber of $\cS_k$ and $P\in\frak{u}$ a general point then there is a subset $(\frak{H}^\circ_q)^{(k)}(P)\subset \frak{H}^\circ_q(P)$ and a surjective map $\rho\colon (\frak{H}^\circ_q)^{(k)}(P)\ra \frak{B}^{(k)}_q(P\in\frak{u})$
such that for each $\frak{b}\in\frak{B}^{(k)}_q(P\in\frak{u})$ the preimage $\rho^{-1}(\frak{b})$ is finite.
We say a tropical disk in $h^\circ\in \rho^{-1}(\frak{b})$ is obtained from $\frak{b}$ by \emph{disk extension}. 
\end{proposition}

\begin{definition}
In view of Proposition~\ref{prop-finite-extension-to-disks}, we call a disk $h^\circ$ \emph{inbound} if it is not contained in the straight line segment that connects $P$ with $B_\infty$ inside $\frak{u}$, otherwise we call $h^\circ$ \emph{outbound}. We use the analogous definition to call a broken line inbound or outbound.
\end{definition}

After extending an inbound broken line $\mathfrak{b}$ with endpoint $P\in\frak{u}$ in an unbounded chamber to a tropical disk $h^\circ$ by the previous proposition, we can further extend this tropical disk to a tropical curve $h$ by adding the segment between $P$ and $B_\infty$. 
Recall $\mathfrak{B}^{(k)}_{p,q}(P\in\mathfrak{u})$ from Definition~\ref{defi:Bpq} and recall $\mathfrak{H}_{p,q}(X,\beta,P)$ from Definition~\ref{def-Hpq}.
As a direct consequence of Proposition~\ref{prop-finite-extension-to-disks} the following proposition is proved in \cite{Gra2}, Corollary 3.7, and \cite{Gra2}, Proposition~4.30.
 
\begin{proposition}
\label{prop:corr}
Let $P$ be a general point in an unbounded chamber $\mathfrak{u}$ of $\mathscr{S}_{p+q}$. There is a surjective map
\[ \mu : \coprod_\beta \mathfrak{H}_{p,q}(X,\beta,P) \rightarrow \mathfrak{B}_{p,q}(P\in\mathfrak{u}) \]
such that for each $\mathfrak{b}\in\mathfrak{B}_{p,q}(P\in\mathfrak{u})$ the preimage $\mu^{-1}(\mathfrak{b})$ is finite and contained in a single component $\mathfrak{H}_{p,q}(X,\beta,P)$.  Moreover, we have
\[ a_{\mathfrak{b}}(\boldsymbol{q}) = \sum_{h\in\mu^{-1}(\mathfrak{b})} m_h(\boldsymbol{q}), \]
where $m_h(\boldsymbol{q})$ is the tropical multiplicity defined in Definition \ref{def:qtropmult}.
We say a tropical curve in $\mu^{-1}(\mathfrak{b})$ is obtained from $\mathfrak{b}$ by \emph{disk extension}.
\end{proposition}

\begin{definition}
Let $\mathfrak{B}_{p,q}(\beta,P\in\mathfrak{u})$ be the set of broken lines $\mathfrak{b}$ in $\mathfrak{B}_{p,q}(P\in\mathfrak{u})$ such that $\in\mu^{-1}(\mathfrak{b})$ is contained in $\mathfrak{H}_{p,q}(X,\beta,P)$. This is well-defined by the above proposition.
\end{definition}

Proposition~\ref{prop:corr} is further refined in \cite{MR21} in the situation where the base ring $A$ is $A_N$ and slab functions are chosen universal with respect to $A_N$ which we assume from now on.
We already explained how to associate to a tropical curves $h$ a cycle $\beta(h)\in H_1(B,B_\infty;\iota_\star\Lambda_N)$.
Given a cycle $\beta\in H_1(\overline{B},B_\infty; \iota_\star\Lambda_N)$, we obtain a monomial $s^{\mathsf{D}_N(\beta)}\in A_N$ that we will refer to by the simpler notation $s^\beta$. 
We quote from \cite{MR21} the following.
\begin{proposition} 
\label{prop:MR21}
For $\mathfrak{u}$ an unbounded chamber of $\mathscr{S}_{p+q}$,
 $\mathfrak{b}\in\mathfrak{B}_{p,q}(P\in\mathfrak{u})$ a broken line and $h\in\mu^{-1}(\mathfrak{b})$ a disk extension, we have
 $ d_\mathfrak{b} = \deg(h)$ and $\beta_\mathfrak{b} = \beta(h)$. 
\end{proposition}

Combining Proposition~\ref{prop:MR21} and \ref{eq-theta-parallel} and using $y=z^{m_{\text{out}}}$, the ending monomial of 
$\mathfrak{b}\in\mathfrak{B}_{p,q}(P\in\mathfrak{u})$ reads
\begin{equation}
\label{eq-bl-expression}
 \sum_{h\in\mu^{-1}(\mathfrak{b})} m_h(\boldsymbol{q}) s^{\beta(h)} t^{\deg(h)} y^{-p}.
 \end{equation}

We obtain the following theorem that proves Step (1) from the introduction.
Recall $W((\frak{u}_k)_{k\in\NN})$ and $\vt_q((\frak{u}_k)_{k\in\NN},\boldsymbol{q})$ from Definitions~\ref{defi:theta} and \ref{defi:thetaq}.

\begin{theorem} 
\label{prop-global-expression-for-W}
Assume that $(B,\P,\varphi)$ is a polarized integral affine $2$-fold with singularities and polyhedral decomposition
with unique asymptotic monomial $y=z^{m_\text{out}}$ arising from a toric del Pezzo degeneration.
For every sequence of unbounded chambers $(\frak{u}_k)_{k\in\NN}$, the Landau--Ginzburg potential $W:=W((\frak{u}_k)_{k\in\NN})$ with respect to the base ring $A_N$ satisfies
$$W/y=1+\sum_{h} m_h s^{\beta(h)} t^{\deg(h)} y^{-\deg(h)}$$
where the sum ranges over the tropical curves $h$ with two unbounded legs of weight $1$ and $p=\deg(h)-1\ge 1$ so that this second leg passes through a fixed general point $P\in B_\infty$.
More generally $\vt_q(\boldsymbol{q}):=\vt_q((\frak{u}_k)_{k\in\NN},\boldsymbol{q})$ satisfies, 
\begin{equation}
\label{eq-theta-expression}
 \vt_q(\boldsymbol{q}) = y^{q} + \sum_{p\geq 1}\sum_{\beta:\beta.E=p+q} R_{p,q}^{\text{trop}}(X,\beta,\boldsymbol{q})s^\beta t^{\deg(\beta)}y^{-p}.
\end{equation}
\end{theorem}
\begin{proof}
By Definition \ref{defi:thetaq} and \eqref{eq-theta-parallel}, $\vt_q((\frak{u}_k)_{k\in\NN},\boldsymbol{q})=\sum_{k\ge 0}\vt^{(k)}_{q m_{\text{out}}}(\frak{u}_k,\boldsymbol{q})$ 
and
\begin{equation*}
\vt^{(k)}_{q m_{\text{out}}}(\mathfrak{u},\boldsymbol{q}) = \sum_{p\ge 1}\sum_{\mathfrak{b}\in\mathfrak{B}_{p,q}^{(k)}(\mathfrak{u},\boldsymbol{q})} a_{\mathfrak{b}}(\boldsymbol{q})t^{k}y^{-p}
\end{equation*}
By \eqref{eq-bl-expression},
\begin{align*}
\vt^{(k)}_{q m_{\text{out}}}(\mathfrak{u},\boldsymbol{q}) &= 
\sum_{{\beta:\deg(\beta)=k}\atop{p\ge 1}}\sum_{\mathfrak{b}\in\mathfrak{B}_{p,q}(\beta,\mathfrak{u},\boldsymbol{q})} a_{\mathfrak{b}}(\boldsymbol{q})t^{k}y^{-p}\\
&= \sum_{{\beta:\deg(\beta)=k}\atop{p\ge 1}}\sum_{\mathfrak{b}\in\mathfrak{B}_{p,q}(\beta,\mathfrak{u},\boldsymbol{q})}\sum_{h\in\mu^{-1}(\mathfrak{b})} 
m_h(\boldsymbol{q})s^{\beta(h)}t^{k}y^{-p}
\end{align*}
By Lemma~\ref{lemma-degree-varphi}, we have $\deg(\beta(h))=p+q$ for a toric del Pezzo degeneration. Inserting this property and the definition of $R_{p,q}^{\text{trop}}(X,\beta,\boldsymbol{q})$ yields the assertion.
\end{proof}

\subsection{Integrality of genus $0$ invariants --- relative and open}

\begin{proposition}
$R_{p,q}^{\text{trop}}$ is integral. If $\gcd(p,q)=1$, then $R_{p,q}^{\text{trop}}$ is divisible by $pq$.
\end{proposition}

\begin{proof}
Let $\mathfrak{u}$ be an unbounded chamber of $\mathscr{S}_{p+q}$. By Proposition \ref{prop:corr} we have
\[ R_{p,q}^{\text{trop}}(X,\beta) = \sum_{\mathfrak{b}\in\mathfrak{B}_{p,q}^{(p+q)}(P\in\mathfrak{u})} a_{\mathfrak{b}}, \]
with
\[ a_{\mathfrak{b}} = \sum_{h\in\mu^{-1}(\mathfrak{b})} m_h. \]
By definition of monomial transport (Definition \ref{defi:brokenline}) the coefficients $a_{\mathfrak{b}}$ of broken lines are integral. Hence $R_{p,q}^{\text{trop}}(X,\beta)$ is integral. If $\gcd(p,q)=1$, then every tropical curve $h\in\mu^{-1}(\mathfrak{b})\subset\mathfrak{H}_{p,q}(X,\beta,P)$ is primitive (not a multiple cover of another tropical curve) and has unbounded legs of weight $p$ and $q$, respectively. Hence, for any $\mathfrak{b}\in\mathfrak{B}_{p,q}^{(p+q)}(P\in\mathfrak{u})$ the coefficient $a_{\mathfrak{b}}$ is integral and has a factor of $pq$. This shows that $R_{p,q}^{\text{trop}}(X,\beta)$ is integral and divisible by $pq$.
\end{proof}

\begin{corollary}
If $\gcd(p,q)=1$, then $R_{p,q}$ is integral and divisible by $q$.
\end{corollary}

A consequence of our main result \eqref{eq-series-with-N} together with \eqref{eq-open-closed} is then the following.

\begin{corollary}
The open Gromov-Witten invariant $n_{\beta+\beta_0}(K_X)$ is integral and divisible by $\beta.E-1$.
\end{corollary}

\begin{remark}
In fact, there is a stronger statement: If $\gcd(p,q)=1$, then $R_{p,q}$ is divisible by $q^2$. This requires a generalization of a result by Cadman and Chen \cite{CC} (see Remark \ref{rem:cc}) and will be proved in forthcoming work \cite{eval}. Related integrality results were recently proved also in \cite{Yu}.
\end{remark}

\section{Theta calculations}
\label{section-theta}

In this section, we generalize the theta function multiplication rule from its genus zero version in \cite{GHS}, assuming $\dim B=2$, to arbitrary genus and then use the rule to prove relations among the two-point invariants that we introduced in the previous section.

\begin{proposition}
\label{prop:multstr} 
Let $(B,\P,\varphi)$ be an integral affine manifold with simple singularities where $\dim B=2$. Let $\frak{u}$ be an unbounded chamber in the wall structure $\cS_k$ for some $k$ and let $p$ and $q$ be asymptotic directions of $(B,\P)$.
The multiplicative structure of the corresponding theta functions is given by
\begin{equation}
\label{eq:multiplication}
\vt_p(\mathfrak{u},\boldsymbol{q}) \cdot \vt_q(\mathfrak{u},\boldsymbol{q}) = \sum_r \alpha_{p,q}^r(\mathfrak{u},\boldsymbol{q}) \cdot \vt_r(\mathfrak{u},\boldsymbol{q}) \mod t^{k+1}
\end{equation}
where the sum ranges over all asymptotic directions $r$ of $(B,\P)$.
If $r$ is an asymptotic direction of $\mathfrak{u}$, then the structure constant satisfies
\[ \alpha_{p,q}^r(\mathfrak{u},\boldsymbol{q}) = \sum_{(\mathfrak{b}_1,\mathfrak{b}_2)} a_{\mathfrak{b}_1}(\boldsymbol{q})a_{\mathfrak{b}_2}(\boldsymbol{q}). \]
 where the sum is over all pairs $(\mathfrak{b}_1,\mathfrak{b}_2)$ of broken lines with asymptotics $p$, $q$, ending in a general point of $\mathfrak{u}$, and with the property that $m_{\mathfrak{b}_1}+m_{\mathfrak{b}_2}=r$.
\end{proposition}

\begin{proof}
The statement is a $\boldsymbol{q}$-refined version of \cite{GHS}, Theorem 3.24. 
The proof of the theorem in \cite{GHS} relies only on the combinatorics of broken lines and generalizes literally to this $\boldsymbol{q}$-refined situation.
\end{proof}

Note that $\alpha_{0,q}^r(\boldsymbol{q})=0$ unless $r=q$ in which case $\alpha_{0,q}^q(\boldsymbol{q})=1$, similarly when $q=0$. On the other hand, $\alpha_{p,q}^0(\boldsymbol{q})$ is typically non-trivial.
Proposition~\ref{prop:multstr} receives a much more explicit form if $B$ is asymptotically cylindrical.

\begin{proposition} 
\label{prop:theta}
Let $B$ be the fan picture of a toric del Pezzo degeneration and let $m_{\text{out}}$ denote the primitive outward pointing asymptotic monomial. We refer to 
the multiple $rm_{\text{out}}$ of $m_{\text{out}}$ simply by $r$. If $r=p+q$ then
 $\alpha_{p,q}^{r}(\boldsymbol{q})=1$ and otherwise
\begin{equation}
\label{eq-alpha-expression}
\alpha_{p,q}^r(\boldsymbol{q}) = \sum_{\beta : \beta.E=p+q-r}(R_{p-r,q}^{\text{trop}}(\beta,\boldsymbol{q}) + R_{q-r,p}^{\text{trop}}(\beta,\boldsymbol{q}))s^\beta t^{\text{deg}(\beta)}.
\end{equation}
\end{proposition}
\begin{proof}
The statement is a consequence of Propositions \ref{prop:parallel}, \ref{prop:corr}, \ref{prop:MR21} and \ref{prop:multstr} as follows. 
By Proposition \ref{prop:multstr} the structure constants $\alpha_{p,q}^r(\boldsymbol{q})$ are given by sums over pairs of broken lines $(\mathfrak{b}_1,\mathfrak{b}_2)$ with $m_{\mathfrak{b}_1}+m_{\mathfrak{b}_2}=rm_{\text{out}}$. 
By Proposition \ref{prop:parallel} we have $m_{\mathfrak{b}_1}=r_1m_{\text{out}}$ and $m_{\mathfrak{b}_2}=r_2m_{\text{out}}$ for some $r_1,r_2\in\mathbb{Z}\setminus\{0\}$, such that $r_1+r_2=r\geq 0$. So we have $r_i>0$ for at least one $i\in\{1,2\}$. 
By Proposition \ref{prop:parallel} this means the corresponding $\mathfrak{b}_i$ is contained in the ray in direction $m_{\text{out}}$ that connects the endpoint $P$ with $B_\infty$ inside $\frak{u}$, in particular $a_{\mathfrak{b}_i}=1$. Similarly, $r_1>0$ implies $r_1=p$ and $r_2>0$ implies $r_2=q$. 
If both $r_1$ and $r_2$ are positive, we readily conclude $\alpha_{p,q}^r(\boldsymbol{q})=1$. Next assume that only one of $r_1,r_2$ is positive, say $r_1$. 
Since $r\ge 0$, we need $|r_2|<r_1$. Replacing $r_1=p$ by $p-r$ the broken lines $\mathfrak{b}_1$ and $\mathfrak{b}_2$ together give an element $\frak{b}\in \mathfrak{B}_{p-r,q}(X,P\in\mathfrak{u})$ with two unbounded legs weights of $p-r$ and $q$ and the first leg contains $P$. 
By Proposition \ref{prop:corr}, $a_{\mathfrak{b}_1}(\boldsymbol{q})a_{\mathfrak{b}_2}(\boldsymbol{q})=a_{\mathfrak{b}}(\boldsymbol{q})=\sum_{h\in\mu^{-1}(\frak{b})} m_h(\boldsymbol{q})$ and summation over such pairs $(\mathfrak{b}_1,\mathfrak{b}_2)$ gives the first term in the sum of the claimed formula. The second term similarly comes from the case $r_2>0$.
\end{proof}

For the purpose of structural clarity in the next lemma, we introduce the notation 
\begin{equation}
\label{eq-def-Rpq}
\mathcal{R}_{p,q} := \sum_{\beta:\beta.E=p+q}R^{\text{trop}}_{p,q}(\beta,\boldsymbol{q})s^\beta t^{\text{deg}(\beta)}  = \sum_{\beta:\beta.E=p+q}pR_{p,q}(X,\beta,\boldsymbol{q})s^\beta t^{\text{deg}(\beta)} 
\end{equation}
with last equality given in Theorem~\ref{prop:trop}. The expressions for $\vt_q$ from \eqref{eq-theta-expression} and for $\alpha_{p,q}^r(\boldsymbol{q})$ from \eqref{eq-alpha-expression} receive the form

\begin{align}
\label{eq-vt-compact}
\vt_q(\boldsymbol{q}) &= y^{q} + \sum_{p\geq 1}\mathcal{R}_{p,q}y^{-p},
\\
\label{eq-alpha-compact}
\alpha_{p,q}^r(\boldsymbol{q}) &=\mathcal{R}_{p-r,q}+\mathcal{R}_{q-r,p}.
\end{align}

Inserting both expressions into the multiplication rule \eqref{eq:multiplication} we obtain relations among the invariants $R_{p,q}^g(X,\beta)$. The multiplication rule implies that powers of $\vt_1(\boldsymbol{q})$ generate the $A_N\lfor t,\hbar\rfor$-algebra $\oplus_{q\ge 0}A_N\lfor t,\hbar\rfor \cdot \vt_q$. 
The resulting relations determine all $R_{p,q}^g(X,\beta)$ upon knowing $R_{p,1}^g(X,\beta)$ or, equivalently, upon knowing $R_{1,q}^g(X,\beta)$. We will know the invariants $R_{1,q}^g(X,\beta)$ from studying the blow-up of $X$ in the next section.

\begin{lemma}
\label{lem:cc}
The relation for invariants with $p+q=n+1$ obtained from the multiplication $\vt_1(\boldsymbol{q})\cdot\vt_k(\boldsymbol{q})$ for $n>k$ takes the form
\[ \mathcal{R}_{n-k,k+1}+\sum_{r=1}^{k-1}\mathcal{R}_{k-r,1}\mathcal{R}_{n-k,r} = \mathcal{R}_{n-k+1,k}+\mathcal{R}_{n,1}+\sum_{a+b=n-k}\mathcal{R}_{a,1}\mathcal{R}_{b,k}. \]
\end{lemma}

\begin{proof}
We insert \eqref{eq-vt-compact} into the product $\vt_1(\boldsymbol{q})\cdot\vt_k(\boldsymbol{q})$ and get
\[ \vt_1(\boldsymbol{q})\cdot\vt_k(\boldsymbol{q}) = \left(y + \sum_{a=1}^\infty \mathcal{R}_{a,1}y^{-a} \right)\left(y^k + \sum_{b=1}^\infty \mathcal{R}_{b,k}y^{-b}\right). \]
After expanding the product on the right, the coefficient of $y^{-(n-k)}$ yields the right hand side of identity in the assertion. 

We will derive the left hand side of the claimed identity from the product rule of Proposition~\ref{prop:multstr} where we insert the expression \eqref{eq-alpha-compact} to obtain
\[ \vt_1(\boldsymbol{q})\cdot\vt_k(\boldsymbol{q}) = \vt_{k+1}(\boldsymbol{q}) + \sum_{r=0}^{k} \mathcal{R}_{k-r,1}\vt_r(\boldsymbol{q}) + \sum_{r=0}^{k} \mathcal{R}_{1-r,k}\vt_r(\boldsymbol{q}). \]
Since $\mathcal{R}_{p,q}$ is zero for $p<1$ and $\vt_0(\boldsymbol{q})=1$, the right sum equals $\mathcal{R}_{1,k}$.
We insert \eqref{eq-vt-compact} for $\vt_{k+1}(\boldsymbol{q})$ and $\vt_r(\boldsymbol{q})$ to find
\begin{equation*}
\vt_1(\boldsymbol{q})\cdot\vt_k(\boldsymbol{q}) = y^{k+1} + \sum_{p=1}^\infty \mathcal{R}_{p,k+1}y^{-p}
+ \sum_{r=0}^{k} \mathcal{R}_{k-r,1}\left(y^r+\sum_{p=1}^\infty \mathcal{R}_{p,r}y^{-p}\right) + \mathcal{R}_{1,k}.
\end{equation*}

We claim that the coefficient of $y^{-(n-k)}$ in this expression is the left hand side of the identity in the assertion.
Indeed, the first sum contributes $\mathcal{R}_{n-k,k+1}$ and the coefficient of $y^{-(n-k)}$ in the second sum is the one where $p=n-k$, so it reads
$\sum_{r=0}^{k} \mathcal{R}_{k-r,1} \mathcal{R}_{n-k,r}$.
Since $\mathcal{R}_{0,1}=0$ and $\mathcal{R}_{n-k,0}=0$, this sum yields the corresponding term in the assertion.
\end{proof}

\begin{theorem}
\label{thm:cadman-chen}
For every $g\geq 0$ and curve class $\beta$ we have
\[ R_{1,n}^g(X,\beta) = n^2R_{n,1}^g(X,\beta) \]
and, equivalently, $R_{1,n}^{g,\trop}(X,\beta) = nR_{n,1}^{g,\trop}(X,\beta)$.
\end{theorem}

\begin{proof}
Subtracting from the identity in Lemma \ref{lem:cc} the sum given on its left hand side yields
\begin{equation}
\label{eq-cadmen-chen-first}
\mathcal{R}_{n-k,k+1} = \mathcal{R}_{n-k+1,k}+\mathcal{R}_{n,1}+\sum_{a+b=n-k}\mathcal{R}_{a,1}\mathcal{R}_{b,k}-\sum_{r=1}^{k-1}\mathcal{R}_{k-r,1}\mathcal{R}_{n-k,r}. 
\end{equation}
Specializing this equation to $k=n-1, n-2, \dots, 1,0$ yields respectively
\begin{eqnarray*}
\mathcal{R}_{1,n}&=&\mathcal{R}_{2,n-1}+\mathcal{R}_{n,1}+\ldots,\\
\mathcal{R}_{2,n-1}&=&\mathcal{R}_{3,n-2}+\mathcal{R}_{n,1}+\ldots,\\
&&\vdots\\
\mathcal{R}_{n-1,2}&=&2\mathcal{R}_{n,1}+\ldots,\\
\mathcal{R}_{n,1}&=&\mathcal{R}_{n,1}.
\end{eqnarray*}
From bottom to top, we substitute each equation into the one above.
The results is
\[ \mathcal{R}_{1,n} = n\mathcal{R}_{n,1} + \sum_{k=1}^{n-1}\sum_{a+b=n-k}\mathcal{R}_{a,1}\mathcal{R}_{b,k} - \sum_{k=1}^{n-1}\sum_{r=1}^{k-1}\mathcal{R}_{k-r,1}\mathcal{R}_{n-k,r}. \]
The right double sum does not change if we let the index $r$ run up until $n-1$ since $\mathcal{R}_{k-r,1}=0$ for $r\ge k$. 
We may then swap the index names $r$ and $k$ in the right double sum and get 
\[ \mathcal{R}_{1,n} = n\mathcal{R}_{n,1} + \sum_{k=1}^{n-1}\sum_{a+b=n-k}\mathcal{R}_{a,1}\mathcal{R}_{b,k} - \sum_{k=1}^{n-1}\sum_{r=1}^{n-1}\mathcal{R}_{r-k,1}\mathcal{R}_{n-r,k}. \]
One now observes that the double sums cancel each other under the identifications $a=r-k$ and $b=n-r$. 
We have thus proved that $\mathcal{R}_{1,n} = n\mathcal{R}_{n,1}$. 
Inserting the definition \eqref{eq-def-Rpq} and comparing the coefficients of $\hbar^{2g}s^\beta t^{\text{deg}(\beta)}$ on both sides yields the assertion.
\end{proof}

\begin{remark}
\label{rem:cc}
The previous theorem generalizes a result by Cadman--Chen (\cite{CC}, Theorem 6.6) to arbitrary genus. In fact, there is a more general formula
\[ p^2R_{p,q}^g(X,\beta) = q^2R_{q,p}^g(X,\beta). \]
This will be proved in forthcoming work \cite{eval}, as it is not needed here and the proof requires some different techniques.
\end{remark}

\begin{remark}
The relations for punctured invariants \cite{ACGS} coming from the theta multiplication rule are currently being investigated independently by Yu Wang \cite{W}.
\end{remark}

\section{One-point versus multi-point log invariants in a surface} 
\label{section-blowup}

We prove a correspondence of 
one-point log invariants of the blow-up $\widehat X$ of a surface $X$ relative to an elliptic curve $E$ with multi-point log invariants in $X$, Theorem~\ref{thm-one-to-multi-point} below. 
The proof will make use of the degeneration formula applied to a blowup of the degeneration to the normal cone of $E$ in $X$. The general fiber in the degeneration is $\widehat X$. 
One component of the degenerate fiber is a blown-up $\PP^1$-bundle $\widehat\PP$ over an elliptic curve and the other component is $X$. 
The natural projection from the union of both components to $X$ will achieve the identification of one-point invariants in the degeneration of $\widehat X$ with multi-point invariants in $X$.
We first prove a few basic lemmas about the vanishing of Gromov--Witten invariants in $\widehat\PP$. 
We assume the reader is familiar with the basic notions of log Gromov--Witten theory and the degeneration formula for which \cite{KLR} will be our primary source. Each contribution in the degeneration formula is indexed by a decorated bipartite graph where one type of vertices belongs to $X$ and the other to $\widehat\PP$. 
The edges correspond to nodes of the stable map.
We also assume familiarity with intersection theory standard notions like Gysin pullbacks and bivariant classes.
\label{section-two-point}
\subsection{Log Gromov--Witten invariants of a blown-up $\PP^1$-bundle}
\label{subsection-P1bundle}
Let $E$ be an elliptic curve and  $\shL$ a line bundle of degree $-k$ for some $k>0$. 
We denote by $E_0$ and $E_\infty$ the zero and infinity section respectively of the $\PP^1$-bundle $\PP:=\PP(\shL\oplus\shO_E)$, so $E_0^2=-k$ and $E_\infty^2=k$.
Gromov--Witten invariants of $\PP$ were recently studied tropically in \cite{Bl}.
Let $P\in E$ be a point and 
let $\widehat \PP$ be the blow-up of $\PP$ in the point that is the intersection of the fiber over $P$ with $E_\infty$. 
Let $C$ denote the exceptional curve and let $C'$ denote the strict transform of the fiber over $P$, so $C\cup C'$ is a fiber of the projection $\pi\colon \widehat \PP\ra E$.
Let $F$ denote a smooth fiber of $\pi$. 
By abuse of notation, we denote the strict transforms of $E_0,E_\infty$ in $\widehat \PP$ by the same symbols.
The resulting geometry is pictured on the right hand side of Figure~\ref{fig:degeneration-blowup}.
\begin{wrapfigure}{r}{.3\textwidth}
\centering
\includegraphics[width=.28\textwidth]{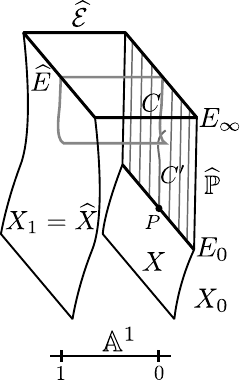}
        \captionsetup{width=.26\textwidth}
        \caption{Degeneration of $\widehat X$}
        \label{fig:degeneration-blowup}
\end{wrapfigure}

An effective curve class $\alpha$ in $\widehat \PP$ is a non-negative linear combination the classes of $E_0,F,C,C'$ and since $[C]=[F]-[C']$, by eliminating $[C]$, we can write $\alpha$ as
\begin{equation}
\label{eq-class}
\alpha = e[E_0] + d[F] + p[C']\hspace{4cm}
\end{equation}
for $d,e\ge 0$ and $p\ge-d$. A curve of class $\alpha$ must have positive genus if $e>0$ because $E_0$ has genus one.
It is easy to compute that $\alpha.E_0=d+p-ke$ and $\alpha.E_\infty=d$.

Let $Y$ denote the log scheme with underlying space $\widehat \PP$ and divisorial log structure given by $E_0\cup E_\infty$. 
Since $-E_0-E_\infty$ is the canonical divisor on $\widehat\PP$, the line bundle $\Omega^2_{Y} = \Omega^2_{\widehat\PP/E}(E_0+E_\infty)$ is isomorphic to $\shO_Y$, that is, $Y$ is log Calabi-Yau.
Let $\alpha$ be an effective curve class in $Y$ given by integers $(k,d,e)$ as above and let $(a_i)_{i=1}^r$ and $(b_j)_{j=1}^s$ be tuples of positive integers satisfying 
\begin{align*}
a_1+a_2+...+a_r&=d,\hspace{4cm}\\
b_1+b_2+...+b_s&=d+p-ke.
\end{align*}
We denote by $\cM:=\cM_{g,n}(Y,\alpha)_{(a_i),(b_j)}$ the stack of basic stable log maps to $Y$ of genus $g$ and class $\alpha$ with $n\ge r+s$ markings out of which the first $r$ have tangency profile $(a_i)$ with $E_\infty$ and the next $s$ markings have tangency profile $(b_j)$ with $E_0$ and the further markings have zero tangency to $E_0$ and $E_\infty$. 
The stack $\cM$ carries a natural virtual fundamental class $[\cM]^\vir$ of virtual dimension
\begin{equation}
\label{eq-vdim}
\vdim(\cM)\ =\ (1-g)(\dim Y-3)-\int_\alpha \Omega^{\dim Y}_{Y}+n\ =\ g-1+n.
\end{equation}

We paraphrase the statement of Lemma 7 in \cite{Bou19} as follows.
\begin{lemma}
\label{lemma-Bousseau}
Let $B$ be a scheme and $\Gamma$ a graph of genus $g>0$. 
Assume given a family of prestable curves $C_V\ra B$ for each vertex $V$ of $\Gamma$. 
Furthermore, for each edge $E$ of $\Gamma$ with vertices $V_1,V_2$ let $s^E_i\colon B\ra C_{V_i}$ be a pair of sections whose images do not meet any nodes, markings nor images of $s^{E'}_i$ for $E'\neq E$. 
We glue the curves $\{C_V\ra B\}_V$ by the natural identification of the images of $s^E_1,s^E_2$ as $B$-schemes for each edge $E$ of $\Gamma$. 
The Hodge bundle $\mathbb{E}$ of the resulting prestable curve $C_\Gamma\ra B$ surjects to the trivial bundle $\shO^{\oplus g}_B$. 
Consequently, the top Chern class $c_g(\mathbb{E})$ vanishes. 
\end{lemma}
Let $\lambda_g$ denote the top Chern class of the Hodge bundle on $\cM$. 
\begin{lemma} 
\label{lem-vanish-positive-e}
If $\alpha$ is an effective curve class given by \eqref{eq-class} with $e>0$ then $$\lambda_g\cap[\cM]^{vir}=0.$$
\end{lemma} 
\begin{proof} 
By the deformation invariance of log Gromov--Witten invariants in log smooth families, see e.g.~Appendix A in \cite{MR20}, it suffices to prove the assertion when $\cM$ is the stack of stable maps to a log smooth degeneration $\widehat\PP_0$ of $\widehat\PP$. 
We may construct such a degeneration by degenerating the elliptic curve $E$ to a necklace of three projective lines and then constructing $\widehat\PP_0$ from the necklace in a similar way as we constructed $\widehat\PP$ from $E$. 
The main point of this construction is that the curve class $[E_0]$ in the degeneration $\widehat\PP_0$ of $\widehat\PP$ will necessarily have to be represented by a nodal curve whose dual intersection graph has a positive genus because it surjects to the necklace of projective lines. Even more generally, if $e>0$, then every stable map of class $\alpha$ to the degeneration of $\widehat\PP$ will have a dual intersection graph with positive genus. By Lemma~\ref{lemma-Bousseau}, the Hodge bundle on a family of curves whose dual intersection graph has positive genus has trivial top Chern class and we therefore haven proven the lemma.
\end{proof}
Recall that $s$ is the number of markings with positive tangency to $E_0$.
\begin{lemma} 
\label{lem-vanish-1}
Assume that $s>0$ and let $\ev\colon \cM\ra E_0$ be the evaluation map of one of the markings with positive tangency to $E_0$. 
Assume that $\ev_*(\lambda_{g}\cap[\cM]^\vir)\neq 0$.
Then we have $e=0$ and $n\le 2$ and moreover
\begin{enumerate}
\item $n=2$ is equivalent to $\alpha=d[F]$ with $a_1=b_1=d$,
\item $n=1$ is equivalent to $\alpha=p[C']$ with $a_1=p>0$.
\end{enumerate}
\end{lemma}
\begin{proof} 
That $e=0$ follows from Lemma~\ref{lem-vanish-positive-e}.
By \eqref{eq-vdim}, the class $\lambda_g\cap[\cM]^\vir$ is of degree $n-1$ and since $\dim(E)=1$, the application of $\ev_*$ is trivial if $n>2$.
The assumption $s>0$ implies $\alpha.E_0=d+p-ke>0$ and thus $d+p>0$. 
If $p\neq 0$ then, by connectedness, every stable map of class $\alpha$ maps into the fiber over $P$, so the image of $\ev_*$ is $P$ which implies $n=1$. 
Since $n\ge s>0$, necessarily $\alpha.E_\infty=0$ and thus $p>0$ and $d=0$, so we are in case (2). The other possibility is $p=0$ which gives (1) as claimed.
\end{proof}
We will also need the following lemma that deals with the higher genus situation for fiber classes. Its proof bears similarities with that of Lemma 14 in \cite{Bou19}.
\begin{lemma} 
\label{lem-vanish-2}
If $\alpha=d[F]$ and $g>0$ and $\ev\colon \cM\ra E_0$ denotes the evaluation map for one of the markings with positive tangency to $E_0$ then
$\ev_*(\lambda_g\cap[\cM]^\vir)= 0$.
\end{lemma}
\begin{proof} 
By Lemma~\ref{lem-vanish-1}, we may thus assume $n=2$ and consequently, by \eqref{eq-vdim}, $\lambda_g\cap[\cM]^\vir$ is a 1-cycle. 
We want to show that its projection under $\ev$ is trivial. 
Equivalently, we may show that $i^!(\lambda_g\cap[\cM]^\vir)=0$ where 
$i:Q\hra E$ denotes the regular embedding of a general point $Q$ and $i^!$ is the Gysin pullback. 
Let $\shN_{Q/E}$ be the normal bundle of $i$. We may assume $F$ is the fiber over $Q$. 
Let $F^\dagger$ denote the log scheme with underlying scheme $F$ equipped with the log structure given by the divisor $F\cap(E_0\cup E_\infty)$.
The closed substack $\ev^{-1}(Q)$ of $\cM$ is isomorphic to the stack of basic stable log maps $\cM_F:=\cM_{g,n}(F^\dagger,\alpha)_{(d),(d)}$ to $F^\dagger$ with tangency $d$ at $0$ and $\infty$ respectively.
Let $j\colon\cM_F\hra \cM$ denote the inclusion.
The obstruction theories on $\cM$ and $\cM_F$ differ by the bundle whose fiber at a stable log map $f:\shC \ra F^\dagger$ is $H^1(\shC,f^*\pi^*\shN_{Q/E})$. 
By Serre duality, this bundle is thus isomorphic to the dual of the Hodge bundle and so its top Chern class is $(-1)^g\lambda_g$. We find that 
$$j_*i^!(\lambda_g\cap[\cM]^\vir)=((-1)^g\lambda_g\cup\lambda_g)\cap[\cM]^\vir$$
and we conclude the proof by noting that $\lambda_g^2=0$.
\end{proof}

\subsection{Log Gromov--Witten invariants of $\shO_{\PP^1}(-1)$}
If $X$ denotes the total space of $\shO_{\PP^1}(-1)$ with log structure given by the divisor $D$ which is a fiber of the bundle, we may consider the stack $\cM_{g,1}(X,p)_{(p)}$ of 1-marked genus $g$ stable log maps to $X$ of degree $p$ where the marking has tangency $p$ to $D$. By \eqref{eq-vdim}, the virtual dimension is $g$.
We define
\begin{equation}
\label{eq-cover-contribution}
N(g,p):=\int_{[\cM_{g,1}(X,p)_{(p)}]^\vir}(-1)^g\lambda_g.
\end{equation}
By \cite{BP}, Theorem 5.1, and \cite{Bou20}, Lemma 5.9, $N(g,p)$ equals the coefficient of $\hbar^{2g}$ in the expression $\frac{(-1)^{p+1}}{p} \frac{i\hbar}{\textbf{q}^{pih/2}-\textbf{q}^{-pih/2}}$, where $\textbf{q}=e^{i\hbar}$. So in particular $N(0,p)=\frac{(-1)^{p+1}}{p}$.

\subsection{Relating one point invariants to multi point invariants}
\label{subsection-two-point}
Let $X$ be a smooth surface containing a smooth elliptic curve $E$. 
Let $\beta$ be a curve class in $X$ 
so that $d:=\beta.E$ is positive.
Let $p_1,...,p_r,q,g,n$ be non-negative integers satisfying $p_i,n\ge 1$ and let $d=q+\sum_i p_i$. 
Let $\cM$ denote the stack of genus $g$ and class $\beta$ basic stable log maps to $(X,E)$ with $n+r$ markings that have tangency profile $(p_1,...,p_r,q,0....,0)$ with $E$.
If $\cN$ denotes the stack of genus $g$ and class $\beta$ ordinary stable log maps to $X$ with $n$ markings, there is a natural functor $\shF\colon\cM\ra\cN$ that forgets the log structure and the first $r$ markings.
Let $\gamma_{2},...,\gamma_n$ be bivariant classes on $\cN$ so that $\gamma_i$ is the pullback of a bivariant class from the universal curve under the section given by the $i$th marking.
For $\ev_{i}\colon \cN\ra X$ the $i$th evaluation, we assume that the restriction of $\gamma_i$ to $\ev_{i}^{-1}(E)$ vanishes.
We may view the $\gamma_i$ as bivariant classes on $\cM$ via pullback under $\shF$.
Set $\gamma=\gamma_{2}\cup...\cup \gamma_n$.
For $s\colon P\ra E$ the embedding of a point and $1\le i\le r$, let $P_i^!$ denote the Gysin pullback $s^!$ acting on cycle classes in the source of the morphism $\ev_i\colon\cM \ra E$.
We denote by
$$R^g_{p_1,...,p_r,q}(X,\beta)(\gamma):=
(-1)^g\cdot \deg\left(\gamma\cap\lambda_{g}\cap P_1^!...P_r^![\cM]^{vir}\right)$$ 
the genus $g$ and $(n+r)$-marked log Gromov-Witten invariant with target $(X,E)$ for curves of class $\beta$ with tangency profile $(p_1,...,p_r,q,0....,0)$ where the first $r$ markings are required to map to specified points on $E$ and we use the insertion $(-1)^g\lambda_g\cup\gamma$.
Note that for $n=r=1$, $R^g_{p,q}(X,\beta)(1)$ is the previously defined invariant $R^g_{p,q}(X,\beta)$. 

Let $\pi:\widehat X\ra X$ be the blow-up of $X$ in a point $P$ on $E$ with exceptional curve denoted by $C$. 
By abuse of notation, we also refer to the corresponding curve class $[C]$ by $C$. 
The strict transform of $E$ in $\widehat X$ we call $\widehat E$. 
We also consider $\gamma_i$ as bivariant classes on the stack of basic stable log maps to $(\widehat X,\widehat E)$ via the natural map to $\cN$.
Similar to the convention just introduced, for $p>0$ and $q=\beta.E-p$,
$R^g(\widehat X,\pi^*\beta-pC)(\gamma)$ denotes 
the genus $g$ and $n$-marked log Gromov--Witten invariant with target $(\widehat X,\widehat E)$ and tangency profile $(q,0,...,0)$ for curves of class $\pi^*\beta-pC$ with insertion $(-1)^g\cdot \gamma\cup\lambda_g$. 
Again, if $\gamma=1$, we simply write $R^g(\widehat X,\pi^*\beta-pC)$ for the invariant.

\begin{theorem} 
\label{thm-one-to-multi-point} 
We have the identity
$$R^{g}(\widehat X,\pi^*\beta-pC)(\gamma) =  \sum_{{p_1+\cdots+p_r=p}\atop{{g_0+\cdots+g_r=g}\atop{r>0,p_i>0,g_i\ge 0}}}\frac{p_1\cdot...\cdot p_r}{r!}\cdot
R^{g_0}_{p_1,...,p_r,q}(X,\beta)(\gamma)\cdot \prod_{i=1}^r N(g_i,p_i).$$
\end{theorem}
\begin{proof}
We prove a slightly stronger statement about Chow cycles using the degeneration formula.
We blow up the product $X\times\AA^1$ first in $E\times\{0\}$ and secondly in the strict transform of $P\times\AA^1$.
The resulting space $\shX$ comes with a flat morphism $f$ to $\AA^1$ by composing the blow-downs with the second projection. 
The fiber $X_1=f^{-1}(1)$ is the blow-up of $X$ in $P$. 
The central fiber $X_0=f^{-1}(0)$ has two irreducible components, one of which is isomorphic to $X$ and the other is isomorphic to the blow-up $\widehat\PP$ of the $\PP^1$-bundle $\PP:=\PP(\shO_X(-E)|_E\oplus\shO_E)$ in the point $(1,0)$ in the fiber over $P$, so there is a natural projection $X_0\ra X$.
Let $E_\infty$ denote the intersection of the strict transform $\widehat\shE$ of $E\times \AA^1$ under both blow-ups with the central fiber $X_0$, so $E_\infty$ is isomorphic to $E$ and sits ``at infinity'' in the blown-up $\PP^1$-bundle, that is, $E_\infty$ does not meet the component of $X_0$ that is isomorphic to $X$.
The geometry is summarized in Figure~\ref{fig:degeneration-blowup}.
 We upgrade the morphism $\shX\ra\AA^1$ to a log smooth morphism using the 
divisorial log structures from the divisors $\widehat\shE\cup X_0$ and $\{0\}$
respectively. We arrive at a commutative diagram of stacks of basic stable log maps where every square is Cartesian
\begin{equation}
\label{eq-F0-F1}
\begin{gathered}
\xymatrix{
& \cM(X_1)\ar[r]\ar^{F_1}[d] & \cM(\shX/\AA^1)\ar^F[d] & \cM(X_0)\ar[l] \ar^{F_0}[d] \\
& \cM(X)\ar[r]\ar[d] & \cM((X\times\AA^1)/\AA^1)\ar[d] & \cM(X)\ar[l] \ar[d] \\
& \{1\} \ar^{i_1}[r] & \AA^1  & \ar_{i_0}[l]\{0\}.
}
\end{gathered}
\end{equation}
We suppress from the notation that we consider for the top row $n$-marked genus $g$ stable log maps of class $\pi^*\beta-pC$ with first marking of tangency $q$ to $\widehat\shE$ and with insertion $(-1)^g\lambda_g\cup \gamma$.
The middle row consists of stacks of ordinary $n-1$-marked genus $g$ stable maps of class $\beta$ and insertion $(-1)^g\lambda_g\cup \gamma$. 
The three vertical arrows in the bottom half of the diagram are just the structure morphisms of the respective stacks. 
The three vertical arrows in top half of the diagram are given by composing a stable map with the blow-down, forgetting the log structure and the first marking and stabilizing.
The virtual fundamental class $[\cM(\shX/\AA^1)]^{vir}$ of the stack in the top middle Gysin pulls back to the respective virtual fundamental classes
$[\cM(X_1)]^{vir}$, $[\cM(X_0)]^{vir}$ for the stacks in the top left and right corner.
Since Gysin pullbacks commute with proper pushforward, we conclude that  
$$
(F_1)_*[\cM(X_1)]^{vir}=i_1^!F_*[\cM(\shX/\AA^1)]^{vir}=i_0^!F_*[\cM(\shX/\AA^1)]^{vir} =(F_0)_*[\cM(X_0)]^{vir}. 
$$
and similarly by the assumption that $\gamma$ pulls back from $\cM(X)$, 
\begin{equation}
\label{eq-nearby-to-degenerate}
(F_1)_*(\gamma\cap[\cM(X_1)]^{vir})=(F_0)_*(\gamma\cap[\cM(X_0)]^{vir}).
\end{equation}
We next recall the degeneration formula which 
provides an expression for $[\cM(X_0)]^{vir}$ as a sum of classes indexed by rigid tropical curves. 
As explained in section 2 and 4 in \cite{KLR}, a rigid tropical curve can be regarded as a decorated bipartite graph $\Gamma$ subject to a list of conditions. Let $V(\Gamma)$ and $E(\Gamma)$ denote the sets of vertices and edges of $\Gamma$ respectively. Each edge $e\in E(\Gamma)$ carries a weight $w_e$ which is a positive integer. The graph $\Gamma$ comes with a total ordering of its edges. We define the map
$$\Delta\colon \prod_{e\in E(\Gamma)} E\ra \prod_{e\in E(\Gamma)}\prod_{V\in V(\Gamma):V\in e} E,\qquad (x_e)_{e}\mapsto (x_e)_{V\in e}.$$ 
For $V$ a vertex of $\Gamma$, let $\cM_V$ denote the moduli stack of basic stable log maps to the component of $X_0$ that $V$ belongs to with divisorial log structure induced by $E$ and with curve class, markings and genus induced by the decorations and adjacent edges of $V$.
For a set $S$, we will use the notation $E^S=\prod_{s\in S} E$. 
The evaluation map for the markings with positive tangencies to $E$ reads $\ev\colon\cM_V\ra E^{\{e\in E(\Gamma):e\in V\}}$. Let $\bigodot_V\cM_V$ denote the fiber product $E^{E(\Gamma)}\times_{\Delta,E^{\{e\in V\}},\ev} \prod_{V} \cM_V$.
The stack of $\Gamma$-marked curves $\cM_\Gamma$ and natural maps $F\colon\cM_\Gamma\ra\cM(X_0)$ and $\phi\colon \cM_\Gamma\ra \bigodot_V\cM_V$ are introduced in diagram (1.5) in \cite{KLR}.
Theorem 1.6 in \cite{KLR} states the following identity
\begin{equation}
\label{degen-formula}
[\cM(X_0)]^{vir} = \sum_\Gamma \frac{\operatorname{lcm}(\{w_e\})}{E(\Gamma)!} F_*\phi^*\Delta^! \prod_{V\in V(\Gamma)} [\cM_V]^{vir}.
\end{equation}

There are two differences between the situation here and the case treated in \cite{KLR}. First, \cite{KLR} only treats the case of two components $X_1,X_2$. While it is stated for two smooth irreducible components, it actually gives a proof for the case of two log smooth components, and in particular the components might be reducible. Then we can apply \cite{KLR} repeatedly to the case where $X_1$ is one of our irreducible components and $X_2$ is the union of the other components. Similar arguments were given in \cite{Gra1}, \S4 and \cite{Bou19},\S 7. The second difference is that in \cite{KLR} the general fiber has trivial log structure, while our general fiber $X$ carries the divisorial log structure from the divisor $E\subset X$. However, $\widehat\shE$, the proper transform of $E\times\mathbb{A}^1$ does not meet the singular locus of the underlying space of $X_0$. Hence locally at the intersection of two or more components there is no additional log structure from $E$ and the situation is as in \cite{KLR}.

We will make use of the fact that the map $F_0$ from \eqref{eq-F0-F1} is compatible with \eqref{degen-formula} in the following sense. 
There are maps $F_{0,V}\colon\cM_V\ra \cN_V$ for each of the stacks $\cM_V$ defined as follows.

If $V$ belongs to $X$, this map $F_{0,V}$ is the map that forgets the log structure induced from $E$, forgets the tangencies to $E$ but keeps the markings. So $\cN_V$ is the stack of ordinary stable maps to $X$ with data induced from $V$ and its adjacent edges.

If $V$ belongs to $\widehat\PP$, the map $F_{0,V}$ is given by composing a stable log map to $\widehat\PP$ with the projection $\widehat\PP\ra E$ and forgetting all log structures with subsequent stabilization, so $\cN_V$ refers to stack of stable log maps to $E$ with markings, genus and class induced from $V$ and its adjacent edges. 
We arrive at a commutative diagram
\begin{equation}
\label{glue-diagram}
\begin{gathered}
\xymatrix{
\cM(X_0) \ar[d]^{F_0} & \cM_\Gamma \ar[l]_(.35)F \ar[r]^\phi & \bigodot_V\cM_V \ar[r]\ar^{\odot_{V} F_{0,V}}[d] & \prod_{V}\cM_V\ar^{\prod_{V} F_{0,V}}[d]   \\
\cM(X) &&
E^{E(\Gamma)}\times_{E^{\{V\in e\}}} \prod_{V} \cN_V \ar_(.65){G}[ll] \ar[r]\ar[d] & \prod_{V}\cN_V \ar^{\ev}[d]\\
&&  E^{E(\Gamma)} \ar^\Delta[r] &  E^{\{V\in e\}}.
}
\end{gathered}
\end{equation}
Let $\gamma_V$ denote the product of those $\gamma_i$ for which the $i$th marking is attached to $V$. In particular, $\gamma=\prod_V \gamma_V$.
Commutativity of the previous diagram implies the equalities 
\begin{equation}
\label{glue-equations}
\begin{split}
\gamma\cap F_*\phi^*\Delta^! \prod_{V\in V(\Gamma)} [\cM_V]^{vir}
&=(F_0)_*F_*\phi^*\Delta^! \prod_{V\in V(\Gamma)} \gamma_V\cap[\cM_V]^{vir}\\  
&= G_*(\odot_{V} F_{0,V})_*\phi_*\phi^*\Delta^! \prod_{V\in V(\Gamma)} \gamma_V\cap[\cM_V]^{vir}\\
&= \deg(\phi)\cdot G_*(\odot_{V} F_{0,V})_*\Delta^! \prod_{V\in V(\Gamma)} \gamma_V\cap[\cM_V]^{vir}\\
&= \deg(\phi)\cdot G_*\Delta^! \prod_{V\in V(\Gamma)} (F_{0,V})_*(\gamma_V\cap[\cM_V]^{vir}).
\end{split}
\end{equation}
For the right hand side to vanish, it suffices that $(F_{0,V})_*(\gamma_V\cap[\cM_V]^{vir})=0$ for a single vertex $V$. 
This happens in particular if one of the markings indexed by $2,...,n$ is attached to a vertex $V$ that belongs to $\widehat\PP$ by the assumption that $\gamma_i$ vanishes on $\ev_i^{-1}(E)$ if $i>1$.

We will next go through a case distinction among all possible graphs $\Gamma$ and argue for most of them that the summand in \eqref{degen-formula} indexed by $\Gamma$ vanishes after push-forward under $F_0$. We will prove the following claim.

\underline{Claim:} The only type of graph that contributes non-trivially to the sum on the right hand side of \eqref{degen-formula} after applying $\gamma$, $\lambda_g$ and taking degree is the one depicted in Figure~\ref{fig:graph-contributing}. 

\begin{figure} 
\includegraphics[width=.28\textwidth]{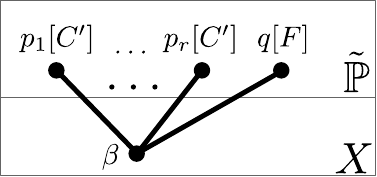}
        \caption{Claim: a graph $\Gamma$ that contributes non-trivially takes this form.}
        \label{fig:graph-contributing}
\end{figure}

We readily assume that none of vertices of $\Gamma$ that belong to $\widehat\PP$ carries a marking with trivial tangency to $E$ apart the markings induced from its adjacent edges.
Let $k$ be the negative of the self-intersection of $E_0$ in $\widehat\PP$, so $k$ is also the self-intersection of $E$ in $X$. 

\underline{i) Graphs $\Gamma$ with $g(\Gamma)>0$.}
By Lemma~\ref{lemma-Bousseau}, the application of $\lambda_g$ to the virtual class $[\cM_{\Gamma}]^{vir}$ of the stack of curves marked by $\Gamma$ is trivial if the underlying graph $\Gamma$ is non-contractible, so we may assume that the underlying graph of $\Gamma$ has genus zero.

\underline{ii) Distributing $\lambda_g$ over the vertices.}
Consider the application of $\lambda_g$ to both sides of the identity \eqref{degen-formula}.
The result on the right hand side for the summand given by a graph $\Gamma$ equals 
$\frac{\operatorname{lcm}(\{w_e\})}{E(\Gamma)!} F_*\phi^*\Delta^! \prod_{V\in V(\Gamma)} \gamma_V\cap\lambda_{g_V}\cap[\cM_{V}]^{vir}$ where $g_V$ is the genus associated with $V$. We will therefore focus on the study of $\lambda_{g_V}\cap[\cM_{V}]^{vir}$ in the following.

\underline{iii) Stable maps to $\widehat\PP$ whose image is not contained in a fiber:}
Consider the situation of a graph $\Gamma$ with a vertex $V$ which belongs to the component $\widehat\PP$. 
Let $\beta_V$ be the curve class that decorates $V$.
We follow the description of effective curves classes in $\widehat\PP$ in \S\ref{subsection-P1bundle}, so we write $\beta_V = \tilde e[E_0] + \tilde d[F] + \tilde p[C']$.
If $\tilde e>0$ then $g_V>0$ and Lemma~\ref{lem-vanish-positive-e} implies that $\lambda_{g_V}$ annihilates $[\cM_{V}]^{vir}$.
On the other hand, if the coefficient $\tilde e$ in $\beta_V$ vanishes, by connectedness, the image of a stable map of this class is then necessarily contained in a fiber of the projection $\widehat\PP\ra E$.

\underline{iv) Constant maps:} 
A constant map to either $\widehat\PP$ or $X$ cannot occur as a factor contribution for any vertex of $\Gamma$ because such a curve has trivial intersection with $E$, so $\Gamma$ does not have any edges and all of $\Gamma$ would be this single vertex. This case is excluded by the non-trivial curve class $\pi^*\beta-pC$. 

\underline{v) Stable maps to $\widehat\PP$ whose intersection with $E_0$ is non-positive:}
We consider a graph $\Gamma$ with a vertex $V$ which belongs to the component $\widehat\PP$ with the property $\beta_V.E_0\le 0$. By iii), we assume that $\tilde e=0$.
That is, we have $0\ge \beta_V.E_0=\tilde d+\tilde p-k\tilde e=\tilde d+\tilde p$ and thus $\tilde p\le -\tilde d$. 
The effectiveness constraint yields $\tilde p=-\tilde d$ and both $\tilde p,\tilde d$ must be non-zero since we excluded constant maps in the previous step. 
Using $[C]+[C']=[F]$, we rewrite the curve class as
$\beta_V = -\tilde p[C]$ and $-\tilde p>0$. 
A graph $\Gamma$ with a vertex $V$ of this class must have no edges and so $\Gamma$ consists entirely of the vertex $V$. 
However, a curve for such a graph cannot have the class $\pi^*\beta-pC$ for any $\beta$, so we may disregards this situation.

\underline{vi) Stable maps to $\widehat\PP$ whose intersection with $E_0$ is positive:}
We consider a graph $\Gamma$ with a vertex $V$ which belongs to the component $\widehat\PP$. 
By iii), we assume that $\tilde e=0$ and so by connectedness, the image of every stable map in $\cM_{V}$ is contained in a fiber of the projection $\widehat\PP\ra E$.
If $\beta_V=q[F]$ for some $q>0$ then $\ev\colon\cM_{V}\ra E^{\{e\in E(\Gamma):e\in V\}}$ factors through the diagonal $E\ra E^{\{e\in E(\Gamma):e\in V\}}$ and if $\beta_V\neq q[F]$ then 
$\ev\colon\cM_{V}\ra E^{\{e\in E(\Gamma):e\in V\}}$ factors through the composition of the point embedding $P\ra E$ with the diagonal $E\ra E^{\{e\in E(\Gamma):e\in V\}}$.
Therefore, since the result of pushing the bottom class in \eqref{glue-equations} to a point agrees with the degree of $\Delta^!\ev_*\prod_V[\cM_{V}]^{vir}$, this degree is zero by \eqref{eq-vdim} unless we are in one of the following two cases. If $g>0$, we take into consideration ii), i.e., we replace $[\cM_{V}]^{vir}$ by $\lambda_V\cap [\cM_{V}]^{vir}$.  

\underline{I) $\beta_V=q[F]$ for some $q>0$:}
The vertex $V$ carries the marking $1$ and has precisely one adjacent edge. By Lemma~\ref{lem-vanish-2}, we may assume $g_V=0$.
In this case, $\ev_*[\cM_{V}]^{vir}=\frac1q[E_0]$ because there is a unique stable map of this type with image in a given fiber and it has a cyclic automorphism group of order $q$.

\underline{II) $\beta_V=\bar p[C']$ with $0<\bar p\le p$:}
The vertex $V$ does not carry the marking $1$ and has precisely one adjacent edge with tangency $\bar p$ to $E_0$. 
Moreover, $\ev_*[\cM_{V}]^{vir}=N(g_V,\bar p)[P]$.

We finishing proving the claim.
A graph as shown in Figure~\ref{fig:graph-contributing} is uniquely determined by specifying positive integers $p_1,...,p_r$ that sum to $p$ and non-negative $g_0,...,g_r$ that sum to $g$. We order its edges by the order shown in Figure~\ref{fig:graph-contributing} from left to right. 
Let $V_\beta$ refer to the unique vertex that belongs to $X$, so $V_\beta$ carries class $\beta$ and genus $g_0$.
By \cite{KLR}, Equation (1.4), we have $\deg(\phi)=\frac{\prod_e w_e}{\operatorname{lcm}(\{w_e\})}$, so that \eqref{degen-formula} and \eqref{glue-equations} yield for
$\deg\left((\gamma\cup\lambda_g)\cap[\cM(X_0)]^{vir}\right)$ the sum
$$
\sum_{{p_1+...+p_r=p}\atop{{g_0+...+g_r=g}\atop{r,p_i>0,g_i\ge 0}}}
\frac{p_1\cdot...\cdot p_r\cdot q}{r!} \deg\left(\gamma\cap\Delta^! \prod_{V\in V(\Gamma)} (F_{0,V})_*(\lambda_{g_V}\cap[\cM_V]^{vir})\right)
$$
$$
=\sum_{{p_1+...+p_r=p}\atop{{g_0+...+g_r=g}\atop{r,p_i>0,g_i\ge 0}}}
\frac{p_1\cdot...\cdot p_r}{r!}
\left(\prod_{i=1}^r N(g_i,p_i)\right)
\deg\left(\gamma\cap(P_1^!...P_r^! (\lambda_{g_0}\cap[\cM_{V_\beta}]^{vir}))\right).
$$
That $\Delta^!$ becomes $P_1^!...P_r^!$ follows from the excess intersection formula with rank zero excess bundle.
Note that $E(\Gamma)!=(r+1)!$ but we always order the edges of $\Gamma$ so that the last edge is adjacent to the vertex $V$ that carries the class $q[F]$ and thereby reduce the number of edge orders to $r!$. In other words, the identification of $\cM_{V_\beta}$ with the stack of $(n+r)$-marked stable log maps to $(X,E)$ with tangencies $p_1,...,p_r,q,0,...$ only involves a choice of the ordering of $r$ edges of $\Gamma$.
The final expression yields the assertion.
\end{proof}

\begin{corollary}
\label{cor-two-and-one-point}
Let $X$ be a del Pezzo surface with smooth anticanonical divisor $E$. 
Let $\pi\colon\widehat X\ra X$ be the blow-up of $X$ in a point. 
Let $\beta$ be a curve class in $X$ so that  $q:=\beta.E-1\ge 0$ and let $C$ be the exceptional curve in $\widehat X$. We have the identity
$$R^{g}(\widehat X,\pi^*\beta-C) =  \sum_{g_0+g_1=g}
R^{g_0}_{1,q}(X,\beta) N(g_1,1).$$
For $g=0$, we have
$R(\widehat X,\pi^*\beta-C) = 
R_{1,q}(X,\beta).$
\end{corollary}

\begin{remark}
A result similar to Corollary~\ref{cor-two-and-one-point} for absolute invariants was previously obtained in \cite{Hu},\,Theorem 1.4.
\end{remark}

\section{Proof of Theorem \ref{thm-main}}
\label{section-main-proof}
We assume $X$ is a toric Fano surface with smooth anticanonical curve $E$. 
Let $\widehat X$ be the blowup of $X$ in a point $P\in E$ with exceptional curve $C$. 
We also use $E$ to refer to the strict transform of $E$ in $\widehat X$.
We use any standard toric degeneration of the log Calabi--Yau pair $(X,E)$ 
to construct a mirror dual $\pi\colon\frak{X}\ra \Spf A_N\lfor t\rfor$ of the pair $(X,E)$.
That is, the fan picture $(B,\P,\varphi)$ of the degeneration of $(X,E)$ is in the list of Figure 1.5 in \cite{Gra1} and the 
triple $(B,\P,\varphi)$ constitutes the cone picture of the degeneration $\frak{X}$. 
We obtain the Landau--Ginzburg potential $W\in \Gamma(\frak{X},\shO_{\frak{X}})$ as described in Section~\ref{section-LG}. 
If $y=z^{m_\text{out}}$ denotes the primitive asymptotic monomial, we want to show that $W/y$ is the open mirror map of $K_X$.
Theorem~\ref{prop-global-expression-for-W} provides the expression
$$W/y = 1 + \sum_{p\geq 1}\sum_{\beta:\beta.E=p+1} R_{p,1}^{\text{trop}}(X,\beta)s^\beta t^{\deg(\beta)}y^{-\deg(\beta)}.$$
We insert the statement of Theorem~\ref{thm:cadman-chen} to get 
$$W/y = 1 + \sum_{p\geq 1}\sum_{\beta:\beta.E=p+1} \frac{1}{p}R_{1,p}^{\text{trop}}(X,\beta)s^\beta t^{\deg(\beta)}y^{-\deg(\beta)}.$$
There are now two paths that lead to the same result:
we either first insert Theorem~\ref{prop:trop} and then use Corollary~\ref{cor-two-and-one-point} or we first use Corollary~\ref{cor:tropcorrg} and then insert Theorem~\ref{prop:trop}.
The result of either path is
$$W/y = 1 + \sum_{p\geq 1}\sum_{\beta:\beta.E=p+1} \frac{1}{p} R(\widehat{X},\pi^{-1}(\beta)-C)s^\beta t^{\deg(\beta)}y^{-\deg(\beta)}.$$
For the next step, we recall a result by van Garrel, Graber and the second author.
For $\gamma$ a curve class in $\widehat{X}$, let $N(\gamma)$ denote the genus zero Gromov--Witten invariant of $K_{\widehat{X}}$ of class $\gamma$ with no markings. 
Set $d:=\gamma.E$. Since $E$ is an anticanonical divisor in $\widehat{X}$,
Corollary~1.3 in \cite{GGR} says the following.
\begin{theorem}
\label{GGR}
$R(\widehat X,\gamma) = (-1)^{d+1} d\cdot N(\gamma)$.
\end{theorem}
We want to apply this theorem to the curve class $\gamma:=\pi^{-1}(\beta)-C$.
Note that $d=\pi^{-1}(\beta).E-C.E=p+1-1=p=\deg(\beta)-1$. 
We insert the statement of Theorem~\ref{GGR} into the last expression for $W/y$ to obtain
$$W/y = 1 + \sum_{\beta} (-1)^{\deg(\beta)}N(\pi^{-1}(\beta)-C)s^\beta t^{\deg(\beta)}y^{-\deg(\beta)}.$$

Recall that we assumed that $X$ is toric and therefore also $K_X$ is toric. 
Recall the open invariants $n_{\beta_0+\beta}$ from Section~\ref{sec-open-mirror-map}.
Theorem 1.1 in \cite{LLW} states the following result.
\begin{theorem}
$n_{\beta_0+\beta}= N(\pi^{-1}(\beta)-C)$
\end{theorem}
We insert the theorem into the last expression for $W/y$ and obtain
$$W/y = 1 + \sum_{\beta\in \op{NE}(X)} (-1)^{\deg(\beta)} n_{\beta_0+\beta} s^\beta t^{\deg(\beta)}y^{-\deg(\beta)}.$$

The substitution $Q^\beta=(-1)^{\deg(\beta)}s^\beta t^{\deg(\beta)}y^{-\deg(\beta)}$ gives the identity $W/y=M(Q)$ where $M$ is the open mirror map given in \eqref{def-mirror-map}. We have therefore proved Theorem~\ref{thm-main}.

\section{Relating scattering diagrams under a blow-up}
\label{sec:scattering-blowup}

In this section we establish a correspondence between broken lines for $X$ and certain walls (or tropical curves) for its blow up $\widehat{X}$ (Theorem \ref{thm:blowupfan}). We restrict to the $2$-dimensional case with very ample anticanonical bundle. Together with the correspondence between walls and $1$-marked invariants \cite{Gra1} and the correspondence between broken lines and $2$-marked invariants \cite{Gra2} this yields a tropical proof of Theorem~\ref{thm-one-to-multi-point}. But it is also a refinement of that proposition, since a tropical curve is finer information than a curve class.

\subsection{Toric degenerations}

\label{S:toricdeg}

Let $X$ be a projective surface and let $E\subset X$ be a smooth very ample anticanonical divisor. Consider a toric degeneration $\mathfrak{X}\rightarrow\mathbb{A}^1$ of $X$ together with a divisor $\mathfrak{D}\subset\mathfrak{X}$ that is a toric degeneration of $E$. 


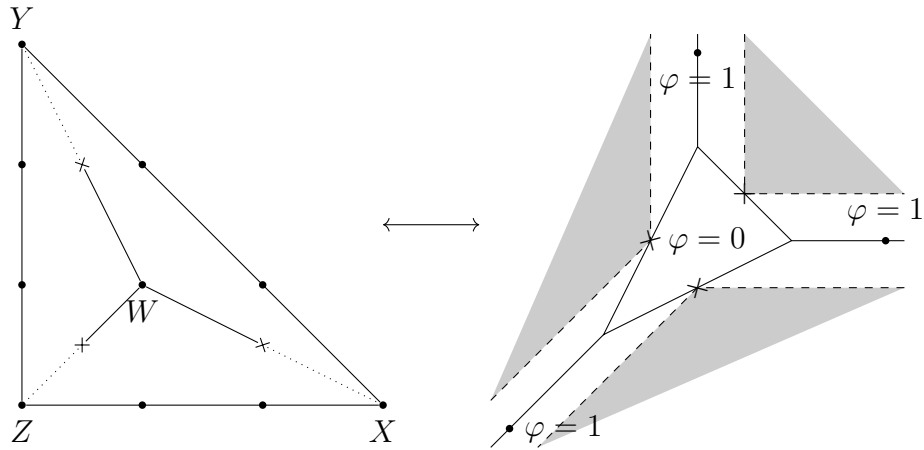
\begin{figure}[h!]
\centering
\begin{tikzpicture}[scale=1.6]
\coordinate[fill,circle,inner sep=1pt,label=below:${W}$] (0) at (0,0);
\coordinate[fill,circle,inner sep=1pt,label=below:${Z}$] (1) at (-1,-1);
\coordinate[fill,circle,inner sep=1pt,label=below:${X}$] (2) at (2,-1);
\coordinate[fill,circle,inner sep=1pt,label=above:${Y}$] (3) at (-1,2);
\coordinate[fill,cross,inner sep=2pt,rotate=45] (1a) at (-0.5,-0.5);
\coordinate[fill,cross,inner sep=2pt,rotate=63.43] (2a) at (1,-0.5);
\coordinate[fill,cross,inner sep=2pt,rotate=26.57] (3a) at (-0.5,1);
\draw (1) -- (2) -- (3) -- (1);
\draw (0) -- (1a);
\draw (0) -- (2a);
\draw (0) -- (3a);
\draw[dotted] (1a) -- (1);
\draw[dotted] (2a) -- (2);
\draw[dotted] (3a) -- (3);
\coordinate[fill,circle,inner sep=1pt] (a) at (0,-1);
\coordinate[fill,circle,inner sep=1pt] (b) at (1,-1);
\coordinate[fill,circle,inner sep=1pt] (c) at (-1,0);
\coordinate[fill,circle,inner sep=1pt] (d) at (-1,1);
\coordinate[fill,circle,inner sep=1pt] (e) at (1,0);
\coordinate[fill,circle,inner sep=1pt] (f) at (0,1);
\draw[<->] (2,0.5) -- (2.8,0.5);
\end{tikzpicture}
\begin{tikzpicture}[scale=1.25]
\draw (1,0) -- (0,1) -- (-1,-1) -- (1,0);
\draw (1,0) -- (2.2,0);
\draw (0,1) -- (0,2.2);
\draw (-1,-1) -- (-2.2,-2.2);
\draw[dashed,fill=black,fill opacity=0.2] (0.5,2.2) -- (0.5,0.5) node[opacity=1,rotate=45]{$\times$} -- (2.2,0.5);
\draw[dashed,fill=black,fill opacity=0.2] (-2.2,-1.7) -- (-0.5,0) node[opacity=1,rotate=60]{$\times$} -- (-0.5,2.2);
\draw[dashed,fill=black,fill opacity=0.2] (2.2,-0.5) -- (0,-0.5) node[opacity=1,rotate=30]{$\times$} -- (-1.7,-2.2);
\draw (0.1,0) node{$\varphi=0$};
\draw (2,0) node[fill,circle,inner sep=1pt,label=above:{$\varphi=1$}]{};
\draw (0,2) node[fill,circle,inner sep=1pt,label=below:{$\varphi=1$}]{};
\draw (-2,-2) node[fill,circle,inner sep=1pt,label=right:{$\varphi=1$}]{};
\end{tikzpicture}
\caption{The intersection complex (left) and the dual intersection complex (right) of $(\mathbb{P}^2,E)$. The shaded region is cut out and the dashed lines are mutually identified, so all unbounded edges are parallel.}
\label{fig:P2}
\end{figure}

Let $M\simeq\mathbb{Z}^2$ be a lattice with the standard area form. The momentum polytope of a toric Fano variety is a \textit{Fano polytope}, i.e., a convex lattice polytope containing the origin and with all vertices being primitive integral vectors. By duality, the bounded maximal cell $\sigma_0$ of the dual intersection complex $(B,\mathscr{P})$ of $\mathfrak{X}$ is a Fano polytope as well. Its boundary has a natural orientation (counterclockwise) and thus we can talk about the orientation of pairs $(v_1,v_2)$ of adjacent vertices of $\sigma_0$ and of flags $(v,e)$, where $v$ is a vertex of $\sigma_0$ and $e$ is an edge of $\sigma_0$ containing $v$. 

\begin{definition}
For two integral vectors $m_1,m_2$ in $M\simeq\mathbb{Z}^2$, the \textit{kink} between them is $\kappa(m_1,m_2):=\text{det}(m_1|m_2)=\braket{m_1^\bot,m_2}$, where $m_1^\bot$ is the primitive normal vector to $m_1$ such that $\text{det}(m_1|m_1^\bot)>0$. The kink at a vertex of the bounded maximal cell is the positive kink between its adjacent edges of the bounded maximal cell.
\end{definition}

\begin{lemma}
\label{lem:fano}
Let $(v,e)$ be a flag of $\sigma_0$ such that $(v,e)$ is oriented positively and let $m$ be its primitive direction. Then the Fano property implies $\kappa(v,m)=1$. Similarly, if $(v,e)$ is oriented negatively, then $\kappa(v,m)=-1$.
\end{lemma}

\begin{proof}
Without loss of generality assume $v=(1,0)$ and $m=(-a,b)$ for some $a,b\in\mathbb{N}$. Then $b=\kappa(v,m)$. If $b>1$, then the integral point $(-\lfloor\frac{a-1}{b}\rfloor,1)$ lies above the line connecting $(-a,b)$ with the origin $(0,0)$, so by convexity lies in the interior of the bounded maximal cell, contradicting the Fano property.
\end{proof}

Let us investigate how blowing up $X$ at a point changes the dual intersection complex.

\begin{construction}
\label{con:blow}
Consider an unbounded maximal cell $\sigma$ of $\mathscr{P}$ with bounded edge of affine length $1$ and let $v_1,v_2$ be its vertices. Then $v_1,v_2$ are vertices of the bounded maximal cell $\sigma_0$. Without loss of generality we assume that $(v_1,v_2)$ is oriented positively. Note that $v_1$ and $v_2$ are also the primitive direction vectors of the unbounded edges adjacent to $v_1$ and $v_2$, respectively. Blowing up the torus fixed point corresponding to $\sigma$ amounts to inserting a vertex $v_1+v_2$ and an unbounded edge in direction $v_1+v_2$. The new vertex $v_1+v_2$ is connected with $v_1$ and $v_2$ via new bounded edges. In exchange the edge connecting $v_1$ with $v_2$ is erased. See Figure \ref{fig:blowup1} for a picture of this procedure. 
\end{construction}

\begin{figure}[h!]
\centering
\begin{tikzpicture}[scale=1.5]
\draw (1,0) -- (0,1);
\draw (1,0) node[below]{$v_1$} -- (2,0);
\draw (0,1) node[left]{$v_2$} -- (0,2);
\draw[->] (1,0) -- (0.2,-0.4) node[above]{$m_1$};
\draw[->] (0,1) -- (-0.4,0.2) node[right]{$m_2$};
\draw[dashed,fill=black,fill opacity=0.2] (0.5,2) -- (0.5,0.5) node[opacity=1,rotate=45]{$\times$} -- (2,0.5);
\draw[fill] (0,0) circle (1pt);
\draw (0.25,0.25) node{$\sigma_0$};
\draw (0.25,1.25) node{$\sigma$};
\draw (2.5,0.5) node{\Large$\leadsto$};
\end{tikzpicture}
\begin{tikzpicture}[scale=1.5]
\draw (1,0) -- (1,1) -- (0,1);
\draw (1,0) node[below]{$v_1$} -- (2,0);
\draw (1,1) node[right]{$v_1+v_2$} -- (2,2);
\draw (0,1) node[left]{$v_2$} -- (0,2);
\draw[->] (1,0) -- (0.2,-0.4) node[above]{$m_1$};
\draw[->] (0,1) -- (-0.4,0.2) node[right]{$m_2$};
\draw[dashed,fill=black,fill opacity=0.2] (0.5,2) -- (0.5,1) node[opacity=1,rotate=0]{$\times$} -- (1.5,2);
\draw[dashed,fill=black,fill opacity=0.2] (2,0.5) -- (1,0.5) node[opacity=1,rotate=0]{$\times$} -- (2,1.5);
\draw[densely dotted] (0,1) -- (1,0);
\fill[opacity=0.1] (0.5,2) -- (0.5,0.5) -- (2,0.5) -- (2,2);
\draw[densely dotted] (0.5,2) -- (0.5,0.5) node[fill,cross,inner sep=2pt,rotate=45]{} -- (2,0.5);
\draw[fill] (0,0) circle (1pt);
\draw (2.5,0.5) node{\Large$\leadsto$};
\end{tikzpicture}
\begin{tikzpicture}[scale=1.5]
\draw (1,0) -- (1,1) -- (0,1);
\draw (1,0) node[below]{$v_1$} -- (2,0);
\draw (1,1) node[right]{$v_1+v_2$} -- (2,2);
\draw (0,1) node[left]{$v_2$} -- (0,2);
\draw[->] (1,0) -- (0.2,-0.4) node[above]{$m_1$};
\draw[->] (0,1) -- (-0.4,0.2) node[right]{$m_2$};
\draw[dashed,fill=black,fill opacity=0.2] (0.5,2) -- (0.5,1) node[opacity=1,rotate=0]{$\times$} -- (1.5,2);
\draw[dashed,fill=black,fill opacity=0.2] (2,0.5) -- (1,0.5) node[opacity=1,rotate=0]{$\times$} -- (2,1.5);
\draw[fill] (0,0) circle (1pt);
\end{tikzpicture}
\caption{Blowing up the fan picture amounts to a subdivision of an unbounded maximal cell.}
\label{fig:blowup1}
\end{figure}
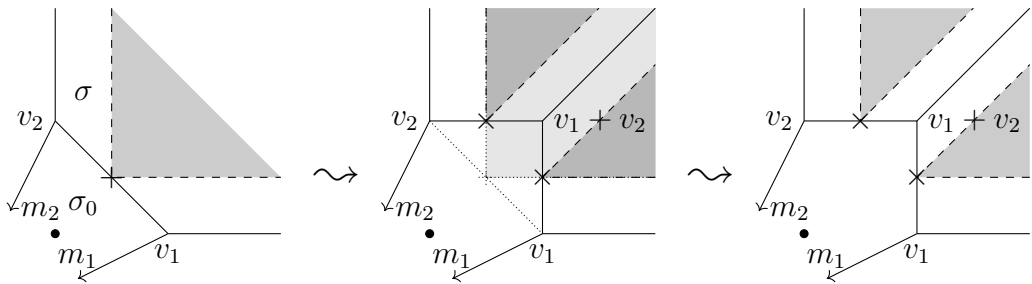

\begin{proposition}
\label{prop:blowup}
Blowing up a point corresponding to a maximal unbounded cell reduces the kink at the adjacent vertices by $1$ and produces a new vertex with kink $1$.
\end{proposition}

\begin{proof}
Let $v_1$ and $v_2$ be as above. Consider the vertex $v_1$ after blowing up. It has two bounded adjacent edges. One is pointing towards $v_1+v_2$ and has primitive direction vector $v_2$. Let $m_1$ be the primitive direction vector of the other. Then the kink at $v_1$ is $\kappa':=\kappa(v_2,m_1)>0$. Now consider $v_1$ before blowing up. Again, it has two bounded adjacent edges. One is pointing towards $v_2$ with primitive direction vector $v_2-v_1$, the other has primitive direction vector $m_1$. Then the kink at $v_1$ is
\[ \kappa = \kappa(v_2-v_1,m_1) = \kappa(v_2,m_1) - \kappa(v_1,m_1) = \kappa' + 1 > 0. \]
The last equality follows since the Fano property implies $\kappa(v_1,m_1)=-1$ by Lemma \ref{lem:fano}. This shows $\kappa'=\kappa-1$, where $\kappa$ and $\kappa'$ is the kink at $v_1$ before blowing up and after blowing up, respectively. For the other vertex $v_2$ the calculation is similar. This proves the first assertion. The new vertex $v_1+v_2$ has adjacent edges with primitive direction vectors $v_1$ and $v_2$, respectively, so the kink at $v_1+v_2$ is $\kappa(v_1,v_2)$. In construction \ref{con:blow} we required the bounded edge connecting $v_1$ with $v_2$ to have affine length $1$. So $v_1$, $v_2$ and $v_1-v_2$ are all trivial vectors and $v_1$, $v_2$ span a triangle with area $\frac{1}{2}$. This implies $\kappa(v_1,v_2)=1$, completing the proof.
\end{proof}

\subsection{Moving worms}

In this section we will give another way how to obtain the dual intersection complex of the blow up at a point $\pi:\widehat{X}\rightarrow X$ from the dual intersection complex $(B,\mathscr{P})$ of $X$. We introduce a new affine singularity and move it along some ray. Such a family of singularities was called a \textit{moving worm} in \cite{KS}. The advantage of this description is that we can still identify the dual intersection complex of $X$, together with its wall structure, inside the dual intersection complex of $\widehat{X}$ (Proposition \ref{prop:initial}).

\begin{remark}
There are two different types of charts for a $2$-dimensional affine manifold $B$, representing affine singularities in a different way. First, one can cut a wedge out of an affine manifold to represent the obstruction to the affine structure near an affine singularity. An example of this is given by the shaded regions on the right hand side in Figure \ref{fig:P2}. Alternatively, one can cut a ray out of an affine manifold and apply an affine transformation when crossing the ray on one particular side of the singularity. For an example see the dotted lines on the left hand side in Figure \ref{fig:P2}. See \cite{GS14}, {\S}2.2, for more details. In the following we will mix these two types of charts. 
\end{remark}

\begin{construction}
\label{con:blowup}
Start with the dual intersection complex of $X$ and let $m_{\text{out}}$ be its unique unbounded direction. Choose a maximal cell $\sigma$ and let $\delta$ be the affine singularity contained in it. We construct the dual intersection complex of the blow up $\widehat{X}$ as follows. Each step is illustrated in Figure \ref{fig:blowup}, with the same numbering as below.
\begin{enumerate}
\item Insert an affine singularity $\delta'$ with monodromy conjugate to $\left(\begin{smallmatrix}1 & 1 \\ 0 & 1 \end{smallmatrix}\right)$ and  invariant direction $m_{\text{out}}$ on the interior of the ray $\delta + \mathbb{R}_{\geq 0}m_{\text{out}}$. There is a wall emanating from $\delta'$ in direction $-m_{\text{out}}$ and going through the bounded maximal cell $\sigma_0$. It scatters at the slab containing $\delta$ to produce another wall in the interior of $\sigma_0$. Both walls are shown in red in Figure \ref{fig:blowup}, (1).
\item Move $\delta'$ along the ray $\delta + \mathbb{R}_{\geq 0}m_{\text{out}}$ onto $\delta$. This leads to an affine singularity $\delta''$ with more complicated monodromy and without any invariant direction. There are still two walls in the interior of $\sigma_0$.
\item Consider the walls inside $\sigma_0$ as slabs and let them be new edges of the dual intersection complex. This gives a new bounded maximal cell $\sigma'_0$ lying inside $\sigma_0$. Conversely, consider all parts of slabs outside of $\sigma'_0$ as walls, not being part of the dual intersection complex of $\widehat{X}$ (purple in Figure \ref{fig:blowup}, (3)).
\item Split the affine singularity $\delta''$ into two singularities $\delta_1$ and $\delta_2$ with invariant directions being the directions of the new slabs, then move the singularities onto the respective slabs.
\end{enumerate}
For an affine singularity $\delta$ let $T_\delta$ be the monodromy of a counterclockwise loop around $\delta$. The construction above uses the fact that the monodromy $T_{\delta''}$ has two splittings $T_{\delta''}=T_\delta \circ T_{\delta'}$ and $T_{\delta''} = T_{\delta_2} \circ T_{\delta_1}$ with different invariant directions. In the coordinates of Figure \ref{fig:blowup} we have
\[ T_\delta = \begin{pmatrix} 0 & -1 \\ 1 & 2 \end{pmatrix} \ \ T_{\delta'} = \begin{pmatrix} 1 & -1 \\ 0 & 1 \end{pmatrix} \ \ T_{\delta''} = \begin{pmatrix} 0 & -1 \\ 1 & 1 \end{pmatrix} \ \ T_{\delta_1} = \begin{pmatrix} 1 & 0 \\ 1 & 1 \end{pmatrix} \ \ T_{\delta_2} = \begin{pmatrix} 1 & -1 \\ 0 & 1 \end{pmatrix} \]
The wall structures in (1)-(4) all have the same global monodromy and the same asymptotic scattering diagram. Now lets identify $\mathscr{S}_0(X)$ in each of the steps (1)-(4).
\begin{enumerate}
\item For $\delta'$ going off to infinity (in the unique unbounded direction $m_{\text{out}}$) this is $\mathscr{S}_0(X)$ with one additional wall in direction $-m_{\text{out}}$. So to get $\mathscr{S}_0(X)$ we simply ignore this additional wall. As $\delta'$ moves towards $\delta$ the walls of $\mathscr{S}_0(X)$ may intersect $\delta + \mathbb{R}_{\geq 0}m_{\text{out}}$ either between $\delta$ and $\delta'$ (left of $\delta'$ in Figure \ref{fig:blowup},(1)) or between $\delta'$ and infinity (right of $\delta'$ in Figure \ref{fig:blowup},(1)). In the first case nothing special happens. We still ignore the additional wall, and the monodromy when crossing $\delta + \mathbb{R}_{\geq 0}m_{\text{out}}$ is the same as for $\mathscr{S}_0(X)$. In the second case, we also have to apply the monodromy of $\delta'$ when crossing $\delta + \mathbb{R}_{\geq 0}m_{\text{out}}$. The monodromy of $\delta'$ is the same as the linear map that describes scattering with a wall in direction $-m_{\text{out}}$:
\[ T_{\delta'} : m \mapsto m-|\kappa(m,m_{\text{out}})|m_{\text{out}}. \]
To get $\mathscr{S}_0(X$) we compensate this monodromy as follows. Every time a wall, say with direction $m$, crosses the wall emanating from $\delta'$ in direction $m_{\text{out}}$ it scatters to order $|\kappa(m,m_{\text{out}})|$.
\item Here we always have the second case (right of $\delta'$) from (1).
\item Same as (2), we simply changed our view on the wall structure.
\item To identify $\mathscr{S}_0(X)$ we have to compare the monodromy of $\delta$ with the monodromies of $\delta_1$ and $\delta_2$. It turns out that we can write $T_\delta = T_{\delta_2} \circ T' \circ T_{\delta_2}$ where $T'$ is the linear map describing scattering with the new wall in direction $m_{\text{out}}$:
\[ T' : m \mapsto m + |\kappa(m,m_{\text{out}})|m_{\text{out}}. \]
In the coordinates of Figure \ref{fig:blowup} this is
\[ T' = \begin{pmatrix} 0 & 1 \\ -1 & 2 \end{pmatrix}: \ \ \begin{pmatrix} -1 \\ 0 \end{pmatrix} \mapsto \begin{pmatrix} 0 \\ 1 \end{pmatrix}, \ \ \begin{pmatrix} 0 \\ 1 \end{pmatrix} \mapsto \begin{pmatrix} 1 \\ 2 \end{pmatrix} \]
\end{enumerate}

\begin{figure}[h!]
\centering
\begin{tikzpicture}[scale=1]
\draw (1,-1) node{(1)};
\draw[thick] (0,-0.5) -- (1,0) -- (0,1) -- (-0.5,0);
\draw[thick] (1,0) -- (2.5,0);
\draw[thick] (0,1) -- (0,2.5);
\draw[dashed] (0.5,0.5) node[opacity=1,rotate=45]{$\times$} -- (0.5,2.5);
\draw[dashed] (0.5,0.5) -- (1.5,0.5);
\draw[red,dotted] (1.5,0.5) node[opacity=1,rotate=0]{$\times$} -- (2.5,0.5);
\fill[opacity=0.2] (0.5,2.5) -- (0.5,0.5) -- (2.5,0.5);
\draw[red,dashed] (0.5,0.5) -- +(1,0);
\draw[red] (0.5,0.5) -- +(-1,0);
\draw[red] (0.5,0.5) -- +(0,-1);
\draw[blue] (1,0) -- +(0.5,-0.5);
\draw[blue] (0,1) -- +(-0.5,0.5);
\draw[blue] (0,1) -- +(0.5,1);
\draw[blue] (1,0) -- +(1,0.5);
\draw (0.3,0.2) node{$\delta$};
\draw (1.5,0.2) node{$\delta'$};
\draw (3,1) node{\large$\leadsto$};
\end{tikzpicture}
\begin{tikzpicture}[scale=1]
\draw (1,-1) node{(2)};
\draw[thick] (0,-0.5) -- (1,0) -- (0,1) -- (-0.5,0);
\draw[thick] (1,0) -- (2.5,0);
\draw[thick] (0,1) -- (0,2.5);
\draw[dashed] (0.5,2.5) -- (0.5,0.5) node[opacity=1,rotate=45]{$\times$} -- (0.5,0.5);
\draw[red,dotted] (0.5,0.5) node[opacity=1,rotate=0]{$\times$} -- (2.5,0.5);
\fill[opacity=0.2] (0.5,2.5) -- (0.5,0.5) -- (2.5,0.5);
\draw[red] (0.5,0.5) -- +(-1,0);
\draw[red] (0.5,0.5) -- +(0,-1);
\draw[blue] (1,0) -- +(0.5,-0.5);
\draw[blue] (0,1) -- +(-0.5,0.5);
\draw[blue] (0,1) -- +(0.5,1);
\draw[blue] (1,0) -- +(1,0.5);
\draw (0.3,0.2) node{$\delta''$};
\draw (3,1) node{\large$\leadsto$};
\end{tikzpicture}
\begin{tikzpicture}[scale=1]
\draw (1,-1) node{(3)};
\draw[thick] (0,-0.5) -- (1,0) -- (1,1) -- (0,1) -- (-0.5,0);
\draw[thick] (1,0) -- (2.5,0);
\draw[thick] (0,1) -- (0,2.5);
\draw[dashed] (1,2.5) -- (1,1) node[opacity=1,rotate=45]{$\times$} -- (1,1);
\draw[dotted] (1,1) node[opacity=1,rotate=0]{$\times$} -- (2.5,1);
\fill[opacity=0.2] (1,2.5) -- (1,1) -- (2.5,1);
\draw[blue] (1,0) -- +(0,-0.5);
\draw[blue] (0,1) -- +(-0.5,0);
\draw[blue] (1,1) -- +(1.5,-1.5);
\draw[blue] (1,1) -- +(-1.5,1.5);
\draw[blue] (0,1) -- +(3/4,3/2);
\draw[blue] (1,0) -- +(3/2,3/4);
\draw[violet] (1,1) -- +(2/3,-2/3);
\draw[violet] (1,1) -- +(-2/3,2/3);
\draw[violet] (1,0) -- +(2/3,1/3);
\draw[violet] (0,1) -- +(1/3,2/3);
\draw (0.8,0.7) node{$\delta''$};
\draw (3,1) node{\large$\leadsto$};
\end{tikzpicture}
\begin{tikzpicture}[scale=1]
\draw (1,-1) node{(4)};
\draw[thick] (0,-0.5) -- (1,0) -- (1,1) -- (0,1) -- (-0.5,0);
\draw[thick] (1,0) -- (2.5,0);
\draw[thick] (1,1) -- (2.5,2.5);
\draw[thick] (0,1) -- (0,2.5);
\draw[dashed,fill=black,fill opacity=0.2] (0.5,2.5) -- (0.5,1) node[opacity=1,rotate=0]{$\times$} -- (2,2.5);
\draw[dashed,fill=black,fill opacity=0.2] (2.5,0.5) -- (1,0.5) node[opacity=1,rotate=0]{$\times$} -- (2.5,2);
\draw[blue] (1,0) -- +(0,-0.5);
\draw[blue] (0,1) -- +(-0.5,0);
\draw[blue] (1,1) -- +(0.5,0);
\draw[blue] (1.5,0.5) -- +(1,-1);
\draw[blue] (1,1) -- +(0,0.5);
\draw[blue] (0.5,1.5) -- +(-1,1);
\draw (0.8,0.3) node{$\delta_1$};
\draw (0.4,0.7) node{$\delta_2$};
\end{tikzpicture}
\caption{The dual intersection complex (black) and walls (blue) of the blow up $\widehat{X}$, constructed from $X$ by (1) introducing an affine singularity $\delta'$ (red); (2) moving it onto an existing singularity $\delta$ to produce a new singularity $\delta''$; (3) taking walls inside the bounded maximal cell to be part of the dual intersection complex; and (4) splitting the new singularity $\delta''$ into two singularities $\delta_1$ and $\delta_2$.}
\label{fig:blowup}
\end{figure}
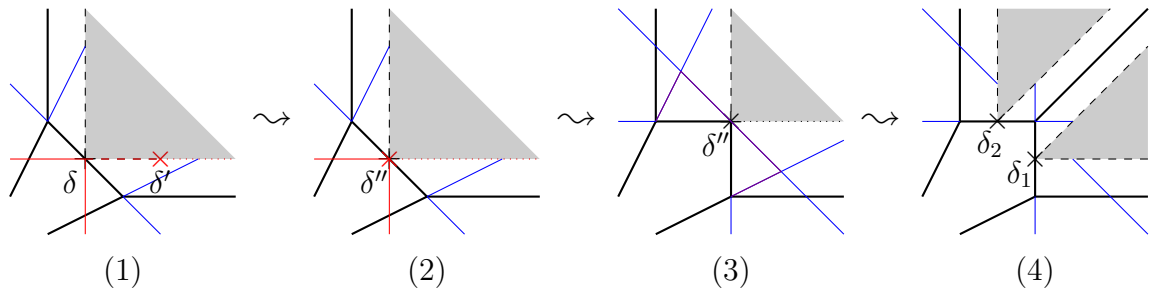

The wall structure in (4) is the wall structure of the blow up $\widehat{X}$. As a consequence $\mathscr{S}_\infty(\widehat{X})$ contains a subset isomorphic to $\mathscr{S}_0(X)$ and by consistency $\mathscr{S}_\infty(\widehat{X})$ contains a subset isomorphic to $\mathscr{S}_\infty(X)$. By the discussion above it is given as follows: ignore the walls emanating from $\delta_1$ and $\delta_2$ and not pointing towards the new wall (in direction $(1,1)$ in Figure \ref{fig:blowup}, (4)); ignore scattering of the walls emanating from $\delta_1$ and $\delta_2$ and pointing towards the new wall; every time a wall in direction $m$ crosses the new wall scatter to order $|\kappa(m,m_{\text{out}})|$ with it. See Proposition \ref{prop:initial} and Corollary \ref{cor:consistent}.
\end{construction}

\begin{remark}
It is an interesting fact, evident from step (3), that the subset of $\mathscr{S}_0(\widehat{X})$ corresponding to $\mathscr{S}_0(X)$ lies outside the central chamber of $\mathscr{S}_0(\widehat{X})$ while in Construction \ref{con:blow} the dual intersection complex of $X$ laid inside the dual intersection complex of $\widehat{X}$.
\end{remark}

Let $\rho_C$ be the unbounded edge of the dual intersection complex $\widehat{B}$ of $\widehat{X}$ corresponding to the exceptional divisor $C$ of the blow up $\pi:\widehat{X}\rightarrow X$. There are two unbounded maximal cells in $\widehat{B}$ containing $\rho_C$. They correspond to a single unbounded maximal cell $\sigma_C$ in $B$. Let $\mathscr{S}'_0(X)$ be the wall structure obtained from $\mathscr{S}_0(X)$ by cutting each wall contained in $\sigma_C$ into two walls, both having the same attached function as the original wall. $\mathscr{S}'_0(X)$ is equivalent to $\mathscr{S}_0(X)$ in the sense that it induces equivalent scattering diagrams (\cite{GPS}, Definition 1.5) at intersection points (``joints'').

\begin{definition}
We say that two (normalized) wall structures $\mathscr{S}$ and $\mathscr{S}'$ are \textit{isomorphic as abstract wall structures} if there is a bijective map $\mathscr{S} \rightarrow \mathscr{S}'$ such that
\begin{enumerate}
\item the coefficients of walls are preserved;
\item two walls intersect if and only if their images intersect, and in this case the kink between their exponents is the same.
\end{enumerate}
In other words, $\mathscr{S}$ and $\mathscr{S}'$ are the same if we forget the embedding of the wall structures into a particular affine manifold with singularities.
\end{definition}

\begin{proposition}
\label{prop:initial}
The initial wall structure $\mathscr{S}_0(X)$ is contained in $\mathscr{S}_\infty(\widehat{X})$ in the following sense: There is a subset $\mathscr{S}'_0(\widehat{X})$ of $\mathscr{S}_\infty(\widehat{X})$ that is isomorphic to $\mathscr{S}'_0(X)$ as an abstract wall structure.
\end{proposition}

\begin{proof}
Guided by Construction \ref{con:blowup} we define $\mathscr{S}'_0(\widehat{X})$ as follows. In $\mathscr{S}_0(\widehat{X})$ there are two walls coming out of the affine singularities adjacent to $\rho_C$ but not pointing towards $\rho_C$ (blue in Figure \ref{fig:blowupinitial}). Look at the complement of these walls in $\mathscr{S}_0(\widehat{X})$ and do not perform scattering at the vertex of $\rho_C$ (blue circle in Figure \ref{fig:blowupinitial}). Every time a wall, say with $z$-exponent $m$, crosses $\rho_C$ replace it by the wall after scattering to order $|\kappa(m,m_{\text{out}})|$ with $\rho_C$ (red in Figure \ref{fig:blowupinitial}). Note that $|\kappa(m,m_{\text{out}})|=1$, since this is true for all walls in $\mathscr{S}_0(\widehat{X})$. The map $\mathscr{S}'_0(\widehat{X})\rightarrow\mathscr{S}'_0(X)$ is clear: A wall in $\mathscr{S}'_0(\widehat{X})\cap\mathscr{S}_0(\widehat{X})$ is mapped to the corresponding wall in $\mathscr{S}'_0(X)$ under Construction \ref{con:blowup}. A general wall in $\mathscr{S}'_0(\widehat{X})$ is traced back to a wall in $\mathscr{S}'_0(\widehat{X})\cap\mathscr{S}_0(\widehat{X})$ by undoing the scattering at $\rho_C$ and then mapped to the corresponding wall in $\mathscr{S}_0(X)$.

Choose local coordinates such that $m_{\text{out}}=(1,0)$. Then the functions attached to walls $\mathfrak{p}$ in $\mathscr{S}'_0(\widehat{X})$ are of the form $f_{\mathfrak{p}}=1+s^\beta t^dz^m$ with $m=(k,\pm 1)$ for some $k>0$ as we will show by induction over $n$, the number of times the wall crossed $\rho_C$. For $n=0$ the wall is in $\mathscr{S}_0(\widehat{X})$ so this is clear. Now consider a wall $\mathfrak{p}$ of $\mathscr{S}'_0(X)$ with $z$-exponent $m$. By the induction hypothesis $m=(k,\pm 1)$ for some $k>0$. In particular $m$ is primitive. Let $\mathfrak{p}'$ be the wall obtained from $\mathfrak{p}$ by scattering to order $|\kappa(m,m_{\text{out}})|=1$ with $\rho_C$. By \cite{GPS}, Lemma 1.9, $f_{\mathfrak{p}'}$ has coefficient $w'|\kappa(m,m_{\text{out}})|$ and $z$-exponent $w'm'=m+m_{\text{out}}$, with $m'$ primitive. But $m+m_{\text{out}}=(k+1,\pm 1)$ in our coordinates, so $w'=1$ and $m'$ is as claimed. Moreover $|\kappa(m,m_{\text{out}})|=1$, so the coefficient of $f_{\mathfrak{p}'}$ is $1$ and $f_{\mathfrak{p}'}$ is of the claimed form.

By construction two walls in $\mathscr{S}'_0(\widehat{X})$ intersect if and only if their corresponding walls in $\mathscr{S}'_0(X)$ intersect. The kink is the same in both cases, since by Proposition \ref{prop:blowup} the kinks of the walls where we consider scattering are modified by $-n+n=0$, where $n$ is the number of times we crossed $\rho_C$.
\end{proof}

Let $\mathscr{S}'_\infty(\widehat{X})$ and $\mathscr{S}'_\infty(X)$ be the consistent wall structures defined by $\mathscr{S}'_0(\widehat{X})$ and $\mathscr{S}'_0(X)$, respectively.

\begin{corollary}
\label{cor:consistent}
The consistent wall structure $\mathscr{S}_\infty(X)$ is contained in $\mathscr{S}_\infty(\widehat{X})$ in the following sense: There is a subset $\mathscr{S}'_\infty(\widehat{X})$ of $\mathscr{S}_\infty(\widehat{X})$ that is isomorphic to $\mathscr{S}'_\infty(X)$ as abstract wall structures.
\end{corollary}

\begin{proof}
The statement follows from Proposition \ref{prop:initial} and consistency of wall structures. 
\end{proof}

\begin{corollary}
\label{cor:chamber}
There is a surjective map between chambers $\widehat{\mathfrak{u}}$ of $\mathscr{S}_m(\widehat{X})$ and chambers $\mathfrak{u}$ of $\mathscr{S}_m(X)$ such that $\widehat{\mathfrak{u}}$ is contained in the image of $\mathfrak{u}$ in $\widehat{B}$ under the map from Corollary \ref{cor:consistent}.
\end{corollary}

\begin{figure}[h!]
\centering
\begin{tikzpicture}[scale=1.6]
\draw[gray] (1,0) -- (1,1) -- (0,1) -- (-1,-1) -- (1,0);
\draw[gray] (1,0) -- (2,0) node[right]{$L-C$};
\draw[gray] (1,1) -- (2,2) node[above right]{$C$};
\draw[gray] (0,1) -- (0,2) node[above]{$L-C$};
\draw[gray] (-1,-1) -- (-2,-2) node[below left]{$L$};
\draw[gray,dashed,fill=black,fill opacity=0.2] (0.5,2) -- (0.5,1) node[opacity=1,rotate=0]{$\times$} -- (1.5,2);
\draw[gray,dashed,fill=black,fill opacity=0.2] (2,0.5) -- (1,0.5) node[opacity=1,rotate=0]{$\times$} -- (2,1.5);
\draw[gray,dashed,fill=black,fill opacity=0.2] (2,-0.5) -- (0,-0.5) node[opacity=1,rotate=30]{$\times$} -- (-1.5,-2);
\draw[gray,dashed,fill=black,fill opacity=0.2] (-0.5,2) -- (-0.5,0) node[opacity=1,rotate=60]{$\times$} -- (-2,-1.5);
\draw[violet,line width=2pt] (0.5,1) -- (1.5,1);
\draw[violet,line width=2pt] (1,0.5) -- (1,1.5);
\draw[blue,line width=1.2pt] (1,1) circle (2pt);
\draw[blue,line width=1.2pt] (0.5,1) -- (-0.5,1);
\draw[blue,line width=1.2pt] (1,0.5) -- (1,-0.5);
\draw[violet,line width=2pt] (-0.5,0) -- (0.5,2);
\draw[violet,line width=2pt] (0,-0.5) -- (2,0.5);
\draw[violet,line width=2pt] (-0.5,0) -- (-1.5,-2);
\draw[violet,line width=2pt] (0,-0.5) -- (-2,-1.5);
\draw (3.2,0) node{\Large$=$};
\draw(3.5,0);
\end{tikzpicture}
\begin{tikzpicture}[scale=1,rotate=90]
\draw[gray] (-2,-1) -- (0,0) -- (0,1) -- (-1,2) -- (-4,3);
\draw[gray,dashed] (-2,-1) -- (3,-1) node[above]{$L$};
\draw[gray] (0,0) -- (3,0) node[above]{$L-C$};
\draw[gray] (0,1) -- (3,1) node[above]{$C$};
\draw[gray] (-1,2) -- (3,2) node[above]{$L-C$};
\draw[gray,dashed] (-4,3) -- (3,3);
\draw[gray,dashed,fill=black,fill opacity=0.2] (-3,-1) -- (-1,-0.5) node[opacity=1,rotate=30]{$\times$} -- (-2/3,0) -- (0,0.5) node[opacity=1,rotate=0]{$\times$} -- (-1/3,1) -- (-1/2,1.5) node[opacity=1,rotate=45]{$\times$} -- (-5/3,2) -- (-2.5,2.5) node[opacity=1,rotate=70]{$\times$} -- (-4,2.8);
\draw[violet,line width=2pt] (0,0.5) -- (0,3);
\draw[blue,line width=1.2pt] (0,0.5) -- (0,-1);
\draw[violet,line width=2pt] (-0.5,1.5) -- (2,-1);
\draw[blue,line width=1.2pt] (-0.5,1.5) -- (-2,3);
\draw[blue,line width=1.2pt] (0,1) circle (3.2pt);
\draw[violet,line width=2pt] (-1,-0.5) -- (-2,-1);
\draw[violet,line width=2pt] (-2.5,2.5) -- (-4,3);
\draw[violet,line width=2pt] (-1,-0.5) -- (2,1);
\draw[blue,line width=1.2pt] (2,1) -- (3,1+1/2);
\draw[red,line width=1.2pt] (2,1) -- (3,1+1/3);
\draw[violet,line width=2pt] (-2.5,2.5) -- (2,1);
\draw[blue,line width=1.2pt] (2,1) -- (3,1-1/3);
\draw[red,line width=1.2pt] (2,1) -- (3,1-1/4);
\end{tikzpicture}
\caption{Some part of $\mathscr{S}_\infty(\mathbb{F}_1)$ in two different charts. The initial wall structure $\mathscr{S}_0(\mathbb{F}_1)$ consists of the blue and purple walls. The walls for $\mathscr{S}'_0(\mathbb{F}_1)$ inside $\mathscr{S}_\infty(\mathbb{F}_1)$ are the red and purple ones. In $\mathscr{S}'_0(\mathbb{F}_1)$ there is no scattering at the blue circle.}
\label{fig:blowupinitial}
\end{figure}
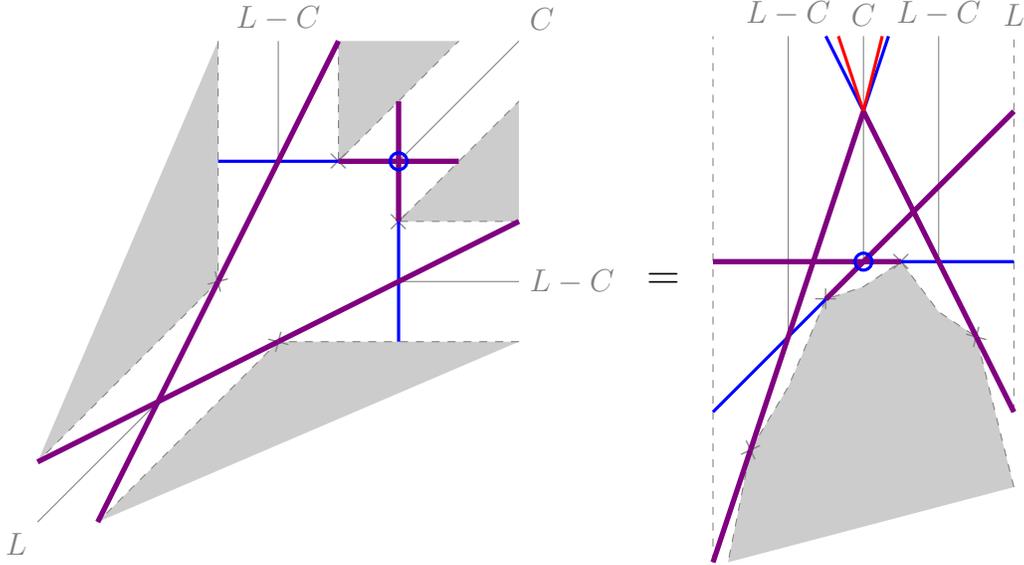

\subsection{Broken lines and tropical curves}

For the definitions and notations concerning broken lines and tropical curves see \S\ref{section-LG} and \S\ref{sec:tropicalcorrespondence}.

Let $\widehat{\mathfrak{u}}$ be an unbounded chamber of $\mathscr{S}_m(\widehat{X})$ and let $\mathfrak{u}$ be the corresponding chamber of $\mathscr{S}_m(X)$ via Corollary \ref{cor:chamber}. Let $P$ be a general point in $\mathfrak{u}$ and let $\widehat{P}$ be a general point in $\widehat{u}$. Let $\mathfrak{B}'_m(\widehat{X},\widehat{P}\in\widehat{u})$ be the subset of $\mathfrak{B}_m(\widehat{X},\widehat{P}\in\widehat{u})$ consisting of broken lines which only break at walls of $\mathscr{S}'_\infty(\widehat{X})$. Let $\rho_C$ be the unbounded edge corresponding to the exceptional line $C$.

\begin{proposition}
\label{prop:broken}
There is a bijective map $\mathfrak{B}_m(X,P\in\mathfrak{u})\rightarrow\mathfrak{B}'_m(\widehat{X},\widehat{P}\in\widehat{\mathfrak{u}})$ preserving $a_{\mathfrak{b}}(\boldsymbol{q})$, the coefficient of the ending monomial of a broken line $\mathfrak{b}$. Whenever a broken line in $\mathfrak{B}'_m(\widehat{X},\widehat{P}\in\widehat{u})$ crosses $\rho_C$ it breaks to order $|\kappa(m,m_{\text{out}})|$, where $a(\boldsymbol{q})s^\beta t^dz^m$ is its monomial before breaking. See Figure \ref{fig:brokencorr} for an example.
\end{proposition}

\begin{proof}
A broken line is determined by its endpoint, asymptotic monomial and the walls at which it breaks, including the order of the breaking. For a broken line $\mathfrak{b}$ in $\mathfrak{B}_m(X,P\in\mathfrak{u})$ consider this data. We get similar data for $\mathfrak{B}'_m(\widehat{X},\widehat{P}\in\widehat{\mathfrak{u}})$ as follows. The asymptotic monomial is $z^m$, the same as for $\mathfrak{b}$. The walls at which $\mathfrak{b}$ breaks correspond to walls of $\mathscr{S}'_\infty(\widehat{X})$ via Corollary \ref{cor:consistent}. This gives walls in $\mathscr{S}'_\infty(\widehat{X})$ together with breaking orders (the same as for $\mathfrak{b}$). Additionally, we require that a broken line in $\mathfrak{B}'_m(\widehat{X},\widehat{P}\in\widehat{\mathfrak{u}})$ breaks at $\rho_C$ to order $|\kappa(m,m_{\text{out}})|$ as in the statement of the proposition. Using this data we get a unique broken line $\widehat{\mathfrak{b}}$ in $\mathfrak{B}'_m(\widehat{X},\widehat{P}\in\widehat{\mathfrak{u}})$. Since $\mathscr{S}_\infty(X)$ and $\mathscr{S}'_\infty(\widehat{X})$ are isomorphic as abstract wall structures (Corollary \ref{cor:consistent}), the coefficients that $a_{\mathfrak{b}}(\boldsymbol{q})$ and $a_{\widehat{\mathfrak{b}}}(\boldsymbol{q})$ pick up by breaking are the same. The breaking of $\widehat{\mathfrak{b}}$ at $\rho_C$ does not change $a_{\widehat{\mathfrak{b}}}(\boldsymbol{q})$, since the wall $\mathfrak{p}_C$ corresponding to $\rho_C$ has function $f_{\mathfrak{p}_C}=1+s^{\beta_C} ty$, where $\beta_C$ is the tropical cycle corresponding to $C$, and by breaking to order $|\kappa(m,m_{\text{out}})|$ with $\mathfrak{p}_C$ the broken lines coefficient $a_{\widehat{\mathfrak{b}}}(\boldsymbol{q})$ picks up the $(s^{\beta_C} ty)^{|\kappa(m,m_{\text{out}})|}$-coefficient of $f_{\mathfrak{p}_C}^{|\kappa(m,m_{\text{out}})|}$ which is $1$. Hence we have $a_{\widehat{\mathfrak{b}}}(\boldsymbol{q})=a_{\mathfrak{b}}(\boldsymbol{q})$.
\end{proof}

\begin{figure}[h!]
\centering
\begin{tikzpicture}[scale=1.4,rotate=90]
\draw[gray] (-3,-1) -- (0,0) -- (0,1) -- (-3,2);
\draw[dashed,gray] (-3,-1) -- (2,-1);
\draw[gray] (0,0) -- (2,0);
\draw[gray] (0,1) -- (2,1);
\draw[dashed,gray] (-3,2) -- (2,2);
\draw[dashed,fill=black,fill opacity=0.2,gray] (-3,-0.8) -- (-1.5,-0.5) node[opacity=1,rotate=20]{$\times$} -- (-1,0) -- (0,0.5) node[opacity=1,rotate=0]{$\times$} -- (-1,1) -- (-1.5,1.5) node[opacity=1,rotate=70]{$\times$} -- (-3,1.8);
\draw[gray] (0,0.5) -- (0,2);
\draw[gray] (0,0.5) -- (0,-1);
\draw[gray] (-1.5,-0.5) -- (2,2/3);
\draw[gray] (-1.5,1.5) -- (2,1/3);
\draw[black!40!green,line width=2pt] (1.95,0.8) node[fill,circle,inner sep=1pt]{} -- (0.6,0.8) -- (1.32,0.44) -- (2,0.44);
\draw[blue,line width=1.2pt] (-1.5,-0.5) -- (1.32,0.44);
\draw[blue,line width=1.2pt] (-1.5,1.5) -- (0.6,0.8);
\draw (0,-1.67) node{\large$\leadsto$};
\draw (0,-2.2);
\end{tikzpicture}
\begin{tikzpicture}[scale=1,rotate=90]
\draw[gray] (-2,-1) -- (0,0) -- (0,1) -- (-1,2) -- (-4,3);
\draw[dashed,gray] (-2,-1) -- (3,-1);
\draw[gray] (0,0) -- (3,0);
\draw[gray] (0,1) -- (3,1) node[above]{$\rho_C$};
\draw[gray] (-1,2) -- (3,2);
\draw[dashed,gray] (-4,3) -- (3,3);
\draw[dashed,fill=black,fill opacity=0.2,gray] (-3,-1) -- (-1,-0.5) node[opacity=1,rotate=30]{$\times$} -- (-2/3,0) -- (0,0.5) node[opacity=1,rotate=0]{$\times$} -- (-1/3,1) -- (-1/2,1.5) node[opacity=1,rotate=45]{$\times$} -- (-5/3,2) -- (-2.5,2.5) node[opacity=1,rotate=70]{$\times$} -- (-4,2.8);
\draw[gray] (0,0.5) -- (0,3);
\draw[gray] (0,0.5) -- (0,-1);
\draw[gray] (-0.5,1.5) -- (2,-1);
\draw[gray] (-0.5,1.5) -- (-2,3);
\draw[gray] (-1,-0.5) -- (-2,-1);
\draw[gray] (-2.5,2.5) -- (-4,3);
\draw[gray] (-1,-0.5) -- (2,1);
\draw[gray] (2,1) -- (3,1+1/2);
\draw[gray] (2,1) -- (3,1+1/3);
\draw[gray] (-2.5,2.5) -- (2,1);
\draw[gray] (2,1) -- (3,1-1/3);
\draw[gray] (2,1) -- (3,1-1/4);
\draw[gray] (2/3,1/3) -- (3,1/3);
\draw[gray] (0,1+2/3) -- (3,1+2/3);
\draw[black!40!green,line width=2pt] (2.9,1.4) node[fill,circle,inner sep=1pt]{} -- (0.8,1.4) -- (1.6,1) -- (1.84,0.92) -- (3,0.92);
\draw[blue,line width=1.2pt] (-1,-0.5) -- (1.84,0.92);
\draw[blue,line width=1.2pt] (0,0.5) -- (0,1) -- (-0.5,1.5);
\draw[blue,line width=1.2pt] (0,1) -- (1.6,1);
\draw[blue,line width=1.2pt] (-2.5,2.5) -- (0.8,1.4);
\end{tikzpicture}
\caption{A broken line (green) for $\mathbb{P}^2$ (left) and the corresponding broken line for $\mathbb{F}_1$ (right). Note that the broken line for $\mathbb{F}_1$ breaks at the edge $\rho_C$. We can complete $\mathfrak{b}$ to a tropical curve $h\in\mu^{-1}(\mathfrak{b})$ (blue).}
\label{fig:brokencorr}
\end{figure}
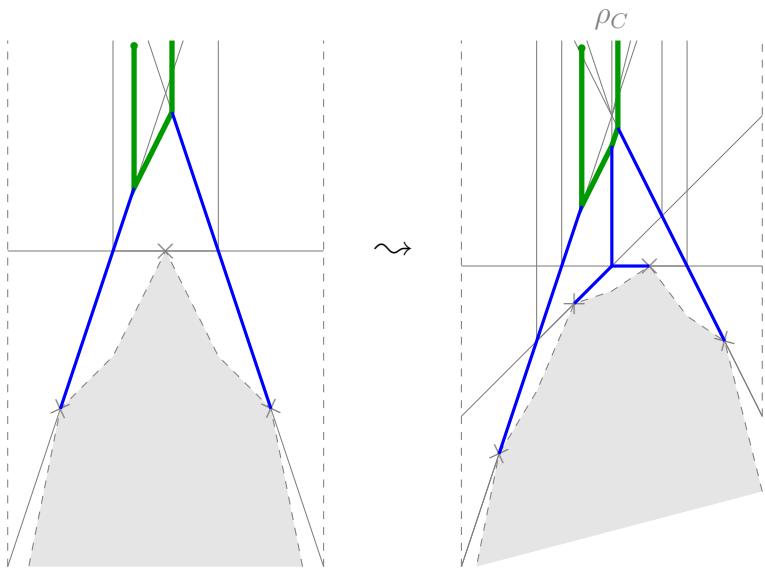

Recall the surjective map $\mu : \coprod_\beta \mathfrak{H}_{p,q}(X,\beta,P\in\mathfrak{u}) \rightarrow \mathfrak{B}_{p,q}(X,P\in\mathfrak{u})$ from Proposition \ref{prop:corr}.

\begin{lemma}
\label{lem:corr}
Let $\widehat{\mathfrak{u}}$ be an unbounded cell of $\mathscr{S}_{p+q}(\widehat{X})$. Let $\mathfrak{b}$ be a broken line in $\mathfrak{B}'_{p,q}(\widehat{X},\widehat{P}\in\widehat{\mathfrak{u}})$ that intersects $\rho_C$. Then each completion $h\in\mu^{-1}(\mathfrak{b})$ contains a tropical sub-disk of class $C$ and without unbounded legs. In particular, if the endpoint of $\mathfrak{b}$ lies on $\rho_C$, then $h$ contains the unique tropical curve of class $C$ with one unbounded leg of weight $1$ (see Figures \ref{fig:corr1} and \ref{fig:corr2}).
\end{lemma}

\begin{proof}
By Proposition \ref{prop:broken}, whenever $\mathfrak{b}$ intersects $\rho_C$ it breaks. Hence, if $\mathfrak{b}$ intersects $\rho_C$, then by construction $h\in\mu^{-1}({\mathfrak{b}})$ contains a tropical sub-disk as claimed. If the endpoint of $\mathfrak{b}$ lies on $\rho_C$, then we can add the unbounded leg of $h$ to the tropical disk to obtain a tropical curve of class $C$ with one unbounded leg of weight $1$.
\end{proof}

\begin{lemma}
\label{lem:intersects}
A tropical curve in $\mathfrak{H}_n(\widehat{X},\pi^\star\beta-C)$ intersects the edge $\rho_C$ or has a bounded leg ending in an affine singularity neighboring $\rho_C$ (see Figures \ref{fig:corr1} and \ref{fig:corr2}).
\end{lemma}

\begin{proof}
A tropical curve of class $\pi^\star\beta-C$ has to intersect an unbounded edge $\rho$ neighboring $\rho_C$, as these are the unbounded edges of class $\pi^\star\beta'-C$ for some class $\beta'$ of $X$. If it does not end in the affine singularity between $\rho$ and $\rho_C$ it has to intersect $\rho_C$ by the balancing condition.
\end{proof}

\begin{figure}[h!]
\centering
\begin{tikzpicture}[scale=1,rotate=90]
\draw[gray] (-2,-1) -- (0,0) -- (0,1) -- (-1,2) -- (-4,3);
\draw[dashed,gray] (-2,-1) -- (3,-1);
\draw[gray] (0,0) -- (3,0);
\draw[gray] (0,1) -- (3,1) node[above]{$\rho_C$};
\draw[gray] (-1,2) -- (3,2);
\draw[dashed,gray] (-4,3) -- (3,3);
\draw[dashed,fill=black,fill opacity=0.2,gray] (-3,-1) -- (-1,-0.5) node[opacity=1,rotate=30]{$\times$} -- (-2/3,0) -- (0,0.5) node[opacity=1,rotate=0]{$\times$} -- (-1/3,1) -- (-1/2,1.5) node[opacity=1,rotate=45]{$\times$} -- (-5/3,2) -- (-2.5,2.5) node[opacity=1,rotate=70]{$\times$} -- (-4,2.8);
\draw[gray] (0,0.5) -- (0,3);
\draw[gray] (0,0.5) -- (0,-1);
\draw[gray] (-0.5,1.5) -- (2,-1);
\draw[gray] (-0.5,1.5) -- (-2,3);
\draw[gray] (-1,-0.5) -- (-2,-1);
\draw[gray] (-2.5,2.5) -- (-4,3);
\draw[gray] (-1,-0.5) -- (2,1);
\draw[gray] (2,1) -- (3,1+1/2);
\draw[gray] (2,1) -- (3,1+1/3);
\draw[gray] (-2.5,2.5) -- (2,1);
\draw[gray] (2,1) -- (3,1-1/3);
\draw[gray] (2,1) -- (3,1-1/4);
\draw[gray] (2/3,1/3) -- (3,1/3);
\draw[gray] (0,1+2/3) -- (3,1+2/3);
\draw[black!40!green,line width=2pt] (2.9,1) node[fill,circle,inner sep=1pt]{} -- (2,1) -- (3,1);
\draw[blue,line width=1.2pt] (0,0.5) -- (0,1) -- (-0.5,1.5);
\draw[blue,line width=1.2pt] (0,1) -- (2,1);
\draw[blue,line width=1.2pt] (-2.5,2.5) -- (2,1) -- (-1,-0.5);
\draw (0,-2) node{\Large$\leadsto$};
\draw (0,-2.8);
\end{tikzpicture}
\begin{tikzpicture}[scale=1,rotate=90]
\draw[gray] (-2,-1) -- (0,0) -- (0,1) -- (-1,2) -- (-4,3);
\draw[dashed,gray] (-2,-1) -- (3,-1);
\draw[gray] (0,0) -- (3,0);
\draw[gray] (0,1) -- (3,1) node[above]{$\rho_C$};
\draw[gray] (-1,2) -- (3,2);
\draw[dashed,gray] (-4,3) -- (3,3);
\draw[dashed,fill=black,fill opacity=0.2,gray] (-3,-1) -- (-1,-0.5) node[opacity=1,rotate=30]{$\times$} -- (-2/3,0) -- (0,0.5) node[opacity=1,rotate=0]{$\times$} -- (-1/3,1) -- (-1/2,1.5) node[opacity=1,rotate=45]{$\times$} -- (-5/3,2) -- (-2.5,2.5) node[opacity=1,rotate=70]{$\times$} -- (-4,2.8);
\draw[gray] (0,0.5) -- (0,3);
\draw[gray] (0,0.5) -- (0,-1);
\draw[gray] (-0.5,1.5) -- (2,-1);
\draw[gray] (-0.5,1.5) -- (-2,3);
\draw[gray] (-1,-0.5) -- (-2,-1);
\draw[gray] (-2.5,2.5) -- (-4,3);
\draw[gray] (-1,-0.5) -- (2,1);
\draw[gray] (2,1) -- (3,1+1/2);
\draw[gray] (2,1) -- (3,1+1/3);
\draw[gray] (-2.5,2.5) -- (2,1);
\draw[gray] (2,1) -- (3,1-1/3);
\draw[gray] (2,1) -- (3,1-1/4);
\draw[gray] (2/3,1/3) -- (3,1/3);
\draw[gray] (0,1+2/3) -- (3,1+2/3);
\draw[blue,line width=1.2pt] (2,1) -- (3,1);
\draw[blue,line width=1.2pt] (-2.5,2.5) -- (2,1) -- (-1,-0.5);
\end{tikzpicture}
\caption{Shifting the endpoint of the broken line (green) onto the edge $\rho_C$, a tropical curve $h\in\mu^{-1}({\mathfrak{b}})$ (green and blue) contains a tropical curve of class $C$. Deleting it we obtain a tropical curve of class $2L-C$.}
\label{fig:corr1}
\end{figure}
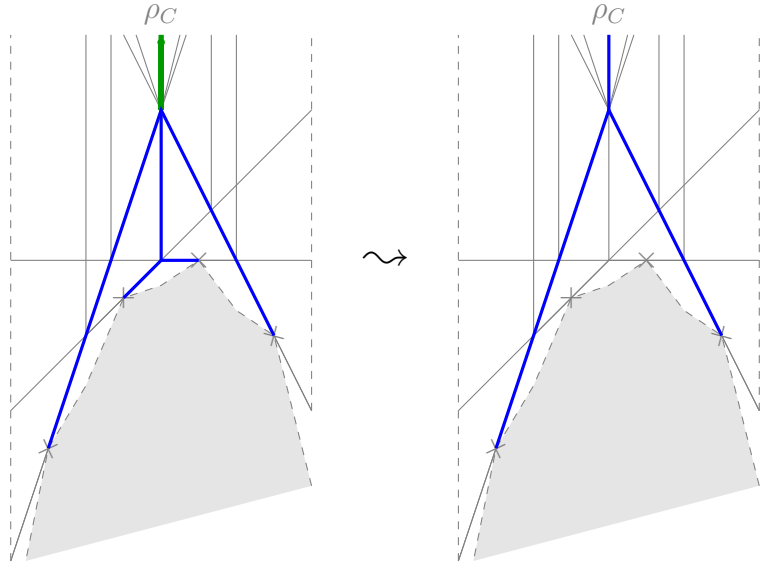

\begin{figure}[h!]
\centering
\begin{tikzpicture}[scale=1,rotate=90]
\draw[gray] (-2,-1) -- (0,0) -- (0,1) -- (-1,2) -- (-4,3);
\draw[dashed,gray] (-2,-1) -- (3,-1);
\draw[gray] (0,0) -- (3,0);
\draw[gray] (0,1) -- (3,1) node[above]{$\rho_C$};
\draw[gray] (-1,2) -- (3,2);
\draw[dashed,gray] (-4,3) -- (3,3);
\draw[dashed,fill=black,fill opacity=0.2,gray] (-3,-1) -- (-1,-0.5) node[opacity=1,rotate=30]{$\times$} -- (-2/3,0) -- (0,0.5) node[opacity=1,rotate=0]{$\times$} -- (-1/3,1) -- (-1/2,1.5) node[opacity=1,rotate=45]{$\times$} -- (-5/3,2) -- (-2.5,2.5) node[opacity=1,rotate=70]{$\times$} -- (-4,2.8);
\draw[gray] (0,0.5) -- (0,3);
\draw[gray] (0,0.5) -- (0,-1);
\draw[gray] (-0.5,1.5) -- (2,-1);
\draw[gray] (-0.5,1.5) -- (-2,3);
\draw[gray] (-1,-0.5) -- (-2,-1);
\draw[gray] (-2.5,2.5) -- (-4,3);
\draw[gray] (-1,-0.5) -- (2,1);
\draw[gray] (2,1) -- (3,1+1/2);
\draw[gray] (2,1) -- (3,1+1/3);
\draw[gray] (-2.5,2.5) -- (2,1);
\draw[gray] (2,1) -- (3,1-1/3);
\draw[gray] (2,1) -- (3,1-1/4);
\draw[gray] (2/3,1/3) -- (3,1/3);
\draw[gray] (0,1+2/3) -- (3,1+2/3);
\draw[blue,line width=1.2] (2.9,1) -- (3,1);
\draw[black!40!green,line width=2pt] (2.9,1) node[fill,circle,inner sep=1pt]{} -- (0,1) -- (0,0) -- (0.8,0.2) -- (3,0.2);
\draw[blue,line width=1.2pt] (-0.5,1.5) -- (0.8,0.2);
\draw[blue,line width=1.2pt] (-1,-0.5) -- (0,0);
\draw (0,-2) node{\Large$\leadsto$};
\draw (0,-2.8);
\end{tikzpicture}
\begin{tikzpicture}[scale=1,rotate=90]
\draw[gray] (-2,-1) -- (0,0) -- (0,1) -- (-1,2) -- (-4,3);
\draw[dashed,gray] (-2,-1) -- (3,-1);
\draw[gray] (0,0) -- (3,0);
\draw[gray] (0,1) -- (3,1) node[above]{$\rho_C$};
\draw[gray] (-1,2) -- (3,2);
\draw[dashed,gray] (-4,3) -- (3,3);
\draw[dashed,fill=black,fill opacity=0.2,gray] (-3,-1) -- (-1,-0.5) node[opacity=1,rotate=30]{$\times$} -- (-2/3,0) -- (0,0.5) node[opacity=1,rotate=0]{$\times$} -- (-1/3,1) -- (-1/2,1.5) node[opacity=1,rotate=45]{$\times$} -- (-5/3,2) -- (-2.5,2.5) node[opacity=1,rotate=70]{$\times$} -- (-4,2.8);
\draw[gray] (0,0.5) -- (0,3);
\draw[gray] (0,0.5) -- (0,-1);
\draw[gray] (-0.5,1.5) -- (2,-1);
\draw[gray] (-0.5,1.5) -- (-2,3);
\draw[gray] (-1,-0.5) -- (-2,-1);
\draw[gray] (-2.5,2.5) -- (-4,3);
\draw[gray] (-1,-0.5) -- (2,1);
\draw[gray] (2,1) -- (3,1+1/2);
\draw[gray] (2,1) -- (3,1+1/3);
\draw[gray] (-2.5,2.5) -- (2,1);
\draw[gray] (2,1) -- (3,1-1/3);
\draw[gray] (2,1) -- (3,1-1/4);
\draw[gray] (2/3,1/3) -- (3,1/3);
\draw[gray] (0,1+2/3) -- (3,1+2/3);
\draw[blue,line width=1.2pt] (-0.5,1.5) -- (0.8,0.2) -- (0,0);
\draw[blue,line width=1.2pt] (0,0.5) -- (0,0) -- (-1,-0.5);
\draw[blue,line width=1.2pt] (0.8,0.2) -- (3,0.2);
\end{tikzpicture}
\caption{An example where the tropical curve $h\in\mu^{-1}(\mathfrak{b})$ (left) goes through an affine singularity neighboring $\rho_C$.}
\label{fig:corr2}
\end{figure}
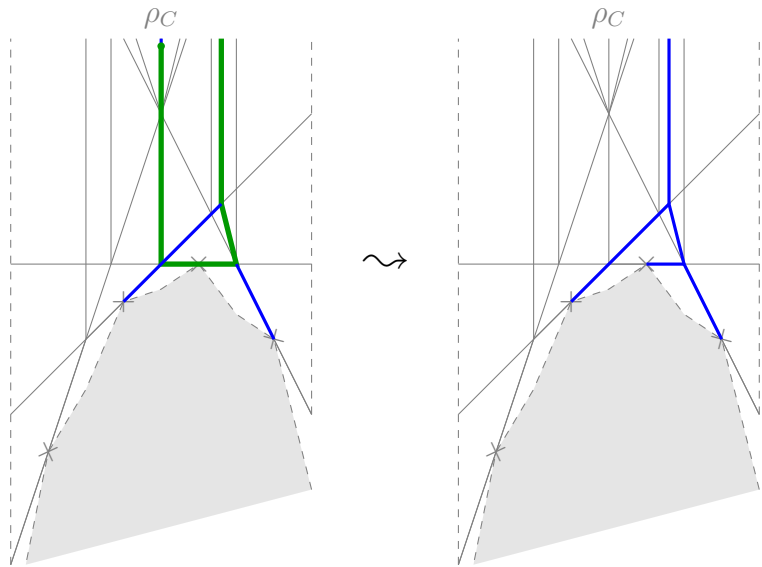

\begin{theorem}
\label{thm:blowupfan}
Let $P$ be a general point in an unbounded chamber $\mathfrak{u}$ of $\mathscr{S}_{n+1}(X)$. There is a surjective map with finite preimages (see Figures \ref{fig:corr1} and \ref{fig:corr2})
\[ \mu': \mathfrak{H}_n(\widehat{X},\pi^\star\beta-C) \rightarrow \mathfrak{B}_{1,n}(X,\beta,P\in\mathfrak{u}) \]
such that 
\[ a_{\mathfrak{b}}(\boldsymbol{q}) = \sum_{h\in{\mu'}^{-1}(\mathfrak{b})} m_h(\boldsymbol{q})\cdot \frac{i\hbar}{(\boldsymbol{q}^{1/2}-\boldsymbol{q}^{-1/2})}. \]
\end{theorem}

\begin{proof}
Consider a broken line in $\mathfrak{B}_{1,n}(X,\beta,P\in\mathfrak{u})$ and let $\mathfrak{b}$ be the corresponding broken line in $\mathfrak{B}'_{1,n}(\widehat{X},\pi^\star\beta,\widehat{P}\in\widehat{\mathfrak{u}})$ via Proposition \ref{prop:broken}. By wall crossing we can assume that the endpoint of $\mathfrak{b}$ lies on $\rho_C$. By Lemma \ref{lem:corr}, any $h\in\mu^{-1}(\mathfrak{b})$ contains the unique tropical curve $h_C$ of class $C$. It has $\boldsymbol{q}$-refined multiplicity
\[ m_{h_C}(\boldsymbol{q}) = \frac{i\hbar}{\boldsymbol{q}^{1/2}-\boldsymbol{q}^{-1/2}}. \] 
Deleting it we obtain a tropical curve in $\mathfrak{H}_n(\widehat{X},\pi^\star\beta-C)$. Conversely, to a tropical curve in $\mathfrak{H}_n(\widehat{X},\pi^\star\beta-C)$ we can add the tropical curve of class $C$ to obtain a tropical curve in $\mathfrak{H}_{1,n}(\widehat{X},\pi^\star\beta)$, which is connected by Lemma \ref{lem:intersects}, and in turn corresponds to a broken line in $\mathfrak{B}'_{1,n}(\widehat{X},\pi^\star\beta,\widehat{P}\in\widehat{\mathfrak{u}})\simeq\mathfrak{B}_{1,n}(X,\beta,P\in\mathfrak{u})$. All steps apart from deleting $h_C$ preserve $a_{\mathfrak{b}}(\boldsymbol{q})$.
\end{proof}

\begin{corollary}
\label{cor:tropcorrg}
There is a bijective map $\mathfrak{H}_{1,n}(X,\beta) \rightarrow \mathfrak{H}_n(\widehat{X},\pi^\star\beta-C), h \mapsto h'$ such that $m_h(\boldsymbol{q})=m_{h'}(\boldsymbol{q})\cdot i\hbar/(\boldsymbol{q}^{1/2}-\boldsymbol{q}^{-1/2})$. In particular,
\[ R_{1,d}^g(X,\beta) = \sum_{g_1+g_2=g} R_d^{g_1}(\widehat{X},\pi^\star\beta-C) \cdot \frac{(-1/4)^{g_2}}{(2g_2+1)!}. \]
\end{corollary}

\begin{proof}
The tropical curve of class $C$ has $\boldsymbol{q}$-refined multiplicity $m_h(\boldsymbol{q})$ with inverse
\[ \frac{1}{m_h(\boldsymbol{q})}=\frac{\boldsymbol{q}^{1/2}-\boldsymbol{q}^{-1/2}}{i\hbar} = \sum_{g\geq 0}\frac{(-1/4)^g}{(2g+1)!}\hbar^{2g} . \]
By Proposition \ref{prop:corr} and Theorem \ref{thm:blowupfan} we have
\[ \sum_{g\geq 0}R_{1,d}^g(X,\beta)\hbar^{2g} = \left(\sum_{g\geq 0}R_d^g(\widehat{X},\pi^\star\beta-C)\hbar^{2g}\right) \cdot \left(\sum_{g\geq 0}\frac{(-1/4)^g}{(2g+1)!}\hbar^{2g}\right), \]
leading to the claimed equality.
\end{proof}

\begin{remark}
Corollary \ref{cor:tropcorrg} together with the higher genus tropical correspondence theorems for $1$-marked invariants (\cite{Gra1}, Theorem 7.2) and $2$-marked invariants (Theorem~\ref{prop:trop}) gives a tropical proof of Corollary~\ref{cor-two-and-one-point}.
\end{remark}


\appendix

\section{Explicit wall-crossing computation}
\label{app:wallcrossing}

Recall the wall-crossing formula as defined in \cite{GPS}.  For each wall with function $f = 1 + c x^ay^b$, we define an automorphism of $\bC[x^{\pm 1},y^{\pm 1}][[t]]$ by $x\mapsto x\cdot f^{-b}$, $y\mapsto y\cdot f^{a}.$  (This transformation demands that we localize the ring at $f$.)
Invariantly, we consider the lattice $M:=\bZ^2$ and the group ring $\bC[M]$ generated by monomials $z^m,$ $m\in M.$  The transformation is then $z^m\mapsto z^m \cdot f^{\braket{n,m}},$ where $f = 1 + cz^r$ and $n\in M^\vee$ is a primitive normal to $r\in M$. 
To interpret the \emph{auto}morphism as an \emph{iso}morphism of charts on opposite side of a wall, we take a primitive
normal vector pointing from the new chamber to the old and make the replacement of variables $z^m\to z^m \cdot f^{\braket{n,m}}.$  Since our path will be vertical, this is why all the normal vectors below are taken with negative $y$-component.

Our plan is to track the superpotential from its value $W_0 = t(x^{-1}y + y + xy^{-2})$ in the central chamber below the horizontal line segment of the scattering diagram (see Figure \ref{fig:scat}) to its value at height infinity as we move along a vertical line, crossing walls along the way.  Note $W_0$ is equivalent to the familiar $\bP^2$ superpotential $t(x + y + x^{-1}y^{-1})$ after the affine substitution $x\to x^{-1}y, y \to y.$

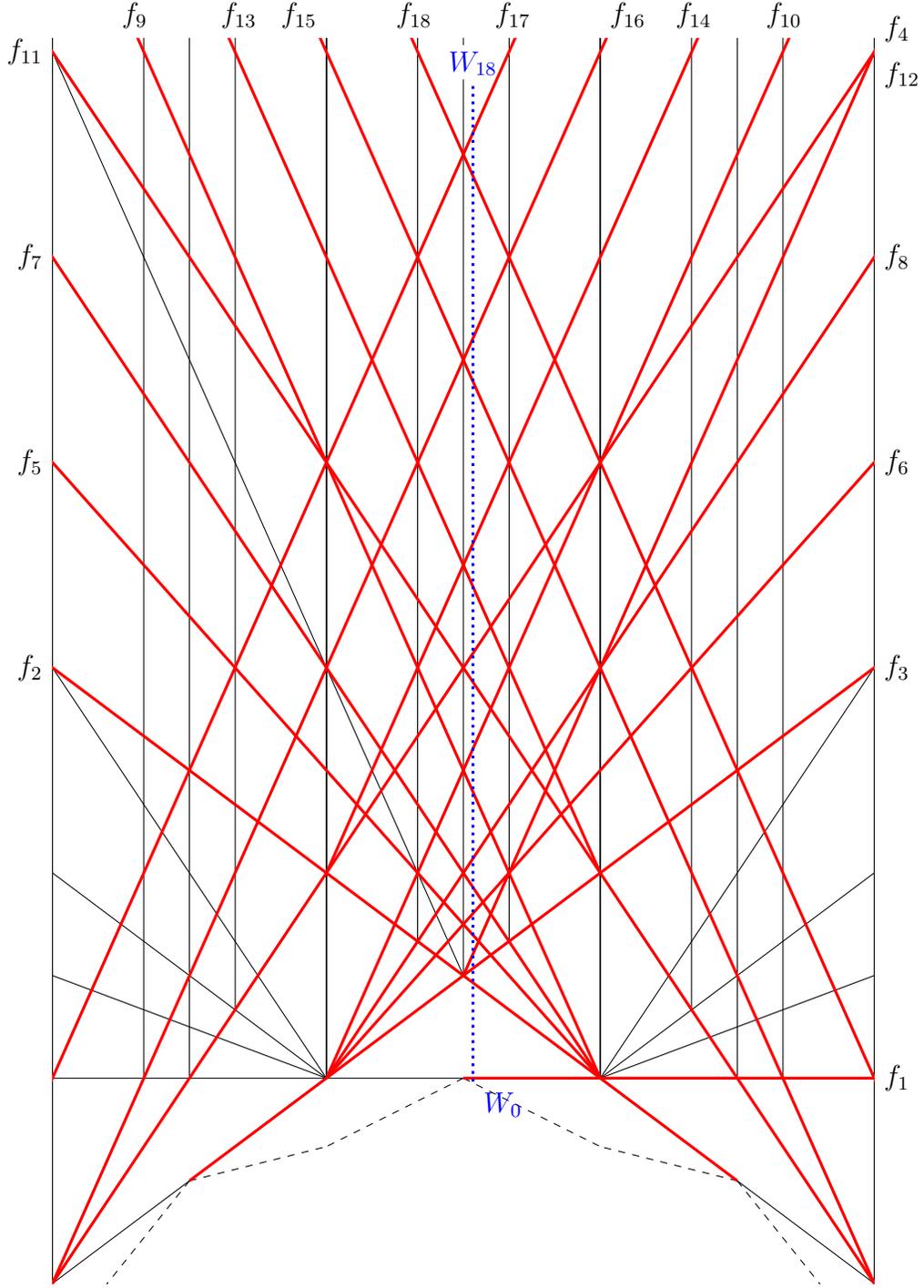
\begin{figure}[h!]
\begin{tikzpicture}[xscale=4,yscale=1,rotate=90]
\begin{scope}
\clip (-3,-1) rectangle (15.2,2);
\draw[dashed] (0.0, 0.5) -- (-1.0, 1.0);
\draw[dashed] (0.0, 0.5) -- (-1.0, 0.0);
\draw[dashed] (-1.5, 1.5) -- (-1.0, 1.0);
\draw[dashed] (-1.5, -0.5) -- (-1.0, 0.0);
\draw[dashed] (-1.5, 1.5) -- (-4.0, 2.0);
\draw[dashed] (-1.5, -0.5) -- (-4.0, -1.0);
\draw[dashed] (-6.0, 2.5) -- (-4.0, 2.0);
\draw[dashed] (-6.0, -1.5) -- (-4.0, -1.0);
\draw[dashed] (-6.0, 2.5) -- (-10.0, 3.0);
\draw[dashed] (-6.0, -1.5) -- (-10.0, -2.0);
\draw[dashed] (-13.5, 3.5) -- (-10.0, 3.0);
\draw[dashed] (-13.5, -2.5) -- (-10.0, -2.0);
\draw[-] (-6.0, -1.5) -- (-12.0, -2.5);
\draw[-] (-6.0, 2.5) -- (-12.0, 3.5);
\draw[-] (-1.5, -0.5) -- (-7.5, -2.5);
\draw[-] (-1.5, 1.5) -- (-7.5, 3.5);
\draw[-] (0.0, 0.5) -- (0.0, -2.5);
\draw[-] (-3.0, -1.0) -- (-3.0, -2.5);
\draw[-] (0.0, 0.5) -- (0.0, 3.5);
\draw[-] (-3.0, 2.0) -- (-3.0, 3.5);
\draw[-] (0.0, 0.0) -- (30.0, 0.0);
\draw[-] (-3.0, -1.0) -- (30.0, -1.0);
\draw[-] (-9.0, -2.0) -- (30.0, -2.0);
\draw[-] (-9.0, 3.0) -- (30.0, 3.0);
\draw[-] (-3.0, 2.0) -- (30.0, 2.0);
\draw[-] (0.0, 1.0) -- (30.0, 1.0);
\draw[-] (0.0, 1.5) -- (30.0, 1.5);
\draw[-] (0.0, -0.5) -- (30.0, -0.5);
\draw[-] (1.5, 0.5) -- (14.6, 0.5);
\draw[-] (15.1, 0.5) -- (16.0, 0.5);
\draw[-] (0.0, 0.0) -- (30.0, 0.0);
\draw[-] (-3.0, 2.0) -- (30.0, 2.0);
\draw[-] (-3.0, -1.0) -- (30.0, -1.0);
\draw[-] (0.0, 1.0) -- (30.0, 1.0);
\draw[-] (0.0, 0.0) -- (30.0, 0.0);
\draw[-] (0.0, 1.0) -- (30.0, 1.0);
\draw[-] (3.0, 0.0) -- (30.0, 0.0);
\draw[-] (3.0, 1.0) -- (30.0, 1.0);
\draw[-] (2.0, 0.333) -- (30.0, 0.333);
\draw[-] (2.0, 0.667) -- (30.0, 0.667);
\draw[-] (0.0, 0.0) -- (3.75, -2.5);
\draw[-] (-1.5, 1.5) -- (10.5, -2.5);
\draw[-] (0.0, 0.0) -- (7.5, -2.5);
\draw[-] (-1.5, -0.5) -- (10.5, 3.5);
\draw[-] (0.0, 1.0) -- (7.5, 3.5);
\draw[-] (0.0, 1.0) -- (3.75, 3.5);
\draw[-] (-6.0, 2.5) -- (24.0, -2.5);
\draw[-] (0.0, 1.0) -- (21.0, -2.5);
\draw[-] (0.0, 0.0) -- (15.0, -2.5);
\draw[-] (-6.0, -1.5) -- (24.0, 3.5);
\draw[-] (0.0, 0.0) -- (21.0, 3.5);
\draw[-] (0.0, 1.0) -- (15.0, 3.5);
\draw[-] (0.0, 1.0) -- (15.75, -2.5);
\draw[-] (-13.5, 3.5) -- (30.0, -1.333);
\draw[-] (-3.0, 2.0) -- (30.0, -1.667);
\draw[-] (0.0, 1.0) -- (30.0, -2.333);
\draw[-] (-13.5, -2.5) -- (30.0, 2.333);
\draw[-] (-3.0, -1.0) -- (30.0, 2.667);
\draw[-] (0.0, 0.0) -- (30.0, 3.333);
\draw[-] (0.0, 0.0) -- (15.75, 3.5);
\draw[-] (1.5, 0.5) -- (30.0, 3.667);
\draw[-] (0, 1.667) -- (30.0, 1.667);
\draw[-] (0, -0.667) -- (30.0, -0.667);
\draw[-] (1.0, 1.333) -- (30.0, 1.333);
\draw[-] (1.0, -0.333) -- (30.0, -0.333);
\draw[very thick, red] (0.0, 0.5) -- (0.0, -5.0);
\draw[very thick, red] (-1.5, 1.5) -- (18.0, -5.0);
\draw[very thick, red] (-1.5, -0.5) -- (18.0, 6.0);
\draw[very thick, red] (-6.0, 2.5) -- (30.0, -3.5);
\draw[very thick, red] (-0.0, 1.0) -- (30.0, -4.0);
\draw[very thick, red] (-6.0, -1.5) -- (30.0, 4.5);
\draw[very thick, red] (-0.0, -0.0) -- (30.0, 5.0);
\draw[very thick, red] (-0.0, 1.0) -- (27.0, -5.0);
\draw[very thick, red] (-3.0, 2.0) -- (30.0, -1.667);
\draw[very thick, red] (-0.0, 1.0) -- (30.0, -2.333);
\draw[very thick, red] (3.0, 1.0) -- (30.0, -2.0);
\draw[very thick, red] (1.5, 0.5) -- (30.0, -2.667);
\draw[very thick, red] (-3.0, -1.0) -- (30.0, 2.667);
\draw[very thick, red] (-0.0, -0.0) -- (30.0, 3.333);
\draw[very thick, red] (3.0, 0.0) -- (30.0, 3.0);
\draw[very thick, red] (-0.0, -0.0) -- (27.0, 6.0);
\draw[very thick, red] (-13.5, 3.5) -- (30.0, -1.333); 
\draw[very thick, red] (-13.5, -2.5) -- (30.0, 2.333);
\end{scope}
\draw (0,-1) node [right]{$f_1$};
\draw (6,2) node[left]{$f_2$};
\draw (6,-1) node [right]{$f_3$};
\draw (15,-1) node [above right]{$f_4$};
\draw (9,2) node[left]{$f_5$};
\draw (9,-1) node [right]{$f_6$};
\draw (12,2) node[left]{$f_7$};
\draw (12,-1) node [right]{$f_8$};
\draw (15.2,1.7) node[above]{$f_9$};
\draw (15.2,-0.67) node[above]{$f_{10}$};
\draw (15,2) node[left]{$f_{11}$};
\draw (15,-1) node [below right]{$f_{12}$};
\draw (15.2,1.32) node[above]{$f_{13}$};
\draw (15.2,-0.34) node[above]{$f_{14}$};
\draw (15.2,1.1) node[above]{$f_{15}$};
\draw (15.2,-0.1) node[above]{$f_{16}$};
\draw (15.2,0.32) node[above]{$f_{17}$};
\draw (15.2,0.68) node[above]{$f_{18}$};
\draw[blue,dotted,very thick] (-.05,.465) node[below right]{$W_0$} -- (14.5,.465) node[above] {$W_{18}$};
\end{tikzpicture}
\caption{The scattering diagram to order $t^{9}$, with relevant walls in red.  We will track the superpotential along the dotted blue line,
from $W_0 = t(x^{-1}y + y + xy^{-2})$ at the bottom to $W_{18}$ at the top, crossing $18$ walls along the way.}
\label{fig:scat}
\end{figure}

We work to order $t^{9}$.  The $t$ order can be no less than the absolute value of the slope of the line, and since these values increase as you move up (a consequence of how new rays are spawned), there are only finitely many walls to consider at any fixed order.

Let us label the $18$ walls to cross by wall functions $f_1, ..., f_{18}$ and normal vectors $n_1, ...,n_{18}.$
The first wall that we cross is special, since it is a slab:  the wall function is proportional to $t$ rather than being $1+O(t)$.  
We have $f_1 = t(1+x^{-1})$ and normal vector $n_1 = (0,-1).$
Let's cross it.  The monomial $x^{-1}y = z^{(-1,1)}$ transforms to $x^{-1}y f_1^{-1}.$ 
Likewise $y\to yf_1^{-1}$, so their sum goes to $t^{-1}y.$  Then $xy^{-2}\to xy^{-2}f_1^2.$  In total, we have
$W_1 = y + t^3 x^{-1}y^{-2} + 2t^3 y^{-2} + t^3 x y^{-2}.$

Here is a list of the 18 wall functions and normal vectors, using the Sage code developed by the first author:
\begin{equation*}
\begin{array}{llcll}
f_1 = t(1+x^{-1})&n_1 = (0,-1) && f_{10} = 1 + 13t^9x^{-1}y^{-9}& n_{10}=(9,-1)\\
f_2 = 1+t^3xy^{-3}&n_2 = (-3,-1) && f_{11} = 1 + t^6xy^{-6}&n_{11}= (-6,-1)\\
f_3 = 1+t^3x^{-1}y^{-3} &n_3 = (3,-1) &&  f_{12} = 1 +  t^6x^{-1}y^{-6}&n_{12} = (6,-1)\\
f_4 = 1+15t^9x^{-1}y^{-9}& n_4 = (9,-1) &&  f_{13} = 1 + 9t^9xy^{-9} & n_{13} = (-9,-1)\\
f_5 = 1+t^9x^{2}y^{-9} &n_5 = (-9,-2) && f_{14} = 1 +9t^9x^{-1}y^{-9}  &n_{14} = (9,-1)\\
f_6 = 1+t^9x^{-2}y^{-9}&n_6=(9,-2) &&f_{15} = 1 +  3t^9xy^{-9} &n_{15} = (-9,-1)\\
f_7 = 1 + 3t^6xy^{-6} & n_7 = (-6,-1) && f_{16} = 1 +3t^9x^{-1}y^{-9}  &n_{16} = (9,-1) \\
f_8 = 1+3t^6x^{-1}y^{-6}&n_8 = (6,-1) && f_{17} = 1 +t^9xy^{-9} &n_{17} = (-9,-1)\\
f_9 = 1 + 13t^9xy^{-9} & n_9 = (-9,-1) && f_{18} = 1 + t^9x^{-1}y^{-9} &n_{18} = (9,-1)\\
\end{array}
\end{equation*} 

With these in hand, we can program a computer to compute the superpotential along the blue line.  Doing
so yields $W_{18},$ to order $t^{12},$ to be
$$W_{18} = y + 2t^{3}y^{-2} + 5t^6y^{-5} + 32t^9y^{-8} + 286t^{12}y^{-11} + \cdots,$$
in agreement with Theorem \ref{thm-main}. Here the omitted terms $\cdots$ are {either} $O(t^{14})$ {or} $O(t^{11})$ but with nontrivial $x$ coefficient; including wall functions to order $t^{12}$ will remove these latter terms.

If we are only interested in ``pure $y$ terms,'' i.e.~terms of the form $cy^a$ (no $x$s), there is a simplification that allows us to compute the terms in $W_{18}$ above by hand:  a pure $y$ term must cross at least \emph{two} walls to generate a new pure $y$ term.  Consider the quantity $W_3$, which is $W_0$ after the transformation of three walls.  Note all of the further walls to cross come with wall functions of the form $f = 1 + O(t^6).$  As a result, a term in $W_3$ such as $2t^3 y^{-2}$ is ``stable'' if we are working to order $t^{12}$, since it would have to cross at least two walls to generate a new pure $y$ term, which would then appear with order at least $t^{15}.$

Let us therefore calculate the pure $y$ terms of $W_{18}.$  
We have from above $W_1 = y + t^3 x^{-1}y^{-2} + 2t^3 y^{-2} + t^3 x y^{-2}.$
Crossing the second wall gives
$W_2 = (y+t^3xy^{-2})f_2^{-1} + 2t^3y^{-2}f_2^2 + t^3x^{-1}y^{-2}f_2^5.$
The first term is $y$ and the next terms can be expanded to
$$W_2 = \underline{y} + {2t^3y^{-2}} + 4t^6xy^{-5} + 2t^9x^{2}y^{-8} + \underline{t^3 x^{-1}y^{-2}} + {5t^6y^{-5}} + {10t^9xy^{-8}} +
\red{10t^{12}x^2y^{-11}}$$
to order $t^{12}.$
The term in red cannot affect a pure $y$ term since it would have to cross a wall with an $x^{-1}$ term, and this will bring the $t$ power
above $12$ --- we henceforth drop it. 
Crossing the next wall with $f_3 = 1 + t^{3}x^{-1}y^{-3}$ and $n_3 = (3,-1),$ the underlined terms arrange to cancel $f_3^{-1}$ giving $y$, so
$$W_3 = y
+ 2t^3y^{-2}f_3^2 + 4t^6xy^{-5}f_3^8 + {2t^9x^2y^{-8}f_3^{14}} + 5t^6y^{-5}f_3^5 + {10t^9xy^{-8}f_3^{11}} + ...,$$
and the ellipses include terms which will have no effect on pure $y$ terms to order $t^{12}.$  As we expand, note that no future wall carries a power of $t$ less than $6$, so: 1) we can drop any non-pure $y$
term with power $t^9$; and 2) any pure $y$ terms with order $t^3$ or higher are stable in the sense that they cannot
affect other pure $y$ terms to order $t^{12}$.  This gives

$$W_3 = y + 4t^6x^{-1}y^{-5} + 4t^6xy^{-5} + \underline{2t^3y^{-2} + 5t^6y^{-5} + 32t^9y^{-8}+11\cdot 10t^{12}y^{-11}} + ...,$$
where the underlined terms are stable.
The $4t^6x^{-1}y^{-5}$ term will contribute to the term $t^{12}y^{-11}$ when it crosses Wall $7$ with exponent $(-1,-5)\cdot (-6,-1) = 11,$ thus picking up a factor of $3\cdot 11$ from the linear term, yielding $4\cdot 3\cdot 11t^{12}y^{-11}.$  Likewise when it crosses Wall $12$ it gets a factor $4\cdot 11t^{12}y^{-11}.$
 Similarly, we get another $11(4\cdot 3+4)t^{12}y^{-11}$ from the $4t^6xy^{-5}$ term. 
The $y$ term will contribute when it crosses two $t^6$ walls with an $x$ and an $x^{-1}$, so Walls $7$ and $8$, $7$ and $13$, $8$ and $12$
and $8$ and $13$.  Since $(0,1)\cdot n$ is always $-1$ for all walls that our path crosses, there will be negative terms in the expansion of the power of the wall function.  From these four pairs of crossings the $y$ term generates terms of the form $t^{12}y^{-11}$ with coefficients,
respectively, $-9\cdot 11, -3\cdot 11, -3\cdot 11, -1\cdot 11.$  Therefore the coefficient of $t^{12}y^{-12}$ in $tW_{18}$ is
$11\left[(10 + 2(4\cdot 3+4) - (9+3+3+1)\right] = 286.$  So
$$W_{18} =  y + 2t^{3}y^{-2} + 5t^6y^{-5} + 32t^9y^{-8} + 286t^{12}y^{-11} + ...$$
as claimed.


\section{Explicit broken line counting}
\label{app:brokenlines}

We use a Sage code to calculate the numbers $R_{1,3d-1}(\mathbb{P}^2,dL)$ for $d\leq 4$. It can be found on \href{timgraefnitz.com}{timgraefnitz.com}. In the code, broken lines are implemented in reversed order.
We start with point $P$ on an unbounded maximal cell of $\mathscr{S}_d(\mathbb{P}^2)$ and a broken line coming out of this point in the negative of the unique unbounded direction $m_{\text{out}}$, with attached monomial $z^{qm_{\text{out}}}$. We can do this, since all broken lines that end in $P$ have to be parallel to $m_{\text{out}}$.
When the broken line hits a wall, we apply the transformation $z^m \mapsto f^{n\cdot m}z^m$, where $n$ is the normal direction of the wall. Each term in $f^{n\cdot m}z^m$ gives a new possible broken line. The above procedure is applied recursively until the direction of the new broken line is $m_{\text{out}}$. Then we have found a broken line with asymptotic monomial $z^{qm_{\text{out}}}$ and ending in $P$.
If we add together the coefficients $a_{\mathfrak{b}}$ of the broken lines $\mathfrak{b}$ with asymptotic monomial $z^{qm_{\text{out}}}$ and resulting monomial $a_{\mathfrak{b}}z^{-pm_{\text{out}}}$ we get the tropical count $R_{p,q}^{\text{trop}}(\mathbb{P}^2,dL)$, where $d=(p+q)/3$. Using the tropical correspondence (Theorem~\ref{prop:trop}) we obtain the log invariants $R_{p,q}(\mathbb{P}^2,dL)$.

\subsection{Degree $1$}

Figure \ref{fig:sage1} shows the broken lines the code produces for $d=1$.

\begin{figure}[h!]
\centering
\begin{tikzpicture}[xscale=1.6,yscale=.7,rotate=90]
\clip (-3,-1) rectangle (6.6,2);
\draw[dashed] (0.0, 0.5) -- (-1.0, 1.0);
\draw[dashed] (0.0, 0.5) -- (-1.0, 0.0);
\draw[dashed] (-1.5, -0.5) -- (-1.0, 0.0);
\draw[dashed] (-1.5, 1.5) -- (-1.0, 1.0);
\draw[dashed] (-1.5, 1.5) -- (-4.0, 2.0);
\draw[dashed] (-1.5, -0.5) -- (-4.0, -1.0);
\draw[dashed] (-6.0, -1.5) -- (-4.0, -1.0);
\draw[dashed] (-6.0, 2.5) -- (-4.0, 2.0);
\draw[gray] (0.0, 0.5) node[fill,cross,inner sep=2pt,rotate=90.0]{} -- (0.0, 2.5);
\draw[gray] (0.0, 0.5) node[fill,cross,inner sep=2pt,rotate=-90.0]{} -- (0.0, -1.5);
\draw[gray] (-1.5, 1.5) node[fill,cross,inner sep=2pt,rotate=-18.435]{} -- (6.0, -1.0);
\draw[gray] (-1.5, -0.5) node[fill,cross,inner sep=2pt,rotate=18.435]{} -- (6.0, 2.0);
\draw[gray] (-1.5, 1.5) node[fill,cross,inner sep=2pt,rotate=161.565]{} -- (-4.5, 2.5);
\draw[gray] (-1.5, -0.5) node[fill,cross,inner sep=2pt,rotate=-161.565]{} -- (-4.5, -1.5);
\draw[gray] (-6.0, 2.5) node[fill,cross,inner sep=2pt,rotate=-9.462]{} -- (6.0, 0.5);
\draw[gray] (-6.0, -1.5) node[fill,cross,inner sep=2pt,rotate=9.462]{} -- (6.0, 0.5);
\draw[gray] (-3.0, 2.0) -- (6.0, 2.0);
\draw[gray] (0.0, 0.0) -- (6.0, 0.0);
\draw[gray] (-3.0, -1.0) -- (6.0, -1.0);
\draw[gray] (0.0, 1.0) -- (6.0, 1.0);
\draw[gray] (1.5, 0.5) -- (6.0, 0.5);
\draw[gray] (3.0, 0.0) -- (6.0, -0.5);
\draw[gray] (3.0, 1.0) -- (6.0, 1.5);
\draw[gray] (-0.0, -0.0) -- (4.5, -1.5);
\draw[gray] (-0.0, -0.0) -- (6.0, -0.0);
\draw[gray] (-3.0, 2.0) -- (-3.0, 2.5);
\draw[gray] (-0.0, -0.0) -- (6.0, 1.0);
\draw[gray] (-0.0, 1.0) -- (4.5, 2.5);
\draw[gray] (-0.0, 1.0) -- (6.0, 0.0);
\draw[gray] (-0.0, 1.0) -- (6.0, 1.0);
\draw[gray] (-3.0, -1.0) -- (-3.0, -1.5);
\fill[black!40!green] (5.9, 0.382) ellipse (3pt and 1.4pt);
\draw[black!40!green,line width=2pt] (5.9, 0.382) -- (0.0, 0.382);
\draw[black!40!green,line width=2pt] (0.0, 0.382) -- (-0.573, -0.191);
\draw[black!40!green,line width=2pt] (-0.573, -0.191) -- (6, -0.191) node[right]{$2$};
\fill[black!40!green] (5.9, 0.382) ellipse (3pt and 1.4pt);
\draw[black!40!green,line width=2pt] (5.9, 0.382) -- (-0.236, 0.382);
\draw[black!40!green,line width=2pt] (-0.236, 0.618) -- (-0.927, 1.309);
\draw[black!40!green,line width=2pt] (-0.927, 1.309) -- (6, 1.309) node[left]{$2$};
\end{tikzpicture}
\hspace{10mm}
\begin{tikzpicture}[xscale=1.6,yscale=.7,rotate=90]
\clip (-3,-1) rectangle (6.6,2);
\draw[dashed] (0.0, 0.5) -- (-1.0, 1.0);
\draw[dashed] (0.0, 0.5) -- (-1.0, 0.0);
\draw[dashed] (-1.5, -0.5) -- (-1.0, 0.0);
\draw[dashed] (-1.5, 1.5) -- (-1.0, 1.0);
\draw[dashed] (-1.5, 1.5) -- (-4.0, 2.0);
\draw[dashed] (-1.5, -0.5) -- (-4.0, -1.0);
\draw[dashed] (-6.0, -1.5) -- (-4.0, -1.0);
\draw[dashed] (-6.0, 2.5) -- (-4.0, 2.0);
\draw[gray] (0.0, 0.5) node[fill,cross,inner sep=2pt,rotate=90.0]{} -- (0.0, 2.5);
\draw[gray] (0.0, 0.5) node[fill,cross,inner sep=2pt,rotate=-90.0]{} -- (0.0, -1.5);
\draw[gray] (-1.5, 1.5) node[fill,cross,inner sep=2pt,rotate=-18.435]{} -- (6.0, -1.0);
\draw[gray] (-1.5, -0.5) node[fill,cross,inner sep=2pt,rotate=18.435]{} -- (6.0, 2.0);
\draw[gray] (-1.5, 1.5) node[fill,cross,inner sep=2pt,rotate=161.565]{} -- (-4.5, 2.5);
\draw[gray] (-1.5, -0.5) node[fill,cross,inner sep=2pt,rotate=-161.565]{} -- (-4.5, -1.5);
\draw[gray] (-6.0, 2.5) node[fill,cross,inner sep=2pt,rotate=-9.462]{} -- (6.0, 0.5);
\draw[gray] (-6.0, -1.5) node[fill,cross,inner sep=2pt,rotate=9.462]{} -- (6.0, 0.5);
\draw[gray] (-3.0, 2.0) -- (6.0, 2.0);
\draw[gray] (0.0, 0.0) -- (6.0, 0.0);
\draw[gray] (-3.0, -1.0) -- (6.0, -1.0);
\draw[gray] (0.0, 1.0) -- (6.0, 1.0);
\draw[gray] (1.5, 0.5) -- (6.0, 0.5);
\draw[gray] (3.0, 0.0) -- (6.0, -0.5);
\draw[gray] (3.0, 1.0) -- (6.0, 1.5);
\draw[gray] (-0.0, -0.0) -- (4.5, -1.5);
\draw[gray] (-0.0, -0.0) -- (6.0, -0.0);
\draw[gray] (-3.0, 2.0) -- (-3.0, 2.5);
\draw[gray] (-0.0, -0.0) -- (6.0, 1.0);
\draw[gray] (-0.0, 1.0) -- (4.5, 2.5);
\draw[gray] (-0.0, 1.0) -- (6.0, 0.0);
\draw[gray] (-0.0, 1.0) -- (6.0, 1.0);
\draw[gray] (-3.0, -1.0) -- (-3.0, -1.5);
\fill[black!40!green] (5.9, 0.382) ellipse (3pt and 1.4pt);
\draw[black!40!green,line width=2pt] (5.9, 0.382) -- (0.0, 0.382);
\draw[black!40!green,line width=2pt] (0.0, 0.382) -- (-1.236,-0.236);
\draw[black!40!green,line width=2pt] (-2.82, -0.764) -- (6, -0.764) node[left]{$1$};
\end{tikzpicture}
\caption{Broken lines used to compute $R_{1,2}(\mathbb{P}^2,L)$ and $R_{2,1}(\mathbb{P}^2,L)$. The numbers above the infinite segments are the genus $0$ coefficients of the broken lines.}
\label{fig:sage1}
\end{figure}
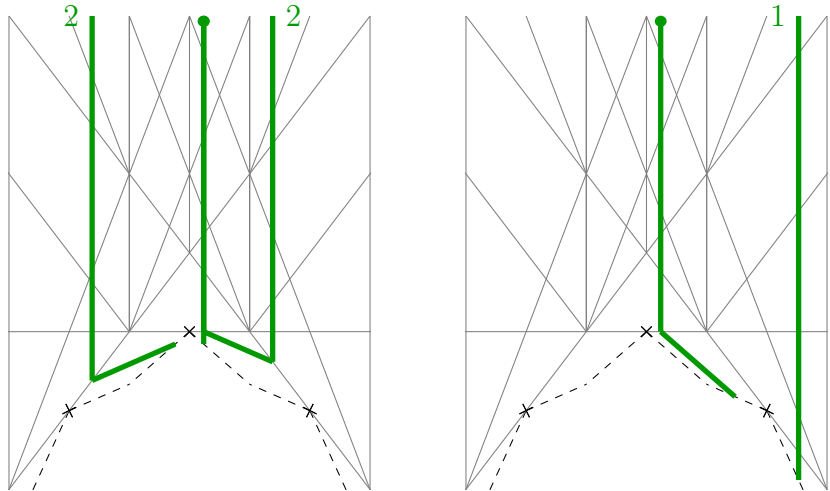

For $(p,q)=(1,2)$ there are two broken lines, giving two tropical curves. One has one vertex of multiplicity $2$ and one bounded leg of weight $1$. The other one has two vertices of multiplicity $2,1$ and two bounded legs of weights $1,1$. Both have multiplicity
\[ m_h(\boldsymbol{q}) = \frac{\boldsymbol{q}-\boldsymbol{q}^{-1}}{\boldsymbol{q}^{1/2}-\boldsymbol{q}^{-1/2}} = \boldsymbol{q}^{1/2}+\boldsymbol{q}^{-1/2} \]
Using $\boldsymbol{q}=e^{i\hbar}$ we see
\[ m_h(\boldsymbol{q}) = \sum_{g\geq 0} \frac{2(-1)^g}{(2g)!2^{2g}}\hbar^{2g}. \]
There are two such curves, hence
\[ R_{1,2}^g = \frac{4(-1)^g}{(2g)!2^{2g}}. \]
The first numbers are:

\vspace{3mm}
\begin{center}
\renewcommand{\arraystretch}{1.2}
\begin{tabular}{|c|c|c|c|c|} 								\hline
$R_{1,2}^0=4$		& $R_{1,2}^1=-\frac{1}{2}$		& $R_{1,2}^2=\frac{1}{1536}$		& $R_{1,2}^3=-\frac{1}{184320}$		& $R_{1,2}^4=\frac{1}{10321920}$		\\ \hline
\end{tabular}
\end{center}
\vspace{3mm}

For $(p,q)=(2,1)$ there is one tropical curve, having the same multiplicity as above. Moreover, we have to divide by $p=2$ to get the non-tropical count $R_{2,1}^g$. This gives
\[ R_{2,1}^g = \frac{(-1)^g}{(2g)!2^{2g}}. \]
The first numbers are:

\vspace{3mm}
\begin{center}
\renewcommand{\arraystretch}{1.2}
\begin{tabular}{|c|c|c|c|c|} 								\hline
$R_{2,1}^0=1$	& $R_{2,1}^1=-\frac{1}{8}$		& $R_{2,1}^2=\frac{1}{384}$		& $R_{2,1}^3=-\frac{1}{46080}$		& $R_{2,1}^4=\frac{1}{2580480}$	\\ \hline
\end{tabular}
\end{center}
\vspace{3mm}

\subsection{Degree $2$}

Figure \ref{fig:sage2} shows the broken lines the code produces for $d=2$.
\begin{figure}[h!]
\centering
\begin{tikzpicture}[xscale=1.6,yscale=.7,rotate=90]
\clip (-3,-1) rectangle (6.6,2);
\draw[dashed,gray] (0.0, 0.5) -- (-1.0, 1.0);
\draw[dashed,gray] (0.0, 0.5) -- (-1.0, 0.0);
\draw[dashed,gray] (-1.5, -0.5) -- (-1.0, 0.0);
\draw[dashed,gray] (-1.5, 1.5) -- (-1.0, 1.0);
\draw[dashed,gray] (-1.5, 1.5) -- (-4.0, 2.0);
\draw[dashed,gray] (-1.5, -0.5) -- (-4.0, -1.0);
\draw[dashed,gray] (-6.0, -1.5) -- (-4.0, -1.0);
\draw[dashed,gray] (-6.0, 2.5) -- (-4.0, 2.0);
\draw[gray] (0.0, 0.5) node[fill,cross,inner sep=2pt,rotate=90.0]{} -- (0.0, 2.5);
\draw[gray] (0.0, 0.5) node[fill,cross,inner sep=2pt,rotate=-90.0]{} -- (0.0, -1.5);
\draw[gray] (-1.5, 1.5) node[fill,cross,inner sep=2pt,rotate=-18.435]{} -- (6.0, -1.0);
\draw[gray] (-1.5, -0.5) node[fill,cross,inner sep=2pt,rotate=18.435]{} -- (6.0, 2.0);
\draw[gray] (-1.5, 1.5) node[fill,cross,inner sep=2pt,rotate=161.565]{} -- (-4.5, 2.5);
\draw[gray] (-1.5, -0.5) node[fill,cross,inner sep=2pt,rotate=-161.565]{} -- (-4.5, -1.5);
\draw[gray] (-6.0, 2.5) node[fill,cross,inner sep=2pt,rotate=-9.462]{} -- (6.0, 0.5);
\draw[gray] (-6.0, -1.5) node[fill,cross,inner sep=2pt,rotate=9.462]{} -- (6.0, 0.5);
\draw[gray] (-3.0, 2.0) -- (6.0, 2.0);
\draw[gray] (0.0, 0.0) -- (6.0, 0.0);
\draw[gray] (-3.0, -1.0) -- (6.0, -1.0);
\draw[gray] (0.0, 1.0) -- (6.0, 1.0);
\draw[gray] (1.5, 0.5) -- (6.0, 0.5);
\draw[gray] (3.0, 0.0) -- (6.0, -0.5);
\draw[gray] (3.0, 1.0) -- (6.0, 1.5);
\draw[gray] (-0.0, -0.0) -- (4.5, -1.5);
\draw[gray] (-0.0, -0.0) -- (6.0, -0.0);
\draw[gray] (-3.0, 2.0) -- (-3.0, 2.5);
\draw[gray] (-0.0, -0.0) -- (6.0, 1.0);
\draw[gray] (-0.0, 1.0) -- (4.5, 2.5);
\draw[gray] (-0.0, 1.0) -- (6.0, 0.0);
\draw[gray] (-0.0, 1.0) -- (6.0, 1.0);
\draw[gray] (-3.0, -1.0) -- (-3.0, -1.5);
\fill[black!40!green] (5.9, 0.382) ellipse (3pt and 1.4pt);
\draw[black!40!green,line width=2pt] (5.9, 0.382) -- (1.146, 0.382);
\draw[black!40!green,line width=2pt] (1.146, 0.382) -- (1.429, 0.524);
\draw[black!40!green,line width=2pt] (1.429, 0.524) -- (6, 0.524) node[left,xshift=2pt]{$5$};
\fill[black!40!green] (5.9, 0.382) ellipse (3pt and 1.4pt);
\draw[black!40!green,line width=2pt] (5.9, 0.382) -- (0.0, 0.382);
\draw[black!40!green,line width=2pt] (0.0, 0.382) -- (-0.573, -0.191);
\draw[black!40!green,line width=2pt] (-0.573, -0.191) -- (0.0, -0.076);
\draw[black!40!green,line width=2pt] (0.0, -0.076) -- (6, -0.076) node[right,xshift=-2pt]{$5$};
\fill[black!40!green] (5.9, 0.382) ellipse (3pt and 1.4pt);
\draw[black!40!green,line width=2pt] (5.9, 0.382) -- (0.0, 0.382);
\draw[black!40!green,line width=2pt] (0.0, 0.382) -- (-1.058, -0.676);
\draw[black!40!green,line width=2pt] (-1.058, -0.676) -- (6, -0.676) node[right,xshift=-2pt]{$5$};
\fill[black!40!green] (5.9, 0.382) ellipse (3pt and 1.4pt);
\draw[black!40!green,line width=2pt] (5.9, 0.382) -- (-0.236, 0.382);
\draw[black!40!green,line width=2pt] (-0.236, 0.618) -- (-0.927, 1.309);
\draw[black!40!green,line width=2pt] (-0.927, 1.309) -- (0.0, 1.124);
\draw[black!40!green,line width=2pt] (0.0, 1.124) -- (6, 1.124) node[left,xshift=2pt]{$5$};
\fill[black!40!green] (5.9, 0.382) ellipse (3pt and 1.4pt);
\draw[black!40!green,line width=2pt] (5.9, 0.382) -- (-0.236, 0.382);
\draw[black!40!green,line width=2pt] (-0.236, 0.618) -- (-1.342, 1.724);
\draw[black!40!green,line width=2pt] (-1.342, 1.724) -- (6, 1.724) node[left,xshift=2pt]{$5$};
\end{tikzpicture}
\hspace{2mm}
\begin{tikzpicture}[xscale=1.6,yscale=.7,rotate=90]
\clip (-3,-1) rectangle (6.6,2);
\draw[dashed,gray] (0.0, 0.5) -- (-1.0, 1.0);
\draw[dashed,gray] (0.0, 0.5) -- (-1.0, 0.0);
\draw[dashed,gray] (-1.5, -0.5) -- (-1.0, 0.0);
\draw[dashed,gray] (-1.5, 1.5) -- (-1.0, 1.0);
\draw[dashed,gray] (-1.5, 1.5) -- (-4.0, 2.0);
\draw[dashed,gray] (-1.5, -0.5) -- (-4.0, -1.0);
\draw[dashed,gray] (-6.0, -1.5) -- (-4.0, -1.0);
\draw[dashed,gray] (-6.0, 2.5) -- (-4.0, 2.0);
\draw[gray] (0.0, 0.5) node[fill,cross,inner sep=2pt,rotate=90.0]{} -- (0.0, 2.5);
\draw[gray] (0.0, 0.5) node[fill,cross,inner sep=2pt,rotate=-90.0]{} -- (0.0, -1.5);
\draw[gray] (-1.5, 1.5) node[fill,cross,inner sep=2pt,rotate=-18.435]{} -- (6.0, -1.0);
\draw[gray] (-1.5, -0.5) node[fill,cross,inner sep=2pt,rotate=18.435]{} -- (6.0, 2.0);
\draw[gray] (-1.5, 1.5) node[fill,cross,inner sep=2pt,rotate=161.565]{} -- (-4.5, 2.5);
\draw[gray] (-1.5, -0.5) node[fill,cross,inner sep=2pt,rotate=-161.565]{} -- (-4.5, -1.5);
\draw[gray] (-6.0, 2.5) node[fill,cross,inner sep=2pt,rotate=-9.462]{} -- (6.0, 0.5);
\draw[gray] (-6.0, -1.5) node[fill,cross,inner sep=2pt,rotate=9.462]{} -- (6.0, 0.5);
\draw[gray] (-3.0, 2.0) -- (6.0, 2.0);
\draw[gray] (0.0, 0.0) -- (6.0, 0.0);
\draw[gray] (-3.0, -1.0) -- (6.0, -1.0);
\draw[gray] (0.0, 1.0) -- (6.0, 1.0);
\draw[gray] (1.5, 0.5) -- (6.0, 0.5);
\draw[gray] (3.0, 0.0) -- (6.0, -0.5);
\draw[gray] (3.0, 1.0) -- (6.0, 1.5);
\draw[gray] (-0.0, -0.0) -- (4.5, -1.5);
\draw[gray] (-0.0, -0.0) -- (6.0, -0.0);
\draw[gray] (-3.0, 2.0) -- (-3.0, 2.5);
\draw[gray] (-0.0, -0.0) -- (6.0, 1.0);
\draw[gray] (-0.0, 1.0) -- (4.5, 2.5);
\draw[gray] (-0.0, 1.0) -- (6.0, 0.0);
\draw[gray] (-0.0, 1.0) -- (6.0, 1.0);
\draw[gray] (-3.0, -1.0) -- (-3.0, -1.5);
\fill[black!40!green] (5.9, 0.382) ellipse (3pt and 1.4pt);
\draw[black!40!green,line width=2pt] (5.9, 0.382) -- (1.146, 0.382);
\draw[black!40!green,line width=2pt] (1.146, 0.382) -- (1.323, 0.559);
\draw[black!40!green,line width=2pt] (1.323, 0.559) -- (6, 0.559) node[left,xshift=2pt]{$8$};
\fill[black!40!green] (5.9, 0.382) ellipse (3pt and 1.4pt);
\draw[black!40!green,line width=2pt] (5.9, 0.382) -- (0.0, 0.382);
\draw[black!40!green,line width=2pt] (0.0, 0.382) -- (-1.236, -0.236);
\draw[black!40!green,line width=2pt] (-2.82, -0.764) -- (-2.292, -0.764);
\draw[black!40!green,line width=2pt] (-2.292, -0.764) -- (-2.646, -0.941);
\draw[black!40!green,line width=2pt] (-2.646, -0.941) -- (6, -0.941) node[left,xshift=2pt]{$8$};
\fill[black!40!green] (5.9, 0.382) ellipse (3pt and 1.4pt);
\draw[black!40!green,line width=2pt] (5.9, 0.382) -- (0.0, 0.382);
\draw[black!40!green,line width=2pt] (0.0, 0.382) -- (-0.573, -0.191);
\draw[black!40!green,line width=2pt] (-0.573, -0.191) -- (6, -0.191) node[right,xshift=-2pt]{$6$};
\fill[black!40!green] (5.9, 0.382) ellipse (3pt and 1.4pt);
\draw[black!40!green,line width=2pt] (5.9, 0.382) -- (-0.236, 0.382);
\draw[black!40!green,line width=2pt] (-0.236, 0.618) -- (-0.927, 1.309);
\draw[black!40!green,line width=2pt] (-0.927, 1.309) -- (6, 1.309) node[right,xshift=-2pt]{$6$};
\end{tikzpicture}
\hspace{2mm}
\begin{tikzpicture}[xscale=1.6,yscale=.7,rotate=90]
\clip (-3,-1) rectangle (6.6,2);
\draw[dashed,gray] (0.0, 0.5) -- (-1.0, 1.0);
\draw[dashed,gray] (0.0, 0.5) -- (-1.0, 0.0);
\draw[dashed,gray] (-1.5, -0.5) -- (-1.0, 0.0);
\draw[dashed,gray] (-1.5, 1.5) -- (-1.0, 1.0);
\draw[dashed,gray] (-1.5, 1.5) -- (-4.0, 2.0);
\draw[dashed,gray] (-1.5, -0.5) -- (-4.0, -1.0);
\draw[dashed,gray] (-6.0, -1.5) -- (-4.0, -1.0);
\draw[dashed,gray] (-6.0, 2.5) -- (-4.0, 2.0);
\draw[gray] (0.0, 0.5) node[fill,cross,inner sep=2pt,rotate=90.0]{} -- (0.0, 2.5);
\draw[gray] (0.0, 0.5) node[fill,cross,inner sep=2pt,rotate=-90.0]{} -- (0.0, -1.5);
\draw[gray] (-1.5, 1.5) node[fill,cross,inner sep=2pt,rotate=-18.435]{} -- (6.0, -1.0);
\draw[gray] (-1.5, -0.5) node[fill,cross,inner sep=2pt,rotate=18.435]{} -- (6.0, 2.0);
\draw[gray] (-1.5, 1.5) node[fill,cross,inner sep=2pt,rotate=161.565]{} -- (-4.5, 2.5);
\draw[gray] (-1.5, -0.5) node[fill,cross,inner sep=2pt,rotate=-161.565]{} -- (-4.5, -1.5);
\draw[gray] (-6.0, 2.5) node[fill,cross,inner sep=2pt,rotate=-9.462]{} -- (6.0, 0.5);
\draw[gray] (-6.0, -1.5) node[fill,cross,inner sep=2pt,rotate=9.462]{} -- (6.0, 0.5);
\draw[gray] (-3.0, 2.0) -- (6.0, 2.0);
\draw[gray] (0.0, 0.0) -- (6.0, 0.0);
\draw[gray] (-3.0, -1.0) -- (6.0, -1.0);
\draw[gray] (0.0, 1.0) -- (6.0, 1.0);
\draw[gray] (1.5, 0.5) -- (6.0, 0.5);
\draw[gray] (3.0, 0.0) -- (6.0, -0.5);
\draw[gray] (3.0, 1.0) -- (6.0, 1.5);
\draw[gray] (-0.0, -0.0) -- (4.5, -1.5);
\draw[gray] (-0.0, -0.0) -- (6.0, -0.0);
\draw[gray] (-3.0, 2.0) -- (-3.0, 2.5);
\draw[gray] (-0.0, -0.0) -- (6.0, 1.0);
\draw[gray] (-0.0, 1.0) -- (4.5, 2.5);
\draw[gray] (-0.0, 1.0) -- (6.0, 0.0);
\draw[gray] (-0.0, 1.0) -- (6.0, 1.0);
\draw[gray] (-3.0, -1.0) -- (-3.0, -1.5);
\fill[black!40!green] (5.8, 0.382) ellipse (3pt and 1.4pt);
\draw[black!40!green,line width=2pt] (5.8, 0.382) -- (1.146, 0.382);
\draw[black!40!green,line width=2pt] (1.146, 0.382) -- (1.146, 0.618);
\draw[black!40!green,line width=2pt] (1.146, 0.618) -- (6, 0.618) node[left,xshift=2pt]{$9$};
\fill[black!40!green] (5.8, 0.382) ellipse (3pt and 1.4pt);
\draw[black!40!green,line width=2pt] (5.8, 0.382) -- (0.0, 0.382);
\draw[black!40!green,line width=2pt] (0.0, 0.382) -- (-1.146, -0.382);
\draw[black!40!green,line width=2pt] (-1.146, -0.382) -- (6, -0.382) node[left,xshift=2pt]{$9$};
\fill[black!40!green] (5.8, 0.382) ellipse (3pt and 1.4pt);
\draw[black!40!green,line width=2pt] (5.8, 0.382) -- (0.0, 0.382);
\draw[black!40!green,line width=2pt] (0.0, 0.382) -- (-0.708, 0.146);
\draw[black!40!green,line width=2pt] (-0.708, 0.854) -- (-1.146, 1.146);
\draw[black!40!green,line width=2pt] (-3.27, 1.854) -- (-1.854, 1.618);
\draw[black!40!green,line width=2pt] (-1.854, 1.618) -- (6, 1.618) node[left,xshift=2pt]{$9$};
\end{tikzpicture}
\hspace{2mm}
\begin{tikzpicture}[xscale=1.6,yscale=.7,rotate=90]
\clip (-3,-1) rectangle (6.6,2);
\draw[dashed,gray] (0.0, 0.5) -- (-1.0, 1.0);
\draw[dashed,gray] (0.0, 0.5) -- (-1.0, 0.0);
\draw[dashed,gray] (-1.5, -0.5) -- (-1.0, 0.0);
\draw[dashed,gray] (-1.5, 1.5) -- (-1.0, 1.0);
\draw[dashed,gray] (-1.5, 1.5) -- (-4.0, 2.0);
\draw[dashed,gray] (-1.5, -0.5) -- (-4.0, -1.0);
\draw[dashed,gray] (-6.0, -1.5) -- (-4.0, -1.0);
\draw[dashed,gray] (-6.0, 2.5) -- (-4.0, 2.0);
\draw[gray] (0.0, 0.5) node[fill,cross,inner sep=2pt,rotate=90.0]{} -- (0.0, 2.5);
\draw[gray] (0.0, 0.5) node[fill,cross,inner sep=2pt,rotate=-90.0]{} -- (0.0, -1.5);
\draw[gray] (-1.5, 1.5) node[fill,cross,inner sep=2pt,rotate=-18.435]{} -- (6.0, -1.0);
\draw[gray] (-1.5, -0.5) node[fill,cross,inner sep=2pt,rotate=18.435]{} -- (6.0, 2.0);
\draw[gray] (-1.5, 1.5) node[fill,cross,inner sep=2pt,rotate=161.565]{} -- (-4.5, 2.5);
\draw[gray] (-1.5, -0.5) node[fill,cross,inner sep=2pt,rotate=-161.565]{} -- (-4.5, -1.5);
\draw[gray] (-6.0, 2.5) node[fill,cross,inner sep=2pt,rotate=-9.462]{} -- (6.0, 0.5);
\draw[gray] (-6.0, -1.5) node[fill,cross,inner sep=2pt,rotate=9.462]{} -- (6.0, 0.5);
\draw[gray] (-3.0, 2.0) -- (6.0, 2.0);
\draw[gray] (0.0, 0.0) -- (6.0, 0.0);
\draw[gray] (-3.0, -1.0) -- (6.0, -1.0);
\draw[gray] (0.0, 1.0) -- (6.0, 1.0);
\draw[gray] (1.5, 0.5) -- (6.0, 0.5);
\draw[gray] (3.0, 0.0) -- (6.0, -0.5);
\draw[gray] (3.0, 1.0) -- (6.0, 1.5);
\draw[gray] (-0.0, -0.0) -- (4.5, -1.5);
\draw[gray] (-0.0, -0.0) -- (6.0, -0.0);
\draw[gray] (-3.0, 2.0) -- (-3.0, 2.5);
\draw[gray] (-0.0, -0.0) -- (6.0, 1.0);
\draw[gray] (-0.0, 1.0) -- (4.5, 2.5);
\draw[gray] (-0.0, 1.0) -- (6.0, 0.0);
\draw[gray] (-0.0, 1.0) -- (6.0, 1.0);
\draw[gray] (-3.0, -1.0) -- (-3.0, -1.5);
\fill[black!40!green] (5.8, 0.382) ellipse (3pt and 1.4pt);
\draw[black!40!green,line width=2pt] (5.8, 0.382) -- (1.146, 0.382);
\draw[black!40!green,line width=2pt] (1.146, 0.382) -- (0.792, 0.736);
\draw[black!40!green,line width=2pt] (0.792, 0.736) -- (6, 0.736) node[left,xshift=2pt]{$8$};
\fill[black!40!green] (5.8, 0.382) ellipse (3pt and 1.4pt);
\draw[black!40!green,line width=2pt] (5.8, 0.382) -- (0.0, 0.382);
\draw[black!40!green,line width=2pt] (0.0, 0.382) -- (-2.0, -0.618);
\draw[black!40!green,line width=2pt] (-2.82, -0.764) -- (6, -0.764) node[left,xshift=2pt]{$6$};
\end{tikzpicture}
\hspace{2mm}
\begin{tikzpicture}[xscale=1.6,yscale=.7,rotate=90]
\clip (-3,-1) rectangle (6.6,2);
\draw[dashed,gray] (0.0, 0.5) -- (-1.0, 1.0);
\draw[dashed,gray] (0.0, 0.5) -- (-1.0, 0.0);
\draw[dashed,gray] (-1.5, -0.5) -- (-1.0, 0.0);
\draw[dashed,gray] (-1.5, 1.5) -- (-1.0, 1.0);
\draw[dashed,gray] (-1.5, 1.5) -- (-4.0, 2.0);
\draw[dashed,gray] (-1.5, -0.5) -- (-4.0, -1.0);
\draw[dashed,gray] (-6.0, -1.5) -- (-4.0, -1.0);
\draw[dashed,gray] (-6.0, 2.5) -- (-4.0, 2.0);
\draw[gray] (0.0, 0.5) node[fill,cross,inner sep=2pt,rotate=90.0]{} -- (0.0, 2.5);
\draw[gray] (0.0, 0.5) node[fill,cross,inner sep=2pt,rotate=-90.0]{} -- (0.0, -1.5);
\draw[gray] (-1.5, 1.5) node[fill,cross,inner sep=2pt,rotate=-18.435]{} -- (6.0, -1.0);
\draw[gray] (-1.5, -0.5) node[fill,cross,inner sep=2pt,rotate=18.435]{} -- (6.0, 2.0);
\draw[gray] (-1.5, 1.5) node[fill,cross,inner sep=2pt,rotate=161.565]{} -- (-4.5, 2.5);
\draw[gray] (-1.5, -0.5) node[fill,cross,inner sep=2pt,rotate=-161.565]{} -- (-4.5, -1.5);
\draw[gray] (-6.0, 2.5) node[fill,cross,inner sep=2pt,rotate=-9.462]{} -- (6.0, 0.5);
\draw[gray] (-6.0, -1.5) node[fill,cross,inner sep=2pt,rotate=9.462]{} -- (6.0, 0.5);
\draw[gray] (-3.0, 2.0) -- (6.0, 2.0);
\draw[gray] (0.0, 0.0) -- (6.0, 0.0);
\draw[gray] (-3.0, -1.0) -- (6.0, -1.0);
\draw[gray] (0.0, 1.0) -- (6.0, 1.0);
\draw[gray] (1.5, 0.5) -- (6.0, 0.5);
\draw[gray] (3.0, 0.0) -- (6.0, -0.5);
\draw[gray] (3.0, 1.0) -- (6.0, 1.5);
\draw[gray] (-0.0, -0.0) -- (4.5, -1.5);
\draw[gray] (-0.0, -0.0) -- (6.0, -0.0);
\draw[gray] (-3.0, 2.0) -- (-3.0, 2.5);
\draw[gray] (-0.0, -0.0) -- (6.0, 1.0);
\draw[gray] (-0.0, 1.0) -- (4.5, 2.5);
\draw[gray] (-0.0, 1.0) -- (6.0, 0.0);
\draw[gray] (-0.0, 1.0) -- (6.0, 1.0);
\draw[gray] (-3.0, -1.0) -- (-3.0, -1.5);
\fill[black!40!green] (5.8, 0.382) ellipse (3pt and 1.4pt);
\draw[black!40!green,line width=2pt] (5.8, 0.382) -- (1.146, 0.382);
\draw[black!40!green,line width=2pt] (1.146, 0.382) -- (-0.27, 1.09);
\draw[black!40!green,line width=2pt] (-0.27, 1.09) -- (6, 1.09) node[left,xshift=2pt]{$5$};
\end{tikzpicture}
\caption{Broken lines used to compute $R_{p,q}(\mathbb{P}^2,2L)$ for $(p,q)=(1,5),(2,4),(3,3),(4,2),(5,1)$. The numbers above the infinite segments are the genus $0$ coefficients of the broken lines.}
\label{fig:sage2}
\end{figure}
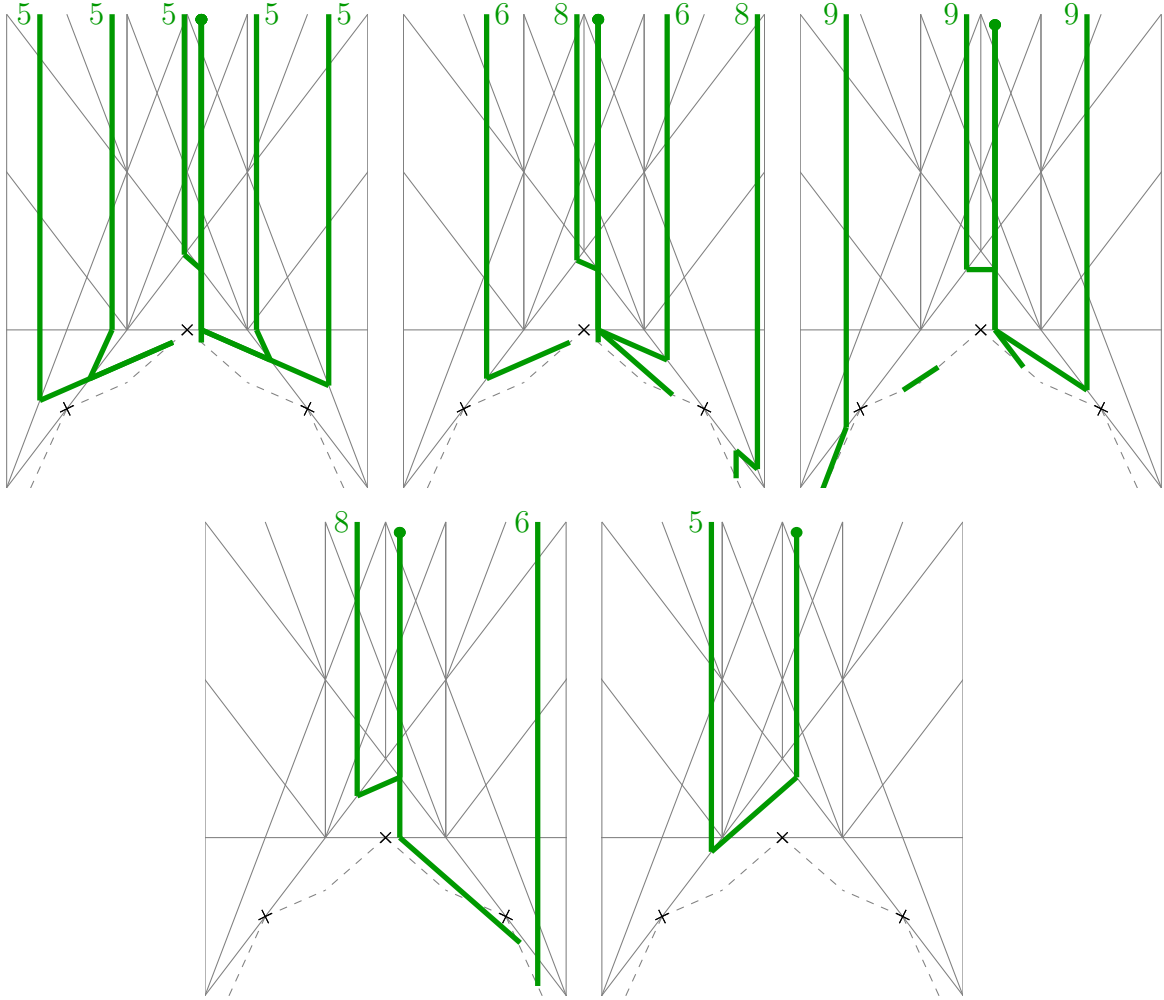
For $(p,q)=(1,5)$ we have $5$ broken lines, giving $5$ tropical curves. Some of these tropical curves are in fact two different tropical curves by the splitting $1+1$ of the weight $2$ bounded leg. However, it turns out that they all have the same multiplicity. The easiest tropical curve has one vertex of multiplicity $5$ and one bounded leg of weight $1$, giving
\begin{equation}
\label{eq:1,5}
m_h(\boldsymbol{q}) = \frac{\boldsymbol{q}^{5/2}-\boldsymbol{q}^{-5/2}}{\boldsymbol{q}^{1/2}-\boldsymbol{q}^{1/2}} = \boldsymbol{q}^2+\boldsymbol{q}+1+\boldsymbol{q}^{-1}+\boldsymbol{q}^{-2}
\end{equation}
Some more involved tropical curve has two vertices of multiplicity $5,4$ and two bounded legs of weights $2,1$. This gives
\begin{eqnarray*}
m_h(\boldsymbol{q}) &=& -\frac{1}{2}\frac{(\boldsymbol{q}^{5/2}-\boldsymbol{q}^{-5/2})(\boldsymbol{q}^2-\boldsymbol{q}^{-2})}{(\boldsymbol{q}-\boldsymbol{q}^{-1})(\boldsymbol{q}^{1/2}-\boldsymbol{q}^{1/2})} \\
&=& -\frac{1}{2}\frac{\boldsymbol{q}^{5/2}-\boldsymbol{q}^{-5/2}}{\boldsymbol{q}^{1/2}-\boldsymbol{q}^{1/2}} \cdot (\boldsymbol{q}+\boldsymbol{q}^{-1})
\end{eqnarray*}
The factor $\frac{1}{2}$ comes from $w_L=2$. The corresponding tropical curve with leg splitting $1+1$ has three vertices of multiplicity $5,2,2$ and three bounded legs of weights $1,1,1$. This gives
\begin{eqnarray*}
m_h(\boldsymbol{q}) &=& \frac{1}{2}\frac{(\boldsymbol{q}^{5/2}-\boldsymbol{q}^{-5/2})(\boldsymbol{q}-\boldsymbol{q}^{-1})^2}{(\boldsymbol{q}^{1/2}-\boldsymbol{q}^{1/2})^3} \\
&=& \frac{1}{2}\frac{\boldsymbol{q}^{5/2}-\boldsymbol{q}^{-5/2}}{\boldsymbol{q}^{1/2}-\boldsymbol{q}^{1/2}} \cdot (\boldsymbol{q}^{1/2}+\boldsymbol{q}^{-1/2})^2
\end{eqnarray*}
The factor $\frac{1}{2}$ comes from $|\text{Aut}(h)|=2$. Using $(\boldsymbol{q}^{1/2}+\boldsymbol{q}^{-1/2})^2 = \boldsymbol{q}+\boldsymbol{q}^{-1}+2$ the two tropical curves together give the same contribution \eqref{eq:1,5} as the first tropical curve. Multiplying by five, the number of tropical curves, we get
\[ \sum_{g\geq 0}R_{1,5}^g\hbar^{2g} = 5\boldsymbol{q}^2+5\boldsymbol{q}+5+5\boldsymbol{q}^{-1}+5\boldsymbol{q}^{-2} \]
This gives
\[ R_{1,5}^g = 5\delta_{g,0} + \frac{(-1)^g}{(2g)!}(10+10\cdot 2^{2g}), \]
where $\delta_{g,0}$ is the Kronecker delta. The first numbers are:

\vspace{3mm}
\begin{center}
\renewcommand{\arraystretch}{1.2}
\begin{tabular}{|c|c|c|c|c|} 								\hline
$R_{1,5}^0=25$	& $R_{1,5}^1=-25$		& $R_{1,5}^2=\frac{85}{12}$		& $R_{1,5}^3=-\frac{65}{72}$		& $R_{1,5}^4=\frac{257}{4032}$	\\ \hline
\end{tabular}
\end{center}
\vspace{3mm}

For $(p,q)=(2,4)$ there four tropical curves. Two of them have multiplicity
\begin{eqnarray*}
m_h(\boldsymbol{q}) &=& \frac{(\boldsymbol{q}^2-\boldsymbol{q}^{-2})(\boldsymbol{q}-\boldsymbol{q}^{-1})}{(\boldsymbol{q}^{1/2}-\boldsymbol{q}^{-1/2})^2} \\
&=& (\boldsymbol{q}+\boldsymbol{q}^{-1})(\boldsymbol{q}^{1/2}+\boldsymbol{q}^{-1/2})^2 \\
&=& \boldsymbol{q}^2+2\boldsymbol{q}+2+2\boldsymbol{q}^{-1}+\boldsymbol{q}^{-2}
\end{eqnarray*}
The other two tropical curves again have different bounded leg splittings. A careful computation gives
\[ m_h(\boldsymbol{q}) = \boldsymbol{q}^2+\boldsymbol{q}+2+\boldsymbol{q}^{-1}+\boldsymbol{q}^{-2} \]
So in total we have (note that we have to divide by $p=2$)
\[ \sum_{g\geq 0}R_{2,4}^g\hbar^{2g} = 2\boldsymbol{q}^2 + 3\boldsymbol{q} + 4 + 3\boldsymbol{q}^{-1} + 2\boldsymbol{q}^{-2} \]
This gives 
\[ R_{2,4}^g = 4\delta_{g,0} + \frac{(-1)^g}{(2g)!}(6+4 \cdot 2^{2g}). \]
The first numbers are

\vspace{3mm}
\begin{center}
\renewcommand{\arraystretch}{1.2}
\begin{tabular}{|c|c|c|c|c|} 								\hline
$R_{2,4}^0=14$	& $R_{2,4}^1=-11$		& $R_{2,4}^2=\frac{35}{12}$		& $R_{2,4}^3=-\frac{131}{360}$		& $R_{2,4}^4=\frac{103}{4032}$	\\ \hline
\end{tabular}
\end{center}
\vspace{3mm}

For $(p,q)=(3,3)$ there are three tropical curves, each of multiplicity
\[ m_h(q) = \frac{(\boldsymbol{q}^{3/2}-\boldsymbol{q}^{-3/2})^2}{(\boldsymbol{q}^{1/2}-\boldsymbol{q}^{-1/2})^2} = \boldsymbol{q}^2+2\boldsymbol{q}+3+2\boldsymbol{q}^{-1}+\boldsymbol{q}^{-2} \]
So in total we have (we have $3$ tropical curves and have to divide by $p=3$)
\[ \sum_{g\geq 0}R_{3,3}^g\hbar^{2g} = \boldsymbol{q}^2+2\boldsymbol{q}+3+2\boldsymbol{q}^{-1}+\boldsymbol{q}^{-2} \]
This gives
\[ R_{3,3}^g = 3\delta_{g,0} + \frac{(-1)^g}{(2g)!}(4+2\cdot 2^{2g}). \]
The first numbers are

\vspace{3mm}
\begin{center}
\renewcommand{\arraystretch}{1.2}
\begin{tabular}{|c|c|c|c|c|} 								\hline
$R_{3,3}^0=9$	& $R_{3,3}^1=-6$		& $R_{3,3}^2=\frac{3}{2}$		& $R_{3,3}^3=-\frac{11}{60}$		& $R_{3,3}^4=\frac{43}{3360}$	\\ \hline
\end{tabular}
\end{center}
\vspace{3mm}

For $(p,q)=(4,2)$ we only have two of the tropical curves as for $(p,q)=(4,2)$, one of each type, and we have to divide by $p=4$ instead of $p=2$, so $R_{4,2}^g=\frac{1}{4}R_{2,4}^g$, giving
\[ R_{4,2}^g = 1\delta_{g,0} + \frac{(-1)^g}{(2g)!}\left(\frac{3}{2}+1 \cdot 2^{2g}\right). \]
The first numbers are

\vspace{3mm}
\begin{center}
\renewcommand{\arraystretch}{1.2}
\begin{tabular}{|c|c|c|c|c|} 								\hline
$R_{4,2}^0=\frac{7}{2}$	& $R_{4,2}^1=-\frac{11}{4}$		& $R_{4,2}^2=\frac{35}{48}$		& $R_{4,2}^3=-\frac{131}{1440}$		& $R_{4,2}^4=\frac{103}{16128}$	\\ \hline
\end{tabular}
\end{center}
\vspace{3mm}

For $(p,q)=(5,1)$ we only have one of the tropical curve as for $(p,q)=(1,5)$, and we have to divide by $p=5$ to get $R_{p,q}^g$, so
\[ R_{5,1}^g = 1\delta_{g,0} + \frac{(-1)^g}{(2g)!}\left(\frac{2}{5}+\frac{2}{5}\cdot 2^{2g}\right). \]
The first numbers are

\vspace{3mm}
\begin{center}
\renewcommand{\arraystretch}{1.2}
\begin{tabular}{|c|c|c|c|c|} 								\hline
$R_{5,1}^0=1$	& $R_{5,1}^1=-1$		& $R_{5,1}^2=\frac{17}{60}$		& $R_{5,1}^3=-\frac{13}{360}$		& $R_{5,1}^4=\frac{257}{100800}$	\\ \hline
\end{tabular}
\end{center}
\vspace{3mm}

\subsection{Degree $3$}

For $(p,q)=(8,1)$ we have three different tropical curves. One has classical vertex multiplicities $m_V=8,1$ and bounded leg weights $w_L=1,1$, giving
\[ m_h(\textbf{q}) = \textbf{q}^{7/2}+\textbf{q}^{5/2}+\ldots+\textbf{q}^{-7/2}. \]
The second has $m_V=8,6,1$ and $w_L=2,1,1$ giving
\[ m_h(\textbf{q}) = -\frac{1}{2}(\textbf{q}^{11/2}+\textbf{q}^{9/2}+2\textbf{q}^{7/2}+2\textbf{q}^{5/2}+3\textbf{q}^{3/2}+3\textbf{q}^{1/2}+ (-\text{exp})). \]
The third is obtained from the second one via leg splitting and has $m_V=8,3,3,1$ and $w_L=1,1,1,1$, giving
\[ m_h(\textbf{q}) = \frac{1}{2}(\textbf{q}^{11/2}+3\textbf{q}^{9/2}+6\textbf{q}^{7/2}+8\textbf{q}^{5/2}+9\textbf{q}^{3/2}+9\textbf{q}^{1/2}+ (-\text{exp})). \]
Putting everything together we obtain the tropical counts
\[ \mathcal{R}_{8,1}(\textbf{q}) = \textbf{q}^{9/2}+3\textbf{q}^{7/2}+4\textbf{q}^{5/2}+4\textbf{q}^{3/2}+4\textbf{q}^{1/2}+ (-\text{exp}). \]
Dividing by $8$ we get the $2$-marked log invariants
\[ R_{8,1}^g = \frac{(-1)^g}{8(2g)!}\left(8\left(\frac{1}{2}\right)^{2g}+8\left(\frac{3}{2}\right)^{2g}+8\left(\frac{5}{2}\right)^{2g}+6\left(\frac{7}{2}\right)^{2g}+2\left(\frac{9}{2}\right)^{2g}\right). \]
The first numbers are

\vspace{3mm}
\begin{center}
\renewcommand{\arraystretch}{1.2}
\begin{tabular}{|c|c|c|c|c|} 								\hline
$R_{8,1}^0=4$	& $R_{8,1}^1=-\frac{23}{2}$		& $R_{8,1}^2=\frac{1037}{96}$		& $R_{8,1}^3=-\frac{59363}{11520}$		& $R_{8,1}^4=\frac{3870617}{2580480}$	\\ \hline
\end{tabular}
\end{center}
\vspace{3mm}

For $(p,q)=(7,2)$ we have six tropical curve. The first has $m_V=7,2$ and $w_L=1,1$ giving
\[ m_h(\textbf{q}) = \textbf{q}^{7/2}+2\textbf{q}^{5/2}+2\textbf{q}^{3/2}+2\textbf{q}^{1/2}+ (-\text{exp}). \]
The second has $m_V=7,6,2$ and $w_L=2,1,1$ giving
\[ m_h(\textbf{q}) = -\frac{1}{2}(\textbf{q}^{11/2}+2\textbf{q}^{9/2}+3\textbf{q}^{7/2}+4\textbf{q}^{5/2}+5\textbf{q}^{3/2}+6\textbf{q}^{1/2} + (-\text{exp})). \]
Splitting the leg we obtain $m_V=7,3,3,2$ and $w_L=1,1,1,1$ giving
\[ m_h(\textbf{q}) = \frac{1}{2}(\textbf{q}^{11/2}+4\textbf{q}^{9/2}+9\textbf{q}^{7/2}+14\textbf{q}^{5/2}+17\textbf{q}^{3/2}+18\textbf{q}^{1/2} + (-\text{exp})). \]
There is another curve with $m_V=7,12$ and $w_L=3,1$ giving
\[ m_h(\textbf{q}) = \frac{1}{3}(\textbf{q}^{15/2}+\textbf{q}^{13/2}+\textbf{q}^{11/2}+2\textbf{q}^{9/2}+2\textbf{q}^{7/2}+2\textbf{q}^{5/2}+3\textbf{q}^{3/2}+2\textbf{q}^{1/2} + (-\text{exp})). \]
Splitting the leg once we get $m_V=7,8,4$ and $w_L=2,1,1$ giving
\[ m_h(\textbf{q}) = -\frac{1}{2}(\textbf{q}^{15/2}+2\textbf{q}^{13/2}+4\textbf{q}^{11/2}+6\textbf{q}^{9/2}+8\textbf{q}^{7/2}+10\textbf{q}^{5/2}+12\textbf{q}^{3/2}+13\textbf{q}^{1/2} + (-\text{exp})). \]
Splitting the leg again weg get $m_V=7,4,4,4$ and $w_L=1,1,1,1$ giving
\[ m_h(\textbf{q}) = \frac{1}{6}(\textbf{q}^{15/2}+4\textbf{q}^{13/2}+10\textbf{q}^{11/2}+20\textbf{q}^{9/2}+32\textbf{q}^{7/2}+44\textbf{q}^{5/2}+54\textbf{q}^{3/2}+59\textbf{q}^{1/2} + (-\text{exp})). \]
Putting everything together we obtain the tropical counts
\[ \mathcal{R}_{7,2}(\textbf{q}) = 2\textbf{q}^{9/2}+6\textbf{q}^{7/2}+10\textbf{q}^{5/2}+12\textbf{q}^{3/2}+12\textbf{q}^{1/2} + (-\text{exp}). \]
Dividing by $7$ we get the $2$-marked log invariants
\[ R_{7,2}^g = \frac{(-1)^g}{7(2g)!}\left(24\left(\frac{1}{2}\right)^{2g}+24\left(\frac{3}{2}\right)^{2g}+20\left(\frac{5}{2}\right)^{2g}+12\left(\frac{7}{2}\right)^{2g}+4\left(\frac{9}{2}\right)^{2g}\right). \]
The first numbers are

\vspace{3mm}
\begin{center}
\renewcommand{\arraystretch}{1.2}
\begin{tabular}{|c|c|c|c|c|} 								\hline
$R_{7,2}^0=12$	& $R_{7,2}^1=-\frac{59}{2}$		& $R_{7,2}^2=\frac{2483}{96}$		& $R_{7,2}^3=-\frac{966893}{80640}$		& $R_{7,2}^4=\frac{8904803}{2580480}$	\\ \hline
\end{tabular}
\end{center}
\vspace{3mm}

\subsection{Central chamber}

Let us compute $\vartheta_2$ in the central chamber via tropical disks.

\begin{figure}[h!]
\centering
\begin{tikzpicture}[xscale=2,yscale=1,rotate=90]
\clip (-3,-1) rectangle (6.6,2);
\draw[dashed,gray] (0.0, 0.5) -- (-1.0, 1.0);
\draw[dashed,gray] (0.0, 0.5) -- (-1.0, 0.0);
\draw[dashed,gray] (-1.5, -0.5) -- (-1.0, 0.0);
\draw[dashed,gray] (-1.5, 1.5) -- (-1.0, 1.0);
\draw[dashed,gray] (-1.5, 1.5) -- (-4.0, 2.0);
\draw[dashed,gray] (-1.5, -0.5) -- (-4.0, -1.0);
\draw[dashed,gray] (-6.0, -1.5) -- (-4.0, -1.0);
\draw[dashed,gray] (-6.0, 2.5) -- (-4.0, 2.0);
\draw[gray] (0.0, 0.5) node[fill,cross,inner sep=2pt,rotate=90.0]{} -- (0.0, 2.5);
\draw[gray] (0.0, 0.5) node[fill,cross,inner sep=2pt,rotate=-90.0]{} -- (0.0, -1.5);
\draw[gray] (-1.5, 1.5) node[fill,cross,inner sep=2pt,rotate=-18.435]{} -- (6.0, -1.0);
\draw[gray] (-1.5, -0.5) node[fill,cross,inner sep=2pt,rotate=18.435]{} -- (6.0, 2.0);
\draw[gray] (-1.5, 1.5) node[fill,cross,inner sep=2pt,rotate=161.565]{} -- (-4.5, 2.5);
\draw[gray] (-1.5, -0.5) node[fill,cross,inner sep=2pt,rotate=-161.565]{} -- (-4.5, -1.5);
\draw[gray] (-6.0, 2.5) node[fill,cross,inner sep=2pt,rotate=-9.462]{} -- (6.0, 0.5);
\draw[gray] (-6.0, -1.5) node[fill,cross,inner sep=2pt,rotate=9.462]{} -- (6.0, 0.5);
\draw[gray] (-3.0, 2.0) -- (6.0, 2.0);
\draw[gray] (0.0, 0.0) -- (6.0, 0.0);
\draw[gray] (-3.0, -1.0) -- (6.0, -1.0);
\draw[gray] (0.0, 1.0) -- (6.0, 1.0);
\draw[gray] (1.5, 0.5) -- (6.0, 0.5);
\draw[gray] (3.0, 0.0) -- (6.0, -0.5);
\draw[gray] (3.0, 1.0) -- (6.0, 1.5);
\draw[gray] (-0.0, -0.0) -- (4.5, -1.5);
\draw[gray] (-0.0, -0.0) -- (6.0, -0.0);
\draw[gray] (-3.0, 2.0) -- (-3.0, 2.5);
\draw[gray] (-0.0, -0.0) -- (6.0, 1.0);
\draw[gray] (-0.0, 1.0) -- (4.5, 2.5);
\draw[gray] (-0.0, 1.0) -- (6.0, 0.0);
\draw[gray] (-0.0, 1.0) -- (6.0, 1.0);
\draw[gray] (-3.0, -1.0) -- (-3.0, -1.5);
\fill[black!40!green] (-0.3, 0.2) ellipse (3pt and 1.4pt);
\draw[black!40!green,line width=2pt] (-0.3, 0.2) -- (-1.3, -0.3);
\draw[black!40!green,line width=2pt] (-2.5, -0.7) -- (6, -0.7) node[above]{$6$};
\fill[black!40!green] (-0.3, 0.2) ellipse (3pt and 1.4pt);
\draw[black!40!green,line width=2pt] (-0.3, 0.2) -- (-0.75, -0.25);
\draw[black!40!green,line width=2pt] (-0.75, -0.25) -- (6, -0.25) node[above]{$5$};
\fill[black!40!green] (-0.3, 0.2) ellipse (3pt and 1.4pt);
\draw[black!40!green,line width=2pt] (-0.3, 0.2) -- (-0.6, 0.2);
\draw[black!40!green,line width=2pt] (-0.6, 0.8) -- (-1.2, 1.4);
\draw[black!40!green,line width=2pt] (-1.2, 1.4) -- (6, 1.4) node[above]{$1$};
\fill[black!40!green] (-0.3, 0.2) ellipse (3pt and 1.4pt);
\draw[black!40!green,line width=2pt] (-0.3, 0.2) -- (6, 0.2) node[above]{$4$};
\fill[black!40!green] (-0.3, 0.2) ellipse (3pt and 1.4pt);
\draw[black!40!green,line width=2pt] (-0.3, 0.2) -- (0.0, 0.35);
\draw[black!40!green,line width=2pt] (0.0, 0.35) -- (6, 0.35) node[above]{$3$};
\fill[black!40!green] (-0.3, 0.2) ellipse (3pt and 1.4pt);
\draw[black!40!green,line width=2pt] (-0.3, 0.2) -- (0.0, 0.5);
\draw[black!40!green,line width=2pt] (0.0, 0.5) -- (6, 0.5) node[above]{$2$};
\draw (-0.8,0.45) -- (0.2,0.45) -- (0.2,-0.2) -- (-0.8,-0.2) -- (-0.8,0.45);
\end{tikzpicture}
\hspace{1cm}
\begin{tikzpicture}[xscale=4,yscale=2,rotate=90]
\draw (-2,0); 
\draw (-1.45,-1.1) rectangle (0.7,0.75);
\clip (-1.45,-1.1) rectangle (0.7,0.75);
\draw[dashed,gray] (0.0, 0.5) -- (-1.0, 1.0);
\draw[dashed,gray] (0.0, 0.5) -- (-1.0, 0.0);
\draw[dashed,gray] (-1.5, -0.5) -- (-1.0, 0.0);
\draw[dashed,gray] (-1.5, 1.5) -- (-1.0, 1.0);
\draw[dashed,gray] (-1.5, 1.5) -- (-4.0, 2.0);
\draw[dashed,gray] (-1.5, -0.5) -- (-4.0, -1.0);
\draw[dashed,gray] (-6.0, -1.5) -- (-4.0, -1.0);
\draw[dashed,gray] (-6.0, 2.5) -- (-4.0, 2.0);
\draw[gray] (0.0, 0.5) node[fill,cross,inner sep=2pt,rotate=90.0]{} -- (0.0, 2.5);
\draw[gray] (0.0, 0.5) node[fill,cross,inner sep=2pt,rotate=-90.0]{} -- (0.0, -1.1);
\draw[gray] (-1.5, 1.5) node[fill,cross,inner sep=2pt,rotate=-18.435]{} -- (6.0, -1.0);
\draw[gray] (-1.5, -0.5) node[fill,cross,inner sep=2pt,rotate=18.435]{} -- (6.0, 2.0);
\draw[gray] (-1.5, 1.5) node[fill,cross,inner sep=2pt,rotate=161.565]{} -- (-4.5, 2.5);
\draw[gray] (-1.5, -0.5) node[fill,cross,inner sep=2pt,rotate=-161.565]{} -- (-4.5, -1.5);
\draw[gray] (-6.0, 2.5) node[fill,cross,inner sep=2pt,rotate=-9.462]{} -- (6.0, 0.5);
\draw[gray] (-3.0, 2.0) -- (6.0, 2.0);
\draw[gray] (0.0, 0.0) -- (6.0, 0.0);
\draw[gray] (0.0, 1.0) -- (6.0, 1.0);
\draw[gray] (1.5, 0.5) -- (6.0, 0.5);
\draw[gray] (3.0, 0.0) -- (6.0, -0.5);
\draw[gray] (3.0, 1.0) -- (6.0, 1.5);
\draw[gray] (-0.0, -0.0) -- (4.5, -1.5);
\draw[gray] (-0.0, -0.0) -- (6.0, -0.0);
\draw[gray] (-3.0, 2.0) -- (-3.0, 2.5);
\draw[gray] (-0.0, -0.0) -- (6.0, 1.0);
\draw[gray] (-0.0, 1.0) -- (4.5, 2.5);
\draw[gray] (-0.0, 1.0) -- (6.0, 0.0);
\draw[gray] (-0.0, 1.0) -- (6.0, 1.0);
\draw[gray] (-3.0, -1.0) -- (-3.0, -1.5);
\fill[black!40!green] (-0.3, 0.2) ellipse (1.7pt and 0.7pt);
\draw[black!40!green,line width=2pt] (-0.3, 0.2) -- (-1.3, -0.3) node[right]{\scriptsize$t^2x^2y^{-4}$};
\fill[black!40!green] (-0.3, 0.2) ellipse (1.7pt and 0.7pt);
\draw[black!40!green,line width=2pt] (-0.3, 0.2) -- (-0.75, -0.25) node[right]{\scriptsize$t^2(\boldsymbol{q}^{1/2}+\boldsymbol{q}^{-1/2})xy^{-1}$};
\fill[black!40!green] (-0.3, 0.2) ellipse (1.7pt and 0.7pt);
\draw[black!40!green,line width=2pt] (-0.3, 0.2) -- (-0.6, 0.2) node[below,xshift=-7mm]{\scriptsize$t^2(\boldsymbol{q}^{1/2}+\boldsymbol{q}^{-1/2})y^{-1}$};
\fill[black!40!green] (-0.3, 0.2) ellipse (1.7pt and 0.7pt);
\draw[black!40!green,line width=2pt] (-0.3, 0.2) -- (0, 0.2) node[above]{\scriptsize$t^2y^2$};
\fill[black!40!green] (-0.3, 0.2) ellipse (1.7pt and 0.7pt);
\draw[black!40!green,line width=2pt] (-0.3, 0.2) -- (0.2, 0.45) node[above,xshift=5mm]{\scriptsize$t^2(\boldsymbol{q}^{1/2}+\boldsymbol{q}^{-1/2})x^{-1}y^2$};
\draw[black!40!green,line width=2pt] (-0.3, 0.2) -- (-0.1, 0.4) node[left]{\scriptsize$t^2x^{-2}y^2$};
\end{tikzpicture}
\caption{Broken lines defining $\vartheta_2$ on the central chamber. The right hand side shows the situation in the central chamber, including the ending monomials.}
\label{fig:theta2}
\end{figure}
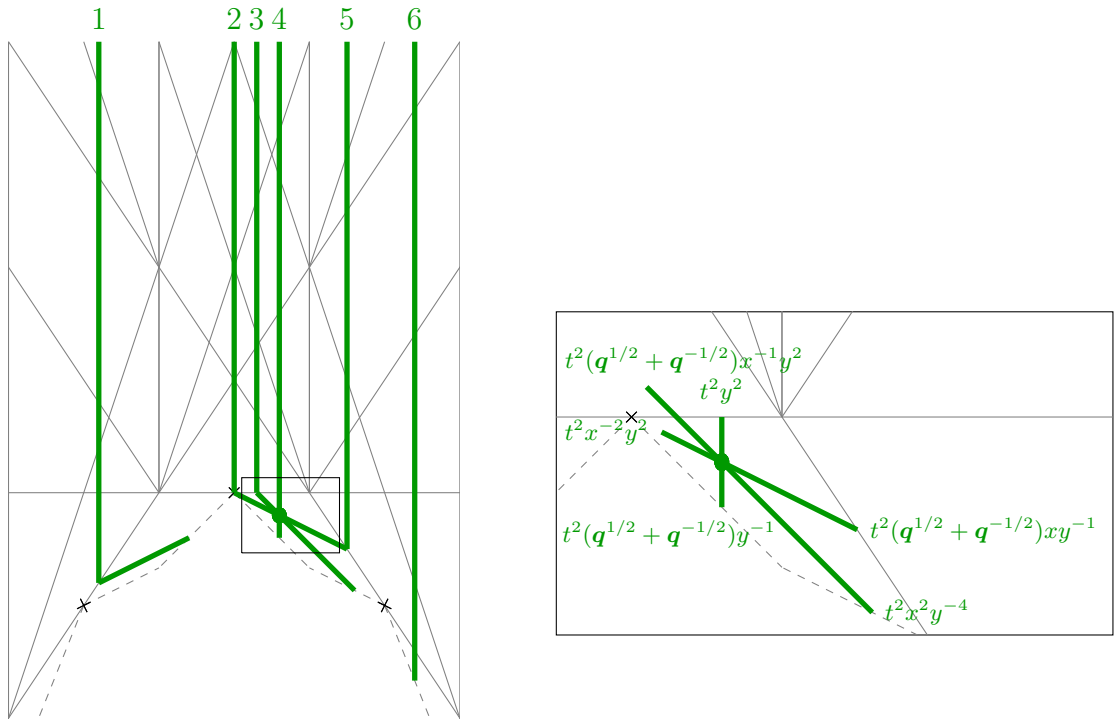

Consider the numbering of broken line as in Figure \ref{fig:theta2}. The broken lines $4$ and $6$ give trivial tropical disks whose only vertex is $V_\infty$. Hence, the $\boldsymbol{q}$-refined tropical multiplicity is trivial:
\[ m_h(\boldsymbol{q}) = 1. \]
The broken lines $1$, $3$ and $5$ give tropical disks with one vertex of multiplicity $m_V=2$ and one bounded leg of weight $w_L=1$. This gives a $\boldsymbol{q}$- refined multiplicity of
\[ m_h(\boldsymbol{q}) = \boldsymbol{q}^{1/2}+\boldsymbol{q}^{-1/2}. \]
The broken line $2$ gives two different tropical disks. One has $m_V=4$ and $w_L=2$, giving
\[ m_h(\boldsymbol{q}) = -\frac{1}{2}(\boldsymbol{q}^1+\boldsymbol{q}^{-1}). \]
The other one, obtained by leg splitting, has $m_V=2,2$ and $w_L=1,1$, giving
\[ m_h(\boldsymbol{q}) = \frac{1}{2}(\boldsymbol{q}^{1/2}+\boldsymbol{q}^{1/2})^2 = \frac{1}{2}\boldsymbol{q}^1+1+\frac{1}{2}\boldsymbol{q}^{-1}. \]
Here the factor $\frac{1}{2}$ comes from the nontrivial automorphism group. The two terms above sum up to simply $1$.

Putting everything together we obtain
\[ \vartheta_2(\boldsymbol{q}) =  t^2 \big(y^2+x^{-2}y^2+x^2y^{-4}+(\boldsymbol{q}^{1/2}+\boldsymbol{q}^{-1/2})(y^{-1}+xy^{-1}+x^{-1}y^2)\big). \]

\section{Verification of computations}
\label{app-verification}

In this section we verify the calculations of the genus $1$ log invariants $R_{1,3d-1}^1(\mathbb{P}^2,dL)$ as calculated tropically in Appendix \ref{app:brokenlines} by relating them to the known local invariants $N^1(\mathbb{F}_1,dL-C)$ via the higher genus log-local correspondence of \cite{BFGW}.

Gopakumar--Vafa integrality
(\cite{GV1,IP}) relates $N^g$ to integers $n^g$ through
contributions of lower classes and genera.
The class $\beta=dL-C$ is primitive, so there are no multiple cover contributions from other classes,
but there are contributions from curves of lower genus.
We have
\[ \sum_{g\geq 0} N^g(\mathbb{F}_1,dL-C)\hbar^{2g} = \sum_g n^g(\mathbb{F}_1,dL-C) \cdot \left(\frac{ih}{q^{1/2}-q^{-1/2}}\right)^{2g-2}. \]
This gives $N^0(\mathbb{F}_1,dL-C)=n^0(\mathbb{F}_1,dL-C)$ and 
\[ N^1(\mathbb{F}_1,dL-C)=n^1(\mathbb{F}_1,dL-C)+ \frac{1}{12} n^0(\mathbb{F}_1,dL-C). \]
In \cite{HKR}, Table B.8 shows the well-known $n^0(\mathbb{F}_1,L-C)=-2$, $n^0(\mathbb{F}_1,2L-C)=5$ and according to Table B.9 we have $n^1(\mathbb{F}_1,L-C)=n^1(\mathbb{F}_1,2L-C)=0$. So,
\[ N^1(\mathbb{F}_1,L-C)=-\frac{1}{6}, \quad N^1(\mathbb{F}_1,2L-C)=\frac{5}{12}. \]
By the higher genus log-local correspondence of \cite{BFGW} (see in particular the proof of Theorem 4.11) we have, for $d=1,2$,
\[ R_{3d-1}^1(\mathbb{F}_1,dL-C) = (-1)^{3d-1}(3d-1)\left(N^1(\mathbb{F}_1,dL-C)-\frac{1}{24}(3d-1)^2N^0(\mathbb{F}_1,dL-C)\right). \]
This gives
\[ R_2^1(\mathbb{F}_1,L-C) = -\frac{1}{3}, \quad R_5^1(\mathbb{F}_1,2L-C) = -\frac{575}{24} \]
By Corollary \ref{cor-two-and-one-point} we have
\[ R_{3d-1}^g(\mathbb{P}^2,dL) = R_{3d-1}^g(\mathbb{F}_1,dL-C). \]
This gives
\[ R_{1,2}^1(\mathbb{P}^2,L-C) = -\frac{1}{2}, \quad R_{1,5}^1(\mathbb{P}^2,2L-C) = -25, \]
which agrees with the log invariants calculated in Appendix \ref{app:brokenlines}.

\sloppy
\tiny{

}
\end{document}